\documentclass[11pt]{amsart}
\usepackage{amsmath,amssymb,amstext,amsgen,amsbsy,amsopn,amsfonts,bbm,graphicx,amsthm,cleveref,mathrsfs,mathtools}

\usepackage[left=3cm, right=3cm, top=3.5cm, bottom=3cm,footskip=1cm,headsep=1.5cm]{geometry}
\usepackage[all]{xy}
\usepackage[OT2,T1]{fontenc}
\usepackage{float}

\crefname{equation}{}{}
\numberwithin{equation}{section}

\newtheorem{theorem}[equation]{Theorem}
\newtheorem{lemma}[equation]{Lemma}
\newtheorem{cor}[equation]{Corollary}
\Crefname{cor}{Corollary}{Corollaries}

\newtheorem{proposition}[equation]{Proposition}

\theoremstyle{remark}
\newtheorem{remark}[equation]{Remark}
\theoremstyle{definition}
\newtheorem{example}[equation]{Example}
\theoremstyle{definition}
\newtheorem{construction}[equation]{Construction}
\theoremstyle{definition}
\newtheorem{defi}[equation]{Definition}
\theoremstyle{definition}
\newtheorem{notation}[equation]{Notation}
\theoremstyle{definition}

\theoremstyle{definition}
\newtheorem{assumption}[equation]{Assumption}

\DeclareSymbolFont{cyrletters}{OT2}{wncyr}{m}{n}
\DeclareMathSymbol{\Sha}{\mathalpha}{cyrletters}{"58}

\def\A{\ensuremath\mathcal{A}}

\def\P{\ensuremath P}

\address{Department of Mathematics, King's College London, Strand, London, WC2R 2LS.}
\email{adam.morgan@kcl.ac.uk}

\begin{document}

\title{Quadratic twists of abelian varieties and disparity in Selmer ranks}
\author{Adam Morgan}

\maketitle

\begin{abstract}
We study the parity of 2-Selmer ranks in the family of quadratic twists of a fixed principally polarised abelian variety over a number field. Specifically, we determine the proportion of twists having odd (resp. even) 2-Selmer rank. This generalises work of Klagsbrun--Mazur--Rubin for elliptic curves and Yu for Jacobians of hyperelliptic curves. Several differences in the statistics arise due to the possibility that the Shafarevich--Tate group (if finite) may have order twice a square. In particular, the statistics for parities of $2$-Selmer ranks and $2$-infinity Selmer ranks need no longer agree and we describe both.
\end{abstract}
  
\setcounter{tocdepth}{1}
\tableofcontents

\section{Introduction}

In this paper we study how various invariants of principally polarised abelian varieties behave under quadratic twist. 

Our first result determines the distribution of the parities of $2$-Selmer ranks in the quadratic twist family of an arbitrary principally polarised abelian variety. Specifically, for a number field $K$ and real number $X>0$ set 
\[\mathcal{C}(K,X)=\{\chi\in \textup{Hom}_{\textup{cnts}}(\textup{Gal}(\bar{K}/K),\{\pm 1\})~~:~~\textup{Norm}(\mathfrak{p})<X \textup{ for all primes }\mathfrak{p}\textup{ at which }\chi\textup{ ramifies}\}.\]

\begin{theorem} \label{theorem}
Let $A/K$ be a principally polarised abelian variety and let $\epsilon:\textup{Gal}(K(A[2])/K)\rightarrow \{\pm 1\}$ be the map $\sigma\mapsto (-1)^{\dim A[2]^\sigma}$. 
\begin{itemize}
\item[(i)]If $\epsilon$ is a homomorphism then for all sufficiently large $X$, \[\frac{|\{\chi \in \mathcal{C}(K,X)~:~\textup{dim}_{\mathbb{F}_2}\textup{Sel}^2(A^\chi/K)~~\textup{is even}\}|}{|\mathcal{C}(K,X)|}=\frac{1+(-1)^{\textup{dim}_{\mathbb{F}_2}\textup{Sel}^2(A/K)}\cdot\delta}{2}\] where $\delta$ is a finite product of explicit local terms $\delta_v$ (see the statement of \Cref{main selmer ranks explicit theorem} for their definition).
\item[(ii)]   If $\epsilon$ fails to be a homomorphism then for all sufficiently large $X$, \[\frac{|\{\chi \in \mathcal{C}(K,X)~:~\textup{dim}_{\mathbb{F}_2}\textup{Sel}^2(A^\chi/K)~~\textup{is even}\}|}{|\mathcal{C}(K,X)|}=1/2.\]
\end{itemize}
(Here for $\chi \in \textup{Hom}_{\textup{cnts}}(\textup{Gal}(\bar{K}/K),\{\pm 1\})$, $A^\chi/K$ denotes the quadratic twist of $A$ by $\chi$.)
\end{theorem}

\Cref{theorem} is known for elliptic curves by work of Klagsbrun--Mazur--Rubin \cite[Theorem A]{MR3043582} and, more generally, for Jacobians of odd degree hyperelliptic curves by work of Yu \cite[Theorem 1]{YU15}. These previous results both fall into Case (i) of \Cref{theorem}, thus the failure of $\epsilon$ to be a homomorphism forcing parity in the distribution is a phenomenon new to this work. Despite this, Case (ii) of  \Cref{theorem} is in some sense the `generic' case since if $\textup{Gal}(K(A[2])/K)$ is the full symplectic group $\textup{Sp}_{2g}(\mathbb{F}_2)$ for $g=\textup{dim}A\geq 3$ then the simplicity of $\textup{Sp}_{2g}(\mathbb{F}_2)$ prevents $\epsilon$ from being a homomorphism. For a discussion of when $\epsilon$ is or is not a homomorphism for various families of abelian varieties, see  \Cref{Main theorems for $2$-Selmer ranks}.

In the two previously known cases above, finiteness of the $2$-part of the Shafarevich--Tate group is known to imply that the parity of the $2$-Selmer rank agrees with that of the Mordell--Weil rank, so that \Cref{theorem} is conjecturally satisfied by Mordell--Weil ranks also. For general principally polarised abelian varieties however this need not be true due to a phenomenon first observed by Poonen--Stoll (see \cite{MR1740984})  that the $2$-part of the Shafarevich--Tate group, if finite, need not have square order. Thus to see how one expects the parity of Mordell--Weil ranks to behave in quadratic twist families  we also prove a version of \Cref{theorem} for $2^{\infty}$-Selmer ranks (by definition, the $2^{\infty}$-Selmer rank is equal to the sum of the Mordell--Weil rank and the (conjecturally trivial) number of copies of $\mathbb{Q}_2/\mathbb{Z}_2$ in the Shafarevich--Tate group). 

\begin{theorem} \label{2-infinity selmer rank main theorem}
Let  $A/K$ be a principally polarised abelian variety. Then for all sufficiently large $X>0$, 
\[\frac{|\{\chi \in \mathcal{C}(K,X)~:~\textup{rk}_2(A^\chi/K)~~\textup{is even}\}|}{|\mathcal{C}(K,X)|}=\frac{1+(-1)^{\textup{rk}_2(A/K)}\cdot\kappa}{2}\]
where $\kappa$ is an explicit finite product of local terms $\kappa_v$ given in \Cref{delta for 2-infinity}.
\end{theorem}

We remark that if $\textup{dim}A$ is odd and $K$ has a real place then $\kappa=0$. In general however, $\kappa$ is often non-zero: see \cite[Example 7.11]{MR3043582} for an example of an elliptic curve for which $\kappa$ is dense in $[-1,1]$ as the base field $K$ is varied, and \cite[Proposition 8.1]{YU15} for an example of an abelian surface over $\mathbb{Q}$ for which $\kappa=1$. 

Combining \Cref{theorem,2-infinity selmer rank main theorem} we see that the distribution of parities of $2$-Selmer ranks and $2^{\infty}$-Selmer ranks in general behave quite differently, as the following example illustrates.

\begin{example}[See \Cref{main example of parities}]
Let $J/\mathbb{Q}$ be the Jacobian of the genus $2$ hyperelliptic curve $C:y^2=x^6+x^4+x+3$. Then the function $\epsilon$ is not a homomorphism for $J/\mathbb{Q}$ so that half of the $2$-Selmer ranks of the quadratic twists of $J$ are even and half odd. On the other hand, $J$ has $\kappa=\frac{3}{16}$ and odd $2^\infty$-Selmer rank, so that $19/32$ of the twists of $J$ have even $2^\infty$-Selmer rank and $13/32$ odd. 
\end{example}

In fact, in the case where $\epsilon$ fails to be a homomorphism we show that the parity of the $2^\infty$-Selmer ranks behave in some sense independently of the parity of the $2$-Selmer ranks (see \Cref{independent variables remark}).  

A key step in passing between \Cref{theorem} and \Cref{2-infinity selmer rank main theorem} is the study of how the `non-square order Shafarevich--Tate group' phenomenon behaves under quadratic twist. Our main result here is: 

 \begin{theorem}  \label{local sha twist decomp}
Let $A/K$ be an abelian variety with principal polarisation $\lambda$ defined over $K$ and let $\chi\in\textup{Hom}_{\textup{cnts}}(\textup{Gal}(\bar{K}/K),\{\pm 1\})$ be a quadratic character corresponding to the quadratic extension $L/K$. 

Then  $\textup{dim}_{\mathbb{F}_2}\Sha_\textup{nd}(A/K)[2]+\textup{dim}_{\mathbb{F}_2}\Sha_\textup{nd}(A^\chi/K)[2]\equiv 0 ~~\textup{ (mod 2)}$ if and only if 
\[\sum_{v ~~\textup{non-split in }L/K} \textup{inv}_v~\mathfrak{g}(A/K_v,\lambda,\chi_v) =0\in \mathbb{Q}/\mathbb{Z}\]
where the local terms $\mathfrak{g}(A/K_v,\lambda,\chi_v)\in \textup{Br}(K_v)[2]$ are given in \Cref{the local terms part 2}.
(Here $\chi_v$ denotes the restriction of $\chi$ to the completion $K_v$ and $\Sha_{\textup{nd}}(A/K)$ denotes the quotient of the Shafarevich--Tate group of $A/K$ by its maximal divisible subgroup.) 
\end{theorem}

In particular, \Cref{local sha twist decomp} shows that the sum 
\[\textup{dim}_{\mathbb{F}_2}\Sha_\textup{nd}(A/K)[2]+\textup{dim}_{\mathbb{F}_2}\Sha_\textup{nd}(A^\chi/K)[2]~\phantom{hi}~\textup{(mod }2)\] is controlled by purely local behaviour. For $A/K$ the Jacobian of a curve it is a result of Poonen--Stoll  \cite[Corollary 12]{MR1740984} that this is in fact true for the parity of  $\textup{dim}_{\mathbb{F}_2}\Sha_\textup{nd}(A/K)[2]$ itself, but whether or not this holds for an arbitrary principally polarised abelian variety remains open. 


In general, the definition of the local terms $\mathfrak{g}(A/K_v,\lambda,\chi_v)$ appearing in \Cref{local sha twist decomp} is quite involved but if the principal polarisation $\lambda$ on $A/K_v$ is induced by a $K_v$-rational symmetric line bundle $\mathcal{L}_v$ then they take a simple form. Specifically, associated to $\mathcal{L}_v$ is a $\textup{Gal}(\bar{K}_v/K_v)$-invariant quadratic refinement $q$ of the Weil pairing on $A[2]$ (we review this classical construction in \Cref{quad refinements of pairing subsection}). As a consequence, $\textup{Gal}(\bar{K}_v/K_v)$ acts on $A[2]$ through the orthogonal group $O(q)$. In particular, we obtain a quadratic character $\psi_v$ of $K_v$ as the composition
\[\psi_v:\textup{Gal}(\bar{K}_v/K_v)\rightarrow O(q)/SO(q)\cong \{\pm 1\}.\]
We then have 
\[\mathfrak{g}(A/K_v,\lambda,\chi_v)=\chi_v \cup \psi_v \in \textup{Br}(K_v).\]
This allows the explicit evaluation of $\mathfrak{g}(A/K_v,\lambda,\chi_v)$ for archimedean places and for nonarchimedean places $v\nmid 2$ at which $A$ has good reduction (\Cref{local terms in sha computation}). The implications for arithmetic of the difference in characteristic $2$ between quadratic forms and symmetric bilinear pairings will be a recurring theme throughout this paper. 

\Cref{local sha twist decomp} may also be used to prove the analogue of \Cref{theorem} for the parity of the dimension of the $2$-torsion in the Shafarevich--Tate group amongst the family of quadratic twists of an arbitrary principally polarised abelian variety, quantifying the failure of the Shafarevich--Tate group to have square order amongst quadratic twists. See \Cref{sha statistics theorem}  for the precise statement.

 To expain the remaining results of the paper we briefly indicate how we prove \Cref{theorem}. As in  \cite{MR3043582} which proves the elliptic curve case, we deduce \Cref{theorem} from a more general theorem determining the distribution of parities of ranks of certain Selmer groups $\textup{Sel}(T,\chi)$ associated to a finite dimensional $\mathbb{F}_p$-vector space $T$ equipped with a $\textup{Gal}(\bar{K}/K)$-action, an alternating pairing, and abstract `twisting data'. Taking $T=A[2]$ along with the Weil pairing and the twisting data detailed in \Cref{twisting abelian varieties section} recovers \Cref{theorem}. The general result, the case $\textup{dim}T=2$ of which combines Theorems 7.6 and 8.2 of op. cit., is as follows. In the statement, the group $\mathcal{C}(K,X)$ for $p>2$ is defined in the identical way to $p=2$, replacing $\textup{Hom}_{\textup{cnts}}(\textup{Gal}(\bar{K}/K),\{\pm 1\})$ (the group of quadratic characters) with the group $\textup{Hom}_{\textup{cnts}}(\textup{Gal}(\bar{K}/K),\boldsymbol \mu_p)$ (of $p$-cyclic characters).
 
 \begin{theorem} \label{main twisting theorem combined intro}
Let $T$ be a finite dimensional $\mathbb{F}_p$-vector space equipped with a continuous $\textup{Gal}(\bar{K}/K)$-action and a non-degenerate $\textup{Gal}(\bar{K}/K)$-equivariant alternating pairing $T\times T\rightarrow \boldsymbol \mu_p.$
Suppose additionally that $T$ is equipped with a global metabolic structure $\textbf{q}$ and twisting data $\boldsymbol \alpha$ (Definitions \ref{global metabolic structure defi} and \ref{twisting data}). Let $\epsilon:\textup{Gal}(K(T)/K)\rightarrow \{\pm 1\}$ be the map $\sigma \mapsto (-1)^{\dim T^\sigma}$ (here $K(T)/K$ is the fixed field of the kernel of the $\textup{Gal}(\bar{K}/K)$-action on $T$).

\begin{itemize}
\item[(i)] If either $p=2$ and $\epsilon$ fails to be a homomorphism, or $p>2$ and $\epsilon$ is non-trivial when restricted to $\textup{Gal}(K(T)/K(\boldsymbol \mu_p))$, then  
\[\lim_{X\rightarrow \infty}\frac{|\{\chi \in \mathcal{C}(K,X)~:~\textup{dim}_{\mathbb{F}_p}\textup{Sel}(T,\chi)~~\textup{is even}\}|}{|\mathcal{C}(K,X)|}=1/2.\]
Moreover, if $p=2$ then it suffices to take $X$ sufficiently large as opposed to taking the limit $X\rightarrow \infty$. 
\item[(ii)] If either $p=2$ and $\epsilon$ is a homomorphism, or $p>2$ and $\epsilon$ is trivial when restricted to $\textup{Gal}(K(T)/K(\boldsymbol \mu_p))$, then for all sufficiently large $X$ we have  \[\frac{|\{\chi \in \mathcal{C}(K,X)~:~\textup{dim}_{\mathbb{F}_p}\textup{Sel}(T,\chi)~~\textup{is even}\}|}{|\mathcal{C}(K,X)|}=\frac{1+(-1)^{\textup{dim}_{\mathbb{F}_p}\textup{Sel}(T,\mathbbm{1})}\cdot\delta}{2}\] where $\delta=\prod_{v}\delta_v$ is an explicit finite product of local terms $\delta_v$ given in \Cref{definition of delta} and $\mathbbm{1}$ denotes the trivial character.
\end{itemize}
(For the definition of the Selmer groups $\textup{Sel}(T,\chi)$ associated to $\textbf{q}$ and $\boldsymbol \alpha$, see \Cref{twisted selmer groups defi}.)
\end{theorem}

When $p=2$ and $\textup{dim}T=2$ then $\epsilon$ is necessarily a homomorphism so that, as with \Cref{theorem}, the case where $\epsilon$ fails to be a homomorphism exhibits behaviour new to this work. On the other hand, when $\textup{dim}T=2$ yet $p>2$, $\epsilon$ is non-trivial when restricted to $\textup{Gal}(K(T)/K(\boldsymbol \mu_p))$ if and only if $p$ divides $[K(T):K]$ (see \cite[Lemma 4.3]{MR3043582}), so that now both cases (i) and (ii) of  \Cref{main twisting theorem combined intro} occur.  In particular, by allowing the dimension of $T$ to be arbitrary one obtains a more uniform picture between $p=2$ and $p>2$. 
For a discussion on conditions on the  $\textup{Gal}(\bar{K}/K)$-action on $T$ which result in Case (i) (resp. (ii)) of \Cref{main twisting theorem combined intro} see \Cref{remark on the function epsilon}. 

Taking $p>2$ and $T=A[p]$ for a principally polarised abelian variety $A/K$, along with the metabolic structure and twisting data detailed in \Cref{odd p twisting data}, enables us to prove an analogue of \Cref{theorem} which applies to Selmer groups of certain $p$-cyclic twists of $A^{p-1}$. See \Cref{p-cyclic abelian variety cor} for the precise statement (again, the case where $A$ is an elliptic curve is shown by Klagsbrun--Mazur--Rubin in \cite{MR3043582}).

 Finally, we remark that a key step in proving \Cref{main twisting theorem combined intro} is, for a character $\chi$, to describe the quantity
 \[\dim_{\mathbb{F}_p}\textup{Sel}(T,\chi)-\dim_{\mathbb{F}_p}\textup{Sel}(T,\mathbbm{1}) \phantom{hello} \textup{(mod }2)\]
as a sum of local terms (see \Cref{local decomposition}). Upon taking $T=A[2]$ for a principally polarised abelian variety $A/K$ one obtains (\Cref{difference of selmer groups for abelian varieties}) a local formula for the difference between the parity of the $2$-Selmer rank of $A/K$ and the $2$-Selmer rank of the quadratic twist $A^\chi/K$. This generalises a theorem of Kramer \cite[Theorem 1]{MR597871} for elliptic curves, and Yu \cite[Theorem 5.11]{YU15}  for Jacobians of odd degree hyperelliptic curves. Combining this with \Cref{local sha twist decomp}, one obtains (\Cref{difference of selmer groups for abelian varieties}) a purely local formula for the parity of the $2^\infty$-Selmer rank of $A$ over the quadratic extension cut out by $\chi$. Such local formulas for (the parity of) $2^\infty$-Selmer ranks  have applications to  the $2$-parity conjecture and these will be examined in a forthcoming paper.

\subsection*{Layout of the paper}

In \S2 we review some standard results in group cohomology that will be used in the seqeuel. In \S3 we review and study quadratic forms on finite dimensional $\mathbb{F}_2$-vector spaces. The main result is \Cref{extension of dickson invariant} which forms a key technical step in the proof of \Cref{local sha twist decomp}. Section 4 recalls the constructions of certain quadratic forms associated to abelian varieties and examines how these behave under quadratic twist. Of particular importance is \Cref{main twist of theta groups lemma} which plays a crucial role in associating twisting data to the group of $2$-torsion points of a principally polarised abelian variety. \Cref{local sha twist decomp} is proven in \S4. Sections 6-9 contain the proof of \Cref{main twisting theorem combined intro} and broadly follow the layout and strategy of \cite[Sections 3-4 and 6-8]{MR597871}. Specifically, in \S6 we recall the notions of metabolic structure and twisting data from op. cit. and generalise them to arbitrary (finite) dimensional $\mathbb{F}_p$-vector spaces, as well as defining the associated Selmer groups. \Cref{main twisting theorem combined intro} (i) is proven in \S7, whilst \S8 uses class field theory to produce certain global characters with specified local behaviour and is a more or less direct generalisation of \cite[Section 6]{MR597871}, albeit with different proofs. The results of \S8 are then applied in \S9 to prove \Cref{main twisting theorem combined intro} (ii). Section 10 associates a metabolic structure and twisting data to the $2$-torsion in a principally polarised abelian variety and deduces  \Cref{theorem,2-infinity selmer rank main theorem}. Finally, \S11 associates a metabolic structure and twisting data to the $p$-torsion in a principally polarised abelian variety for $p$ odd. 

\subsection*{Notation} 
For a group $G$ acting on an abelian group $M$, for $\sigma \in G$ we write
\[M^\sigma:=\{m\in M~~:~~\sigma(m)=m\}.\]
For a field $F$ we denote its separable closure by $\bar{F}$, its absolute Galois group by $G_F$ and, for $p$ different from the characteristic of $F$, we denote by $\boldsymbol \mu_p$ the $G_F$-module of $p$-th roots of unity in $\bar{F}$. We denote by $\textup{Br}(F)$ the Brauer group of $F$. 

For an abelian variety $A/F$ we write $A^{\vee}/F$ for the dual of $A$. A \emph{principally polarised abelian variety} over $F$ is a pair $(A/F,\lambda)$ consisting of an abelian variety $A/F$ and a principal polarisation  $\lambda:A\rightarrow A^{\vee}$ defined over $F$. We often omit $\lambda$ from the notation.  For a quadratic character $\chi\in \textup{Hom}_{\textup{cnts}}(G_F,\{\pm 1\})$ the \emph{quadratic twist} of $A$ by $\chi$ is the pair $(A^\chi,\psi)$ consisting of an abelian variety $A^\chi/F$ and an isomorphism $\psi:A\rightarrow A^\chi$ such that $\psi^{-1}\psi^\sigma=[\chi(\sigma)]$ for all $\sigma \in G_F$. We often denote this by $A^\chi$, the map $\psi$ being understood. 

 For a number field $K$ we denote by $M_K$  the set of places of $K$ and write $K_v$ for the completion of $K$ at $v\in M_K$.  We denote by $\textup{inv}_v:\textup{Br}(K_v)\rightarrow \mathbb{Q}/\mathbb{Z}$ the local invariant map and, if $v$ is nonarchimedean, denote by $K_v^{\textup{ur}}$ the maximal unramified extension of $K_v$. We implicitly fix embeddings $\bar{K}\hookrightarrow \bar{K}_v$ for each $v\in M_K$ and thus view $G_{K_v}$ as a subgroup of $G_K$ for each $v$. In particular, for a (finite) Galois extension $L/K$ of number fields and a non-archimedean place $v\in M_K$ unramified in $L/K$ we have a well defined Frobenius element $\textup{Frob}_v$ in $\textup{Gal}(L/K)$. 
 
 For a $G_K$-module $M$, the injections $G_{K_v}\hookrightarrow G_K$ induce restriction maps on cohomology $H^i(K,M)\rightarrow H^i(K_v,M)$ for each $i\geq 0$ and $v\in M_K$. For a cocycle $\xi$ we write $\xi_v$ for its restriction to $K_v$ (see  \Cref{cohomology section} for our notation and conventions concerning group cohomology). We define, for $v$ a non-archimedean place of $K$,
\[H^i_{\textup{ur}}(K_v,M):=\ker\left(H^i(K_v,M)\stackrel{\textup{res}}{\longrightarrow}H^i(K_v^{\textup{ur}},M)\right).\]

%

\subsection{Acknowledgements}
We thank K\k{e}stutis \v{C}esnavi\v{c}ius for many useful conversations and comments, and for correspondence regarding the material in Section 4.  
We thank Tim Dokchitser, Vladimir Dokchitser, C\'{e}line Maistret and Jeremy Rickard for helpful conversations.

\section{Group cohomology and group extensions} \label{cohomology section}

In the following sections we'll make several computations involving group cohomology. Here we set up the relevant notation and review some basic results. All material in this section is standard: see e.g. \cite{MR0219512}. 

\subsection{Group cohomology}
Let $G$ be a finite group and $M$ a $G$-module. For $i\geq 0$ we write $C^i(G,M)$ for the group of $i$-cochains with values in $M$ and $d:C^i(G,M)\rightarrow C^{i+1}(G,M)$ for the usual differential. When $i=0$ we have $(dm)(g)=gm-m$ for $m\in M=C^0(G,M)$ and $g\in G$, and when $i=1$ we have \[(df)(g,h)=f(g)+gf(h)-f(gh)\] for $f\in C^1(G,M)$ and $g,h\in G$. We write $Z^i(G,M)$ (resp. $B^i(G,M)$) for the group of $i$-cocycles (resp. $i$-coboundaries) with values in $M$. We'll always think of the $i$'th-cohomology group $H^i(G,M)$ as the quotient $Z^i(G,M)/B^i(G,M)$.When making computations involving group cohomology, we'll make the convention that Fraktur letters such as $\mathfrak{a},\mathfrak{b}$ etc. denote cohomology classes and that the corresponding lower case Roman letters $a,b$ etc., denote cocycles representing these cohomology classes. More generally, if $G$ is a profinite group we consider continuous cochains, cocycles and coboundaries, using the same notation and conventions to talk about them. 

\subsection{Cup product on cochains} \label{cup product of cochains} Let $G$ be a finite (or profinite) group and let $M$ and $N$ be $G$-modules. Then for $i,j\geq 0$ the \emph{cup-product} map
\[\cup:C^i(G,M)\times C^j(G,N)\longrightarrow C^{i+j}(G,M\otimes N)\]
is defined by
\[(a\cup b)(g_1,...,g_{i+j})=a(g_1,...,g_i)\otimes g_1...g_ib(g_{i+1},...,g_{i+j}).\]
For $a\in C^i(G,M)$ and $b\in C^j(G,N)$ we have the equality 
\begin{equation} \label{cup product on cochains}
d(a\cup b)=da\cup b+(-1)^i a\cup db
\end{equation}
inside $C^{i+j+1}(G,M\otimes N)$. 

For $i,j\geq 0$ the cup product map above induces a cup product map on cohomology
\[\cup:H^i(G,M)\times H^j(G,N)\rightarrow H^{i+j}(G,M\otimes N)\]
which satisfies $\mathfrak{a}\cup \mathfrak{b}=(-1)^{ij}\mathfrak{b}\cup \mathfrak{a}$.

\subsection{Group extensions} \label{group extensions subsection}

Let $G$ be a finite group and $M$ an abelian group with trivial $G$-action. In what follows we write the group law on $G$ multiplicatively and the group law on $M$ additively. Let $\mathfrak{a}\in H^2(G,M)$ and $a$ be a $2$-cocycle representing $\mathfrak{a}$. Define a group structure on the set $G\times M$ by the rule
\[(g,m)\cdot (g',m')=(gg',m+m'+a(g,g'))\]
and let $E_a$ denote the resulting group. 
The maps $\alpha:M\rightarrow E_a$ and $\beta:E_a\rightarrow G$ defined by  $m\mapsto (1,m-a(1,1))$ and $(g,m)\mapsto g$ respectively give rise to the short exact sequence
\[0\rightarrow M\stackrel{\alpha}{\longrightarrow}E_a \stackrel{\beta}{\longrightarrow}G\rightarrow 0\]
realising $E_a$ as a central extension of $G$ by $M$. The isomorphism class of this extension is independent of the choice of cocycle representing $\mathfrak{a}$ and the sequence splits if and only if $\mathfrak{a}$ is the trivial class in $H^2(G,M)$. More specifically, let $s:G\rightarrow E_a$ denote the set section $g\mapsto (g,0)$ to $\beta$. Then if $\phi:E_a\rightarrow M$ is a homomorphism splitting the exact sequence (i.e. giving a section to $\alpha$) then the function $f=\phi\circ s\in C^1(G,M)$ is a 1-cochain satisfying $df=a$. 

\begin{remark}
The above correspondence in fact gives rise to a bijection between elements of $H^2(G,M)$ and the set of isomorphism classes of central extensions of $G$ by $M$, and one can generalise this correspondence to include the case where the action of $G$ on $M$ is non-trivial (though now the relevant extensions are, in general, no longer central). See \cite[Section 2]{MR0219512} for more details. 
 \end{remark} 
 
\section{Quadratic forms on finite dimensional $\mathbb{F}_2$-vector spaces} \label{f_2 vector spaces section}

The aim of this section is to prove \Cref{extension of dickson invariant,change of form} which are needed for the proof of \Cref{local sha twist decomp}. In Sections 3.1, 3.2 and 3.3 we review the theory of quadratic forms on finite dimensional $\mathbb{F}_2$-vector spaces. The material in 3.1 and 3.2 is standard, see e.g. \cite[Section 9.4]{MR770063}. In 3.3 we review a construction due to Pollatsek (given in the discussion preceeding \cite[Theorem 1.11]{MR0280596}) which we use in the proof of \Cref{extension of dickson invariant}.

For the rest of this section fix a finite dimensional $\mathbb{F}_2$-vector space $V$ equipped with a non-degenerate alternating pairing 
\[\left \langle ~,~\right\rangle:V\times V\longrightarrow \mathbb{F}_2\]
(so in particular, $\text{dim}V$ is necessarily even). We denote by $\text{Sp}(V)$ the symplectic group  of linear automorphisms of $V$ preserving the pairing. 
 
\subsection{Quadratic refinements and the class $\mathfrak{c}\in H^1(\mathrm{Sp}(V),V)$} \label{quadratic refinements subsection}

\begin{defi}[Quadratic refinement]
A function $q:V\rightarrow \mathbb{F}_2$ is called a \emph{quadratic refinement} of $\left \langle~,~\right\rangle$ if we have
\[q(v+v')+q(v)+q(v')=\left \langle v,v'\right \rangle\]
for all $v,v'\in V$.
\end{defi}

Let $\mathcal{Q}$ denote the set of all quadratic refinements of $\left \langle~,~\right \rangle$. It is a principal homogeneous space for $V$ where, for $v\in V$, we define $q+v\in \mathcal{Q}$ by setting
\[(q+v)(v')=q(v')+\left \langle v,v'\right \rangle\] 
for $v'\in V$.
The symplectic group $\textup{Sp}(V)$ acts on the set of quadratic refinements via $q\mapsto q\circ \sigma ^{-1}$ (for $\sigma \in \textup{Sp}(V)$). This action is compatible with addition by elements of $V$ and so associated to $\mathcal{Q}$ is a class
\[\mathfrak{c}\in H^1(\textup{Sp}(V),V).\]
Explicitly, picking a quadratic refinement $q$ and defining $\lambda :V\rightarrow V^*:=\text{Hom}(V,\mathbb{F}_2)$ to be the map $v\mapsto \left \langle v,- \right\rangle$, the function $c_q:\textup{Sp}(V)\rightarrow V$ given by setting
\[c_q(\sigma)= \lambda^{-1}(q\circ \sigma ^{-1}-q)\] is a $1$-cocycle representing $\mathfrak{c}$.   

\begin{remark} \label{exact sequence quadratic refinement remark}
Let $\mathcal{A}\textit{lt}$ denote the group of (possibly degenerate) alternating pairings on $V$ under addition. It has an action of $\textup{Sp}(V)$ given by $\sigma \cdot \left \langle \left \langle ~,~\right \rangle \right \rangle=\left \langle \left \langle \sigma^{-1}(~),\sigma^{-1}(~)\right \rangle \right \rangle$. Similarly, let $\mathcal{Q}\textit{uad}$ denote the group of quadratic forms on $V$ under addition which also carries an action of $\textup{Sp}(V)$ via $\sigma \cdot q=q\circ \sigma^{-1}$. Then we have a short exact sequence of $\textup{Sp}(V)$-modules
\begin{equation} \label{construction of torsor of quadratic refinements}
0\rightarrow V^* \longrightarrow \mathcal{Q}\textit{uad} \longrightarrow \mathcal{A}\textit{lt} \longrightarrow 0
\end{equation}
where the map $V^*\rightarrow  \mathcal{Q}\textit{uad}$ is inclusion and the map $\mathcal{Q}\textit{uad} \rightarrow \mathcal{A}\textit{lt}$ sends a quadratic form to its associated pairing. The associated long exact sequence for cohomology gives a map
\[\delta:H^0(\textup{Sp}(V),\mathcal{A}\textit{lt})\rightarrow H^1(\textup{Sp}(V),V^*).\] Our pairing $\left \langle ~,~\right \rangle$ is an element of $H^0(\textup{Sp}(V),\mathcal{A}\textit{lt})$ and the class $\mathfrak{c}\in H^1(\textup{Sp}(V),V)$ constructed above is the image of $\left \langle ~,~\right \rangle$ under $\delta$, once we use the map $\lambda$ above to identify $H^1(\textup{Sp}(V),V)$ with $H^1(\textup{Sp}(V),V^*)$. 
\end{remark}

\begin{remark}
It is shown in \cite[Theorems 4.1 and 4.4]{MR0280596} that if $\textup{dim}(V)\geq 4$ then $H^1(\textup{Sp}(V),V)\cong \mathbb{Z}/2\mathbb{Z}$, generated by  $\mathfrak{c}$.
\end{remark}

\subsection{Orthogonal groups, Special orthogonal groups and the Dickson homomorphism}

For a given quadratic refinement $q$, denote by $O(q)$ the corresponding orthogonal group of linear automorphisms preserving $q$ rather than just the pairing. The orthogonal group $O(q)$ has an index $2$ subgroup $SO(q)$ which is by definition the kernel of the Dickson homomorphism, whose definition we now recall. Let $C(q)$ denote the Clifford algebra associated to $q$ (see \cite[Definition 9.2.1]{MR770063}), $C^0(q)$  its even graded sub-algebra  and $Z(q)$ the centre of $C^0(q)$. Then  $Z(q)$ is a rank $2$ \'{e}tale algebra over $\mathbb{F}_2$ (see Theorem 9.4.8 of op. cit.). Since $O(q)$ acts naturally on $C(q)$ and preserves the grading, it acts on $Z(q)$ by $\mathbb{F}_2$-algebra homomorphisms. Noting that the automorphism group of any rank $2$  \'{e}tale algebra over $\mathbb{F}_2$ (or indeed any field) is canonically isomorphic to $\mathbb{Z}/2\mathbb{Z}$, we obtain a homomorphism $d_q:O(q)\rightarrow \mathbb{Z}/2\mathbb{Z}$, the \emph{Dickson homomorphism}. 

We will also need the following alternative characterisation of the Dickson homomorphism. 

\begin{proposition} \label{dimension proposition}
Let $q$ be a quadratic refinement of $\left\langle~,~ \right\rangle$ and $\sigma \in O(q)$. Then 
\[d_q(\sigma)= \textup{dim}V^{\sigma} \textup{ (mod 2)}.\]
\end{proposition}

\begin{proof}
This is \cite[Theorem 3]{MR0453639}. 
\end{proof}

\subsection{An extension of the Dickson homomorphism to the full symplectic group}

The following is a version of a construction due to Pollatsek (\cite{MR0280596}) which gives an extension of the Dickson homomorpism to the whole of $\textup{Sp}(V)$. We caution however that the resulting function $\textup{Sp}(V)\rightarrow \mathbb{Z}/2\mathbb{Z}$ is not a homomorphism (we cannot ask for this since for $\textup{dim}V\geq 6$, $\textup{Sp}(V)$ is simple).

\begin{construction}[Pollatsek]  \label{construction of phi_q}
Fix a quadratic refinement $q$ of $\left \langle ~,~\right \rangle$. 
Set $U=\mathbb{F}_2^2$ equipped with its unique non-degenerate alternating form $\left \langle ~,~\right \rangle_U$. Further, let $q_U$ denote the unique quadratic refinement of $\left \langle ~,~\right \rangle_U$ with Arf invariant 1. Thus for $(\lambda,\lambda')\in U$ we have
\[q_U((\lambda,\lambda'))=\lambda +\lambda'+\lambda \lambda'.\]
Let $x=(1,0)$ and $y=(0,1)$ so that $q_U(x)=1=q_U(y)$ and $\left \langle x,y\right \rangle_U=1$. Now let $W:=V\oplus U$ be the orthogonal direct sum of $V$ and $U$, so that $W$ comes equipped with the quadratic form $q_W:=q+q_U$, whose associated (non-degenerate, alternating) pairing is $\left \langle ~,~\right \rangle_W:=\left \langle ~,~\right \rangle +\left \langle ~,~\right \rangle _U$. 


Now given $g=(\sigma,\alpha)\in \textup{Sp}(V)\times \mathbb{F}_2$, define the linear automorphism $\phi_q(g)$ of $W$ by setting
\[\phi_q(g)(x)=x,\]
\[\phi_q(g)(y)=\alpha x+c_q(\sigma)+y,\]
and for $v\in V$,
\[\phi_q(g)(v)=\sigma(v)+\left \langle  c_q(\sigma),\sigma(v) \right \rangle x\]
and extending linearly.
\end{construction}

A key property of this construction, as shown in the discussion preceeding \cite[Theorem 1.11]{MR0280596}, is that for each $g\in \textup{Sp}(V)\times \mathbb{F}_2$ we have $\phi_q(g)\in O(q_W)$. Moreover, Pollatsek shows in loc. cit. that for each $\sigma\in \textup{Sp}(V)$, there is a unique $\alpha(\sigma)\in \mathbb{F}_2$ such that $\phi_q((\sigma,\alpha(\sigma)))\in SO(q_W)$. One has $\alpha(\sigma)=d_q(\sigma)$ for all $\sigma \in O(q)$, so the map $\sigma \mapsto \alpha(\sigma)$ gives an extension of the Dickson homomorphism to the full symplectic group $\textup{Sp}(V)$.

\subsection{Triviality of $\mathfrak{c}\cup \mathfrak{c}$}

The pairing $\left\langle ~,~\right\rangle$ induces a cup-product map
\[\cup:H^1(\textup{Sp}(V),V)\times H^1(\textup{Sp}(V),V) \longrightarrow H^2(\textup{Sp}(V),\mathbb{F}_2).\] We now use the construction of the previous subsection to analyse the element $\mathfrak{c}\cup \mathfrak{c}\in H^2(\textup{Sp}(V),\mathbb{F}_2)$.

\begin{notation}
Given a quadratic refinement $q\in \mathcal{Q}$, let $E_q$ denote the central extension of $\textup{Sp}(V)$ by $\mathbb{F}_2$ corresponding to the $2$-cocycle $c_q\cup c_q$, so that as a set $E_q=\textup{Sp}(V)\times \mathbb{F}_2$, and is equipped with the group structure
\[(\sigma,\alpha)\cdot (\sigma',\alpha')=(\sigma \sigma',\alpha+\alpha'+(c_q\cup c_q)(\sigma,\sigma')).\]
\end{notation}

We then have:

\begin{lemma} \label{construction of hom}  
The function $\phi_q$ of \Cref{construction of phi_q} is a homomorphism $E_q\rightarrow O(q_W)$. 
\end{lemma}

\begin{proof}
As above, $\phi_q$ gives a map from $E_q$ into $O(q_W)$. An easy computation shows additionally that it is a homomorphism. 
\end{proof}


We may now prove the main result of the section. 

\begin{proposition} \label{extension of dickson invariant}
For each quadratic refinement $q\in \mathcal{Q}$, there is a unique function
$f_q:Sp(V)\rightarrow \mathbb{F}_2$
such that $df_q=c_q\cup c_q \in Z^2(Sp(V),\mathbb{F}_2)$ and such that the restriction of $f_q$ to the orthogonal group $O(q)$ is the Dickson homomorphism. In particular, we have
\[\mathfrak{c}\cup \mathfrak{c}=0\in H^2(Sp(V),\mathbb{F}_2).\]  
\end{proposition}

\begin{proof}
We first show uniqueness. If $f_q'$ is another function with $df_q'=c_q\cup c_q$ then the difference $f_q-f_q'$ is a homomorphism from $Sp(V)$ to $\mathbb{F}_2$. If $\text{dim}V\geq 6$ then $Sp(V)$ is simple and hence $f_q=f_q'$. If $\text{dim}V$ (which is necessarily even) is $2$ or $4$, then $Sp(V)$ has a unique index $2$ subgroup and hence a unique non-trivial homomorphism to $\mathbb{F}_2$. In each case this homomorphism is non-trivial when restricted to $O(q)$ for each quadratic refinement $q$, whence the result.   

In the notation of \Cref{construction of phi_q}, associated to $q_W$ is the Dickson homomorphism  \[d_{q_W}:O(q_W)\longrightarrow \mathbb{F}_2.\]
We claim that $d_{q_W}\circ \phi_q:E_q\rightarrow \mathbb{F}_2$ gives a section to the map $\mathbb{F}_2\rightarrow E_q$ sending $\alpha$ to $(1,\alpha)$, thus splitting the extension $E_q$. Indeed, let $\alpha \in \mathbb{F}_2$. Then $\phi_q((1,\alpha))=\text{id}_V\oplus m_\alpha$ where $m_\alpha\in O(q_U)$ is defined by $m_\alpha(x)=x$, $m_\alpha(y)=\alpha x+y.$ One sees (either using the definition in terms of  Clifford algebras, or by applying \Cref{dimension proposition}) that $\text{id}_V\oplus m_\alpha$ is in $SO(q_W)$ if and only if $\alpha=0$, whence $d_{q_W}((1,\alpha))=\alpha$ as desired. 

It now follows that the function $f_q:Sp(V)\rightarrow \mathbb{F}_2$ defined by $f_q(\sigma)=(d_{q_W}\circ \phi_q)((\sigma,0))$ satisfies $df_q=c_q\cup c_q$ (see \Cref{group extensions subsection} and note that $(c_q\cup c_q)(1,1)=0$). 

It remains to show that the restriction of $f_q$ to $O(q)$ is the Dickson homomorphism $d_q$. To see this, note that for any $\sigma\in O(q)$ we have $c_q(\sigma)=0$ and so  
\[\phi_q((\sigma,0))=\sigma \oplus \text{id}_U.\]
Since this is in $SO(q_W)$ if and only if $\sigma$ is in $SO(q)$ (again by looking at Clifford algebras or using \Cref{dimension proposition}), we have the claim. 
\end{proof}

We now describe how $f_q$ changes upon changing the quadratic refinement $q$.

\begin{proposition} \label{change of form}
Let $q$ and $q'$ be two quadratic refinements of $\left\langle~,~ \right\rangle$ and let $v\in V$ be such that $q'=q+v$, so that $c_{q'}=c_q+dv$. Then we have an equality of cochains in $C^1(\textup{Sp}(V),\mathbb{F}_2)$:
\[f_{q'}=f_q+c_q\cup v+v\cup c_q+v\cup dv.\]
\end{proposition} 

\begin{proof}
One readily computes
\[d(f_q+c_q\cup v+v\cup c_q+v\cup dv)=c_{q'}\cup c_{q'},\]
so it remains to show that the restriction of $f_q+c_q\cup v+v\cup c_q+v\cup dv$ to $O(q')$ is the Dickson homomorphism $d_{q'}$. To do this we'll use the characterisation of the Dickson homomorphism given in \Cref{dimension proposition}.

Fix $\sigma \in O(q')$. Then $c_{q}(\sigma)=(dv)(\sigma)$. In the notation of \Cref{construction of phi_q}, 
given $w\in W$, writing $w=z+\epsilon_1x+\epsilon_2 y$ with $z\in V$ and $\epsilon_1,\epsilon_2\in\mathbb{F}_2$,  one sees that $w$ is fixed by $\phi_q((\sigma,0))$ if and only if 
\begin{equation} \label{first eq}
\sigma(z)-z=\epsilon_2 (dv)(\sigma)
\end{equation}
and
\begin{equation} \label{second eq}
\left\langle (dv)(\sigma),\sigma(z)\right\rangle =0.
\end{equation}
Now \Cref{first eq} is equivalent to $z=z'+\epsilon_2 v$ for some $z'\in V^{\sigma}$. If $z$ has this form, then using invariance of $z'$ under $\sigma$ one computes
\[\left\langle (dv)(\sigma),\sigma(z)\right\rangle = \epsilon_2 \left\langle \sigma(v),v\right\rangle.\]
Thus if $\left\langle\sigma(v),v\right\rangle=0$ then the second condition \cref{second eq} is redundant, whilst if $\left \langle\sigma(v),v\right\rangle=1$ then it may be replaced with the condition $\epsilon_2=0$. We conclude that
\[\text{dim}W^{\phi_q((\sigma,0))}\equiv \text{dim}V^{\sigma}+\left\langle \sigma(v),v\right \rangle~~\text{ (mod 2)}\]
and hence (using \Cref{dimension proposition})
\[f_q(\sigma)=d_{q'}(\sigma)+\left\langle \sigma(v),v\right\rangle=d_{q'}(\sigma)+(v\cup dv)(\sigma).\]
Thus the restriction of $f_q$ to $O(q')$ is equal to $d_{q'}+v\cup dv$. Noting also that the restriction of $c_q$ to $O(q')$ is equal to $dv$, the result follows easily.
\end{proof}


\begin{remark}
Let $\tilde{V}$ denote the group whose underlying set is $V\times \mathbb{F}_2$, endowed with the group law
\[(v,\alpha)\cdot (v',\alpha')=(v+v',\alpha+\alpha'+\left \langle v,v'\right \rangle).\]
Then $\tilde{V}$ sits in a short exact sequence
\begin{equation} \label{sequence cup product}
0\rightarrow \mathbb{F}_2\longrightarrow \tilde{V}\longrightarrow V \rightarrow 0,
\end{equation}
the map $\mathbb{F}_2\rightarrow \tilde{V}$ sending $\alpha$ to $(0,\alpha)$ and the map $\tilde{V}\rightarrow V$ being projection onto the first factor. Making $\textup{Sp}(V)$ act trivially on $\mathbb{F}_2$ and diagonally on $\tilde{V}$ this sequence becomes an exact sequence of $\textup{Sp}(V)$-modules. Using the relation $df_q=c_q\cup c_q$, one can show that for each quadratic refinement $q$, the function $\tilde{c}_q:\textup{Sp}(V)\rightarrow \tilde{V}$ defined by
\[\tilde{c}_q(\sigma)=(c_q(\sigma),f_q(\sigma))\]
is a $1$-cocycle. One may then use the relationship between $f_q$ and $f_q'$ given in \Cref{change of form} to show that the class $\tilde{\mathfrak{c}}$  of $\tilde{c}_q$ in $H^1(\textup{Sp}(V),\tilde{V})$ does not depend on $q$ so that the results of this section prove that $\mathfrak{c}\in H^1(\textup{Sp}(V),V)$ admits a canonical lift to $H^1(\textup{Sp}(V),\tilde{V})$. (It is shown in \cite[Corollary 2.8(b)]{MR2915483} that the connecting homomorphism $H^1(\textup{Sp}(V),V)\rightarrow H^2(\textup{Sp}(V),\mathbb{F}_2)$ arising from \cref{sequence cup product} sends $\mathfrak{a}\in H^1(\textup{Sp}(V),V)$ to $\mathfrak{a}\cup \mathfrak{a}$ so that the triviality of $\mathfrak{c}\cup \mathfrak{c}$ is equivalent to the existence of some lift of $\mathfrak{c}$ to  $H^1(\textup{Sp}(V),\tilde{V})$.) 
\end{remark}


\section{Quadratic forms associated to abelian varieties} \label{Quadratic forms associated to abelian varieties}

In this section we study the behaviour under quadratic twist of certain quadratic forms associated to abelian varieties. Though several results in this section will be used in what follows, the most important is \Cref{main twist of theta groups lemma} which provides the technical input required to generalise \cite[Theorem 5.10]{YU15} to the case of arbitrary principally polarised abelian varieties (this is done in \Cref{is a lagrangian subspace}). Sections 4.1-4.3 review some standard results in the theory of abelian varieties as can be found, for example, in \cite{MUMFORD1966}.

For the rest of this section, fix a field $F$ of characteristic $0$ (which for applications will be either a number field or the completion of one). Let $A/F$ be an abelian variety. For $x\in A(\bar{F})$ denote by $\tau_x$ the translation-by-$x$ map $\tau_x:A\rightarrow A$. 

\subsection{Line bundles and self-dual homomorphisms}

Let $\mathcal{L}$ be a (not necessarily $F$-rational) line bundle on $A$. We denote by $\phi_{\mathcal{L}}$ the homomorphism $A\rightarrow A^{\vee}$ sending $x\in A(\bar{F})$ to the element of $A^{\vee}(\bar{F})$ corresponding to the line bundle $\tau_{x}^*\mathcal{L}\otimes \mathcal{L}^{-1}$. We write $K(\mathcal{L})$ for the kernel of $\phi_{\mathcal{L}}$. If $\mathcal{L}$ is ample then $K(\mathcal{L})$ is a finite subgroup of $A$.

 We have a short exact sequence of $G_F$-modules
 \begin{equation} \label{neron severi sequence}
 0\longrightarrow A^{\vee}(\bar{F})\longrightarrow \textup{Pic}A_{\bar{F}}\longrightarrow \textup{Hom}_{\textup{self-dual}}(A_{\bar{F}},A^{\vee}_{\bar{F}})\longrightarrow 0,
 \end{equation}
the map $A^{\vee}(\bar{F})\rightarrow \textup{Pic}A_{\bar{F}}$ being the natural inclusion and the map $\textup{Pic}A_{\bar{F}}\rightarrow \textup{Hom}_{\textup{self-dual}}(A_{\bar{F}},A^{\vee}_{\bar{F}})$ sending a line bundle $\mathcal{L}$ to $\phi_\mathcal{L}$. 
As in \cite[Section 3.2]{MR2915483},  \cref{neron severi sequence} induces a short exact sequence of $G_F$-modules
 \begin{equation} \label{neron severi refinement}
 0\longrightarrow A^{\vee}[2]\longrightarrow \textup{Pic}^{\textup{sym}}A_{\bar{F}}\longrightarrow \textup{Hom}_{\textup{self-dual}}(A_{\bar{F}},A^{\vee}_{\bar{F}})\longrightarrow 0,
 \end{equation}
 where here $\textup{Pic}^{\textup{sym}}A_{\bar{F}}$ denotes the group of symmetric line bundles on $A$ (i.e. those satisfying $[-1]^*\mathcal{L}\cong \mathcal{L}$).

 \subsection{Quadratic refinements of the Weil pairing on $A[2]$} \label{quad refinements of pairing subsection}

 Let $\left(~,~\right)_{e_2}:A[2]\times A^{\vee}[2]\longrightarrow \boldsymbol \mu_2$ denote the Weil pairing. It's bilinear, non-degenerate and $G_F$-equivariant. If $\lambda:A\rightarrow A^{\vee}$ is a self-dual homomorphism  then it induces an alternating pairing \[\left(~,~\right)_\lambda:A[2]\times A[2]\longrightarrow \boldsymbol \mu_2\]
 defined by $\left(a,b\right)_\lambda=\left(a,\lambda(b)\right)_{e_2}$ for $a,b\in A[2]$. If $\lambda$ is defined over $F$ then  $(~,~)_\lambda$ is $G_F$-invariant. In general, for a line bundle $\mathcal{L}$ on $A$ set $\left(~,~\right)_\mathcal{L}:=\left(~,~\right)_{\phi_\mathcal{L}}$. 

%

\begin{defi} \label{defi of quadratic refinement cor line bundle}
Let $\mathcal{L}$ be a symmetric line bundle on $A$. Define the map $q_\mathcal{L}:A[2]\rightarrow \boldsymbol \mu_2$ as follows. Given $x\in A[2]$, we have $x^*[-1]^*\mathcal{L}=x^*\mathcal{L}$. In particular, the restriction of the normalised\footnote{Writing $e\in A(\bar{F})$ for the identity section, an isomorphism $\tau:\mathcal{L}\stackrel{\sim}{\longrightarrow} [-1]^*\mathcal{L}$ is called \emph{normalised} if \[e^*(\tau):e^*\mathcal{L}\stackrel{\sim}{\longrightarrow}e^*[-1]^*\mathcal{L}=e^*\mathcal{L}\] is the identity. There is a unique such $\tau$ for each symmetric line bundle (see \cite[Section 2]{MUMFORD1966}).} isomorphism $\tau:\mathcal{L}\stackrel{\sim}{\longrightarrow}[-1]^*\mathcal{L}$ to $x$ is multiplication by an element $\eta_x\in \bar{F}^{\times}$ on $x^*\mathcal{L}$. One in fact has $\eta_x\in \boldsymbol \mu_2$ and we set $q_\mathcal{L}(x):=\eta_x$.
\end{defi}

\begin{remark}
The map $q_\mathcal{L}$ defined above is denoted $e_*^\mathcal{L}$ in \cite{MUMFORD1966} (fourth definition in Section 2).
\end{remark}

The following well known lemma summarises the properties of ${q}_\mathcal{L}$. 

\begin{lemma} \label{properties of quadratic forms associated to line bundles}
Let $\mathcal{L}$ be a symmetric line bundle on $A$. Then we have 

\begin{itemize}

\item[(i)] if $\mathcal{L}\cong \mathcal{L}'$ then $q_\mathcal{L}=q_{\mathcal{L}'}$,
\item[(ii)] the function $q_\mathcal{L}$ is a quadratic form on $A[2]$ (valued in $\boldsymbol \mu_2$) whose associated bilinear pairing is $(~,~)_\mathcal{L}$,
\item[(iii)] if $\mathcal{M}$ is another symmetric line bundle then $q_{\mathcal{L}\otimes \mathcal{M}}=q_\mathcal{L}\cdot q_{\mathcal{M}}$.
\end{itemize}
\end{lemma}

\begin{proof}
Part (i) is immediate. For parts $(ii)$ and $(iii)$ see e.g. \cite[Section 2]{MUMFORD1966} or \cite[Proposition 3.2]{MR2915483}.  
\end{proof}

For a principal polarisation $\lambda:A\rightarrow A^{\vee}$ defined over $F$, we can use \Cref{properties of quadratic forms associated to line bundles} to give a geometric interpretation of the principal homogeneous space for $A[2]$ associated to the set of quadratic refinements of the Weil pairing  $(~,~)_\lambda$ on $A[2]$.

\begin{defi} \label{identification of torsors: Kestutis remark}
Let $\lambda:A\rightarrow A^{\vee}$ be a self-dual homomorphism defined over $F$. We define $\mathfrak{c}_\lambda\in H^1(F,A^{\vee}[2])$ to be the image of $\lambda$ under the connecting homomorphism in the long exact for Galois cohomology associated to \cref{neron severi refinement}. If $\lambda$ is a principal polarisation we will also, by an abuse of notation, write $\mathfrak{c}_\lambda$ for the element $\lambda^{-1}(\mathfrak{c}_\lambda)\in H^1(F,A[2])$.
\end{defi}

\begin{lemma}  \label{identification of torsors: Kestutis lemma}
Let $\lambda:A\rightarrow A^{\vee}$ be a principal polarisation defined over $F$, so that $(~,~)_\lambda$ is a non-degenerate, $G_F$-equivariant, alternating pairing on $A[2]$. Then $G_F$ acts on $A[2]$ through the symplectic group $\textup{Sp}(A[2])$ associated to the pairing $(~,~)_\lambda$. Let $\mathfrak{c}\in H^1(F,A[2])$ be the cohomology class associated  to the set of quadratic refinements of $(~,~)_\lambda$ as in \Cref{quadratic refinements subsection}.

Then we have the equality $\mathfrak{c}=\mathfrak{c}_\lambda$ inside $H^1(F,A[2])$. 
\end{lemma}

\begin{proof}
We remark that this Lemma is implicit in \cite[Section 3]{MR2915483}. First note that by \Cref{properties of quadratic forms associated to line bundles}(ii),  for any symmetric line bundle $\mathcal{L}$ for which $\lambda=\phi_\mathcal{L}$, the function $q_\mathcal{L}$ is a quadratic refinement of $(~,~)_\lambda$. The result now follows either by an explict computation using the association $\mathcal{L}\mapsto q_\mathcal{L}$ or, more conceptually, from the long exact sequences for cohomology associated to the commutative diagram $(16)$ of \cite[Section 3.4]{MR2915483}, the top row of which is our sequence \cref{neron severi refinement} and the bottom row of which is the exact sequence \cref{construction of torsor of quadratic refinements} of \Cref{exact sequence quadratic refinement remark}.
\end{proof}

\subsection{Theta groups} \label{theta groups subsection}

In this subsection we suppose that $\mathcal{L}$ is an ample line bundle on $A$ so that $K(\mathcal{L})$ is finite.  We recall the definition of the Theta group assoicated to $\mathcal{L}$ (see \cite{MUMFORD1966} for more details of what follows). 

\begin{defi}
The \emph{Theta group} $\mathscr{G}(\mathcal{L})$ associated to $\mathcal{L}$ is the set of pairs $(x,\varphi)$ where $x\in K(\mathcal{L})$ and $\varphi$ is an isomorphism $\varphi:\mathcal{L}\stackrel{\sim}{\longrightarrow} \tau_x^*\mathcal{L}$ (over $\bar{F}$). The group operation is given by
\[(x,\varphi)\cdot (x',\varphi')=(x+x',\tau_{x'}^*(\varphi) \circ \varphi').\]
\end{defi}

\begin{remark} \label{invariance of theta group}
If $\mathcal{L}\cong \mathcal{L}'$ then fixing an isomorphism $\alpha:\mathcal{L}\stackrel{\sim}{\longrightarrow}\mathcal{L}'$ we obtain an isomorphism $\mathscr{G}(\mathcal{L})\stackrel{\sim}{\longrightarrow}\mathscr{G}(\mathcal{L'})$ given by 
\[(x,\varphi)\longmapsto (x,\tau_x^*(\alpha)\circ \varphi \circ \alpha^{-1})\]
which is  independent of $\alpha$ (since any two choices differ by a scalar). As such,  $\mathscr{G}(\mathcal{L})$ is canonically isomorphic to $\mathscr{G}(\mathcal{L}')$.
\end{remark}

\begin{remark} \label{theta group exact sequence}
The group $\mathscr{G}(\mathcal{L})$ sits in a short exact sequence 
\begin{equation}
0\rightarrow \bar{F}^{\times}\longrightarrow \mathscr{G}(\mathcal{L})\longrightarrow K(\mathcal{L})\rightarrow 0,
\end{equation} the map $\mathscr{G}(\mathcal{L})\rightarrow K(\mathcal{L})$ being projection onto the first factor and the map $\bar{F}^{\times}\rightarrow \mathscr{G}(\mathcal{L})$ sending $\eta\in \bar{F}^{\times}$ to the pair $(0,\textup{mult. by }\eta)$.
\end{remark}

\begin{lemma} \label{functoriality of theta}
We have the following  functorial properties of $\mathscr{G}$:
\begin{itemize}
\item[(i)] let $A/F$ and $B/F$ be abelian varieties, let $\mathcal{L}$ be an ample line bundle on $B$ 
and let $f:A\rightarrow B$ be an isomorphism. Then the map
$\tilde{f}:\mathscr{G}(f^*\mathcal{L})\stackrel{\sim}{\longrightarrow}\mathscr{G}(\mathcal{L})$   given by 
\[\left(x,\varphi\right)\mapsto \left(f(x),(f^{-1})^*(\varphi)\right)\]
is an isomorphism making the diagram 
\[
\xymatrix{0\ar[r] & \bar{F}^{\times}\ar@{=}[d]\ar[r] & \mathscr{G}(f^*\mathcal{L})\ar[d]^{\tilde{f}}\ar[r] & K(f^*\mathcal{L})\ar[d]^{f}\ar[r] & 0\\
0\ar[r] & \bar{F}^{\times}\ar[r] & \mathscr{G}(\mathcal{L})\ar[r] & K(\mathcal{L})\ar[r] & 0
}
\]
commute,
\item[(ii)] given abelian varieties $A/F,B/F$ and $C/F$, isomorphisms $f_1:A\rightarrow B$ and $f_2:B\rightarrow C$ and an ample line bundle $\mathcal{L}$ on $C$, we have
\[\widetilde{f_2\circ f_1}=\tilde{f_2}\circ \tilde{f_1}:\mathscr{G}(f_1^*f_2^*\mathcal{L})\stackrel{\sim}{\longrightarrow}\mathscr{G}(\mathcal{L}).\] 
\end{itemize}
\end{lemma}

\begin{proof}
In both cases this is a simple computation. We remark that we crucially require that $f$ is an isomorphism in $(i)$, the situation for a general homomorphism being more subtle. See, for example, \cite[Proposition 2]{MUMFORD1966} and the surrounding discussion.
\end{proof}

\subsection{Theta groups in the main case of interest} \label{case of interest}

Suppose that $A$ is equipped with a fixed principal polarisation $\lambda:A\rightarrow A^{\vee}$ defined over $F$ and  take $\mathcal{L}=(1,\lambda)^*\mathcal{P}$ where $\mathcal{P}$ is the Poincar\'{e} line bundle on $A\times A^{\vee}$ (here, for a homomorphism $\mu:A\rightarrow A^{\vee}$ we denote by $(1,\mu):A\rightarrow A\times A^{\vee}$ the composition of the diagonal morphism $\Delta:A\rightarrow A\times A$ with the morphism $1\times \mu:A\times A\rightarrow A\times A^{\vee}$). Then $\mathcal{L}$ is an $F$-rational, ample, symmetric line bundle on $A$ such that $\phi_\mathcal{L}=2\lambda$ (see \cite[Remark 4.5]{MR2833483}). In particular, we have $K(\mathcal{L})=\ker(2\lambda)=A[2]$.

Since $[-1]^*\mathcal{L}\cong \mathcal{L}$ we have an induced isomorphism $\widetilde{[-1]}$ of $\mathscr{G}(\mathcal{L})$ as in \Cref{functoriality of theta}.

%

\begin{lemma} \label{triviality of mult by -1}
With $\mathcal{L}=(1,\lambda)^*\mathcal{P}$ as above, the automorphism $\widetilde{[-1]}$ of $\mathscr{G}(\mathcal{L})$ is trivial.
\end{lemma}

\begin{proof}
By \cite[Proposition 3]{MUMFORD1966}, if $\mathcal{F}$ is any ample symmetric line bundle on $A$ and $(x,\varphi)\in \mathscr{G}(\mathcal{F})$ is such that $x\in A[2]$, then the automorphism $\widetilde{[-1]}$ of  $\mathscr{G}(\mathcal{F})$ sends $(x,\varphi)$ to $(x,q_\mathcal{F}(x)\varphi)$. 

In particular, since $K(\mathcal{L})=A[2]$ in our case, it suffices to show that $q_\mathcal{L}$ is trivial. Pick a symmetric line bundle $\mathcal{M}$ such that $\lambda=\phi_\mathcal{M}$ (whilst it may not be possible to choose an $F$-rational such $\mathcal{M}$, this is always possible over $\bar{F}$). By standard properties of the Poincar\'{e} line bundle we have $(1\times\phi_\mathcal{M})^*\mathcal{P}\cong m^*\mathcal{M}\otimes p_1^*\mathcal{M}^{-1}\otimes p_2^*\mathcal{M}^{-1}$, where $m:A\times A\rightarrow A$ is addition and $p_1$ (resp. $p_2$) denotes projection onto the first (resp. second) factor. Pulling back along the diagonal morphism $\Delta:A\rightarrow A\times A$ we obtain 
\[\mathcal{L}\cong [2]^*\mathcal{M}\otimes \mathcal{M}^{-2}\cong \mathcal{M}^2\]
where for the second isomorphism above we use the symmetry of $\mathcal{M}$ along with the fact that for any line bundle $\mathcal{F}$ on $A$ we have $[2]^*\mathcal{F}\cong \mathcal{F}^3\otimes[-1]^*\mathcal{F}$ (see e.g. \cite[Corollary 6.6]{MR861974}). By \Cref{properties of quadratic forms associated to line bundles}(iii) we conclude that $q_\mathcal{L}=(q_\mathcal{M})^2=1$ as desired.
\end{proof}

\begin{remark} \label{Galois action}
As $\mathcal{L}$ is $F$-rational, the group $\mathscr{G}(\mathcal{L})$ carries a natural $G_F$-action. Explicitly, for $\sigma \in G_F$ and $(x,\varphi)\in \mathscr{G}(\mathcal{L})$, we have \[\sigma\cdot (x,\varphi)=(\sigma(x),\sigma^*(\varphi))\in \mathscr{G}(\sigma^*\mathcal{L})=\mathscr{G}(\mathcal{L})\] where for the equality  $\mathscr{G}(\sigma^*\mathcal{L})=\mathscr{G}(\mathcal{L})$ we combine \Cref{invariance of theta group} with the assumption that $\mathcal{L}$ is $F$-rational. In particular, the exact sequence of \Cref{theta group exact sequence} becomes a short exact sequence of $G_F$-modules
\begin{equation}\label{specific theta exact seqeunce}
0\rightarrow \bar{F}^{\times}\longrightarrow \mathscr{G}(\mathcal{L})\longrightarrow A[2]\rightarrow 0.
\end{equation}
This short exact sequence will be important in what follows. More specifically, as in \cite[Corollary 4.7]{MR2833483}, the associated connecting homomorphism $H^1(F,A[2])\rightarrow H^2(F,\bar{F}^{\times})$ is a quadratic form whose associated bilinear pairing is that arising from cup-product and the Weil pairing $(~,~)_\lambda:A[2]\times A[2]\rightarrow \boldsymbol \mu_2\hookrightarrow \bar{F}^{\times}$.
\end{remark}

%
%

\subsection{Quadratic twists} \label{quadratic twisting theta groups}

Maintaining the notation of \Cref{case of interest} (so in particular $\mathcal{L}=(1,\lambda)^*\mathcal{P}$) let $\chi:G_F\rightarrow \boldsymbol \mu_2$ be a quadratic character. Write $(A^\chi,\psi)$ for the quadratic twist of $A$ by $\chi$.
We now consider the effect of quadratic twisting on the constructions appearing earlier in this section. Note that $\psi$ restricts to a $G_F$-equivariant isomorphism $A[2]\stackrel{\sim}{\longrightarrow}A^\chi[2]$. 

\begin{lemma} \label{descent of polarisation}
The morphism $\lambda_\chi:=(\psi^\vee)^{-1} \lambda \psi^{-1}:A^{\chi}\rightarrow A^{\chi \vee}$ is a principal polarisation defined over $F$.
\end{lemma}

\begin{proof}
This  is a manifestation of the fact that $[-1]^*$ acts trivially on the N\'{e}ron--Severi group. More precisely, one computes immediately that $\lambda_\chi$ is defined over $F$, and it's a polarisation since if $\mathcal{M}$ is a line bundle on $A$ (not necessarily $F$-rational) such that $\lambda=\phi_\mathcal{M}$ then one has $\lambda_\chi=\phi_{\mathcal{M}_\chi}$ where $\mathcal{M}_\chi=(\psi^{-1})^*\mathcal{M}$.  
\end{proof}

More generally, we have. 

\begin{lemma} \label{commutative twist diagram symm}
We have a commutative diagram of $G_F$-modules 
\[
\xymatrix{0\ar[r] & A^{\chi \vee}[2]\ar[d]^{\psi^{\vee}}\ar[r] & \textup{Pic}^{\textup{sym}}A_{\bar{F}}^\chi\ar[d]^{\psi^*}\ar[r] & \textup{Hom}_{\textup{self-dual}}(A_{\bar{F}}^\chi,A_{\bar{F}}^{\chi \vee})\ar[d]\ar[r] & 0\\
0\ar[r] & \A^{\vee}[2]\ar[r] & \textup{Pic}^{\textup{sym}}A_{\bar{F}}\ar[r] &\textup{Hom}_{\textup{self-dual}}(A_{\bar{F}},A_{\bar{F}}^{\vee})\ar[r] & 0,
}
\]
where the rightmost vertical map  sends $\mu$ to $\psi \mu \psi^{\vee}$. 
\end{lemma}

\begin{proof}
As with \Cref{descent of polarisation} this follows from an explicit computation, and results from the fact that $[-1]$ acts trivially on each group appearing.
\end{proof}

\begin{cor} \label{twisting the torsor}
Let $\bar{\psi}^{-1}$ denote the isomorphism $H^1(F,A^\chi[2])\rightarrow H^1(F,A[2])$ induced by $\psi^{-1}$ and let $\mathfrak{c}_\lambda \in H^1(F,A[2])$ (resp. $\mathfrak{c}_{\lambda_\chi}\in H^1(F,A^\chi[2])$) be the cohomology class associated to $\lambda$ (resp. $\lambda_\chi$) as in \Cref{identification of torsors: Kestutis remark}. Then we have $\bar{\psi}^{-1}(\mathfrak{c}_{\lambda_\chi})=\mathfrak{c}_\lambda$.
\end{cor}

\begin{proof}
This follows immediately from the long exact sequences for cohomology associated to the commutative diagram of \Cref{commutative twist diagram symm}.
\end{proof}

We now consider the effect of quadratic twisting on the Theta group associated to $\mathcal{L}=(1,\lambda)^*\mathcal{P}$.

\begin{lemma} \label{pull back of line bundle}
Let $\mathcal{L}=(1,\lambda)^*\mathcal{P}$, write $\mathcal{P}_\chi$ for the Poincar\'{e} line bundle on $A^\chi\times A^{\chi \vee}$ and define $\mathcal{L}_{\chi}:=(1,\lambda_\chi)^*\mathcal{P}_\chi$. Then $\psi^*\mathcal{L}_\chi \cong \mathcal{L}$.
\end{lemma}
 
\begin{proof}
Standard properties of the Poincar\'{e} line bundle (see e.g. \cite[Section 11]{MR861974}) give 
\[(1\times\psi^{\vee})^*\mathcal{P}\cong (\psi\times 1)^*\mathcal{P}_\chi\]
as line bundles on $A\times A^{\chi \vee}$. Since $\psi^{\vee}$ is an isomorphism we obtain
\[\mathcal{L}=(1, \lambda)^*\mathcal{P}\cong (1,\lambda)^*(1\times (\psi^{\vee})^{-1})^*(\psi\times1)^*\mathcal{P}_\chi.\]
The right hand side of the above expression is easily seen to be equal to
\[\psi^*\Delta^*(1\times \lambda_\chi)^*\mathcal{P}_{\chi}=\psi^*\mathcal{L}_{\chi}\]
as desired (here $\Delta:A\rightarrow A\times A$ is the diagonal morphism).
\end{proof}

\begin{lemma} \label{main twist of theta groups lemma}
The isomorphism $\tilde{\psi}:\mathscr{G}(\mathcal{L})\rightarrow \mathscr{G}(\mathcal{L}_\chi)$ (arising from \Cref{functoriality of theta} and \Cref{pull back of line bundle}) is Galois equivariant. In particular, $\tilde{\psi}$ fits into a commutative diagram of $G_F$-modules 
\[
\xymatrix{0\ar[r] & \bar{F}^{\times}\ar@{=}[d]\ar[r] & \mathscr{G}(\mathcal{L})\ar[d]^{\tilde{\psi}}\ar[r] & A[2]\ar[d]^{\psi}\ar[r] & 0\\
0\ar[r] & \bar{F}^{\times}\ar[r] & \mathscr{G}(\mathcal{L}_\chi)\ar[r] & A^\chi[2]\ar[r] & 0,
}
\]
where all vertical maps are isomorphisms. 
\end{lemma}

\begin{proof}
Write $\textup{Isom}_{\mathcal{L},\mathcal{L}_\chi}(A,A^\chi)$ for the set of $\bar{F}$-isomorphisms $f:A\rightarrow A^\chi$ for which $f^*\mathcal{L}_\chi\cong \mathcal{L}$. Then using the explicit Galois action given in \Cref{Galois action} one sees that the map
$\textup{Isom}_{\mathcal{L},\mathcal{L}_\chi}(A,A^\chi)\rightarrow \textup{Isom}\left(\mathscr{G}(\mathcal{L}),\mathscr{G}(\mathcal{L}_\chi)\right)$
given by $f\mapsto\tilde{f}$ is Galois equivariant. It then follows from \Cref{functoriality of theta} that we have, for all $\sigma \in G_F$,
\[({\widetilde{\psi}})^{\sigma}=\widetilde{\psi^\sigma}=\widetilde{\psi\circ [\chi(\sigma)]}=\widetilde{\psi}\circ \widetilde{[\chi(\sigma)]}=\tilde{\psi}\]
where the last equality follows from \Cref{triviality of mult by -1}.
\end{proof}


\section{Controlling the parity of $\dim_{\mathbb{F}_2}\Sha_{\mathrm{nd}}(A/K)[2]$ under quadratic twist}

In this section we prove \Cref{local sha twist decomp} concerning the behaviour under quadratic twist of the Shafarevich--Tate group of a principally polarised abelian variety.

For the rest of the section, fix a number field $K$ and and let $(A/K,\lambda)$ be a principally polarised abelian variety. To fix notation, we briefly recall the definition of the $2$-Selmer and Shafarevich--Tate groups of $A/K$.

\subsection{The  $2$-Selmer group and the Shafarevich--Tate group}

For a place $v$ of $K$ we denote by $\delta_v:A(K_v)/2A(K_v)\hookrightarrow H^1(K_v,A[2])$ the connecting homomorphism associated to the multiplication-by-two Kummer sequence 
\begin{equation}
0\rightarrow A[2]\longrightarrow A(\bar{K_v})\stackrel{[2]}{\longrightarrow}A(\bar{K_v})\rightarrow 0.
\end{equation}
over the completion $K_v$ of $K$ at $v$. 

The $2$-Selmer group of $A/K$ is the group
\[\textup{Sel}^2(A/K):=\{\xi\in H^1(K,A[2])~~:~~\xi_v\in \textup{im}(\delta_v)~~\forall v\in M_K\}.\]

It sits in a short exact sequence 
\begin{equation} \label{2-selmer sequence}
0\rightarrow A(K)/2A(K)\longrightarrow \textup{Sel}^2(A/K)\longrightarrow \Sha(A/K)[2]\rightarrow 0
\end{equation}
where \[\Sha(A/K):=\ker(H^1(K,A)\rightarrow \prod_{v\in M_K}H^1(K_v,A))\] 
is the Shafarevich--Tate group of $A/K$. 

\subsection{The Cassels--Tate pairing}
Denote  by $\Sha_{\mathrm{nd}}(A/K)$ the quotient of $\Sha(A/K)$ by its maximal divisible subgroup. The Cassels--Tate pairing is a bilinear pairing
\[\left \langle ~,~\right \rangle_{\mathrm{CT}}:\Sha(A/K)\times \Sha(A^\vee/K)\longrightarrow \mathbb{Q}/\mathbb{Z}\]
the left (resp. right) kernel of which is $\Sha_{\textup{nd}}(A/K)$ (resp. $\Sha_{\textup{nd}}(A^{\vee}/K$)). The principal polarisation $\lambda:A\rightarrow A^{\vee}$ induces a non-degenerate bilinear pairing  
\[\left \langle ~,~\right \rangle_{\mathrm{CT},\lambda}:\Sha_{\mathrm{nd}}(A/K)\times \Sha_{\mathrm{nd}}(A/K)\longrightarrow \mathbb{Q}/\mathbb{Z}\]
defined by $ \left \langle a,b\right \rangle_{\mathrm{CT},\lambda}=\left \langle a,\lambda(b)\right \rangle_{\mathrm{CT}}$ for $a,b\in \Sha(A/K)$. This pairing is antisymmetric (\cite[Theorem 2]{MR1079004}, see also \cite[Corollary 6]{MR1740984}). 

Via the map $\textup{Sel}^2(A/K)\rightarrow \Sha(A/K)[2]$ of \cref{2-selmer sequence} the Cassels--Tate pairing $\left \langle ~,~\right \rangle_{\textup{CT},\lambda}$ induces an  antisymmetric pairing on $\textup{Sel}^2(A/K)$ (though this is no longer non-degenerate). By an abuse of notation, we denote this by $\left \langle ~,~\right \rangle_{\textup{CT},\lambda}$ also.

 \subsection{Description of the Cassels--Tate pairing on $\textup{Sel}^2(A/K)$} \label{description of CT pairing}

We will need an explicit description of the Cassels--Tate pairing $\left \langle ~,~\right \rangle_{\mathrm{CT},\lambda}$ on $\textup{Sel}^2(A/K)$. We use the `Weil pairing' definition as in \cite[Section 12.2]{MR1740984} which we copy almost verbatim and to which we refer for more details.

\begin{defi}[Cassels--Tate pairing] \label{defi of ct pairing}
Let $\mathfrak{a}$, $\mathfrak{b}\in \textup{Sel}^2(A/K)$. There will be several choices involved in the definition of $\left \langle \mathfrak{a},\mathfrak{b}\right \rangle_{\mathrm{CT},\lambda}$. We begin with the global choices.

Pick cocycles $a$ and $b$ representing $\mathfrak{a}$ and $\mathfrak{b}$ respectively. Next, pick $\sigma \in C^1(K,A[4])$ such that $2\sigma=a$. Then $d\sigma$ is a 2-cocycle with values in $A[2]$, i.e. an element of $Z^2(K,A[2])$. The Weil pairing $(~,~)_\lambda:A[2]\times A[2] \rightarrow \boldsymbol \mu_2\hookrightarrow \bar{K}^{\times}$ induces a cup-product map $\cup:Z^2(K,A[2])\times Z^1(K,A[2]) \rightarrow Z^3(K,\bar{K}^{\times})$. As $K$ is a number field  $H^3(K,\bar{K}^{\times})=0$, so we may choose $\epsilon \in C^2(K,\bar{K}^{\times})$ such that $d\sigma \cup b= d\epsilon$.

Now for the local choices. Fix a place $v$ of $K$. The class of $a_v$  is trivial in $H^1(K_v,A(\bar{K}_v))$ so we may choose $P_v\in A(\bar{K_v})$ with $a_v=dP_v$. Pick $Q_v\in A(\bar{K_v})$ with $2Q_v=P_v$. Then $\rho_v:=dQ_v$ is an element of $Z^1(K_v,A[4])$ and $\sigma_v-\rho_v$ takes values in $A[2]$, i.e. is an element of $C^1(K_v,A[2])$. Then we may form the element $(\sigma_v-\rho_v)\cup b_v$ of $C^2(K_v,\bar{K_v}^{\times})$ (again defining the cup-product map using the Weil pairing on $A[2]$). The difference $(\sigma_v-\rho_v)\cup b_v   - \epsilon_v$ is a $2$-cocycle with values in $\bar{K_v}^{\times}$. Let $\mathfrak{c}_v$ denote its class in $H^2(K_v,\bar{K}_v^{\times})=\textup{Br}(K_v)$. Then 
$\left \langle \mathfrak{a},\mathfrak{b}\right \rangle _{\mathrm{CT},\lambda}$ is defined as
\[\left \langle \mathfrak{a},\mathfrak{b}\right \rangle _{\mathrm{CT},\lambda}:=\sum_{v\in M_K} \textup{inv}_v(\mathfrak{c}_v)~~\in \mathbb{Q}/\mathbb{Z}.\]
The value of the sum above is independent of all choices made.
\end{defi}

\subsection{Controlling the parity of $\dim_{\mathbb{F}_2}\Sha_{\mathrm{nd}}(A/K)[2]$ globally}

If $A$ is an elliptic curve and $\lambda$ its canonical principal polarisation then it's well known that $\left \langle ~,~\right \rangle_{\mathrm{CT},\lambda}$ is in fact alternating and it follows that $\dim_{\mathbb{F}_2}\Sha_{\mathrm{nd}}(A/K)[2]$ is even. 
For general principally polarised abelian varieties however, Poonen and Stoll showed in \cite{MR1740984} that $\dim_{\mathbb{F}_2}\Sha_{\mathrm{nd}}(A/K)[2]$ need not be even and gave a criterion for determining whether or not this is the case. Specificaly, let $\mathfrak{c}_\lambda\in H^1(K,A[2])$ be the cohomology class associated to $\lambda$ as in \Cref{identification of torsors: Kestutis remark}. By  \cite[Lemma 1]{MR1740984} we in fact have $\mathfrak{c}_\lambda\in \textup{Sel}^2(A/K)$. 

We then have the following theorem of Poonen--Stoll.

\begin{theorem} \label{poonen stoll}
The group $\Sha_{\mathrm{nd}}(A/K)[2]$ has even $\mathbb{F}_2$-dimension if and only if 
\[\left\langle \mathfrak{c}_\lambda,\mathfrak{c}_\lambda \right \rangle_{\mathrm{CT},\lambda}=0\in \mathbb{Q}/\mathbb{Z}.\]
\end{theorem}

\begin{proof}
The image of $\mathfrak{c}_\lambda$ in $\Sha(A/K)[2]$ is the homogeneous space associated to $\lambda$ as in \cite[Section 2]{MR1740984}. Theorem 8 of op. cit. now gives the result.
\end{proof}

\begin{remark} \label{order of pairing remark}
Since the image of $\mathfrak{c}_\lambda$ in $\Sha(A/K)$ is annihilated by $2$ we have $\left\langle \mathfrak{c}_\lambda,\mathfrak{c}_\lambda \right \rangle_{\mathrm{CT},\lambda}\in \{0,\frac{1}{2}\}$.
\end{remark}

\subsection{Quadratic twists}  \label{notation for quadratic twist cohomology}

For the rest of the section fix a quadratic character $\chi$ and let $(A^\chi,\psi)$ be the quadratic twist of $A$ by $\chi$. We now set up the notation which we will use when computing with $A^\chi$ in what follows.  We endow $A^\chi$ with the $K$-rational principal polarisation $\lambda_\chi:=(\psi^\vee)^{-1} \lambda \psi^{-1}$ (see \Cref{quadratic twisting theta groups}). Associated to $\lambda_\chi$ is the Weil pairing \[(~,~)_{\lambda_\chi}:A^\chi[2]\times A^{\chi}[2]\rightarrow \boldsymbol \mu_2\] and the Cassels--Tate pairing \[\left \langle ~,~\right \rangle_{\textup{CT},\lambda_\chi}:\Sha(A^\chi/K)[2]\times \Sha(A^\chi/K)[2]\rightarrow \mathbb{Q}/\mathbb{Z}\] (which we also view as a pairing on $\textup{Sel}^2(A^\chi/K)$). Using the isomorphism $\psi$ 
we identify $A^\chi[2]$ and $A[2]$ as $G_K$-modules. Note that this identification also respects the Weil pairing (i.e. identifies $(~,~)_{\lambda_\chi}$ with $(~,~)_\lambda$ ; to see this e.g. combine \Cref{properties of quadratic forms associated to line bundles}(ii) and \Cref{commutative twist diagram symm}). In this way, we  identify $H^1(K,A^\chi[2])$ with $H^1(K,A[2])$ and thus view the $2$-Selmer group $\textup{Sel}^2(A^\chi/K)$ inside $H^1(K,A[2])$. In particular, we may talk about the intersection of $\textup{Sel}^2(A/K)$ and $\textup{Sel}^2(A^{\chi}/K)$.  

We also use $\psi$ to identify $A[4](\bar{K})$ with $A^\chi[4](\bar{K})$. This last identification does not respect the $G_K$-action. Thus for each $i$, we have identified $C^i(K,A^\chi[4])$ with $C^i(K,A[4])$ but the differential $d:C^i(K,A^\chi[4])\rightarrow C^{i+1}(K,A^\chi[4])$ is not identified with the usual differential on $C^i(K,A[4])$; we write $d_\chi$ for the map $C^i(K,A[4]) \rightarrow C^{i+1}(K,A[4])$ to which is does correspond. For example, the map $d:C^1(K,A^\chi[4])\rightarrow C^2(K,A^\chi[4])$ corresponds to the map $d_\chi:C^1(K,A[4])\rightarrow C^2(K,A[4])$ defined by 
\[(d_\chi f)(\sigma,\tau)=f(\sigma)+\chi(\sigma)\sigma f(\tau)- f(\sigma\tau).\] Similarly, we use $\psi$ to identify $C^i(K,A^\chi(\bar{K}))$ and $C^i(K,A(\bar{K}))$ for each $i$, and define differentials $d_\chi$ on $C^i(K,A(\bar{K}))$ similarly.

\subsection{Strategy of the  proof of \Cref{local sha twist decomp}}

To motivate what follows, we briefly sketch the proof of  \Cref{local sha twist decomp}. 

For $\mathfrak{a},\mathfrak{b}\in \Sha(A/K)$, in the definition of  $\left \langle \mathfrak{a},\mathfrak{b}\right \rangle_{\textup{CT},\lambda}$   the local terms $\mathfrak{c}_v$ (in the notation of \Cref{defi of ct pairing}) depend on the global choices. In particular, it is not clear  that $\left\langle \mathfrak{c}_\lambda,\mathfrak{c}_\lambda \right \rangle_{\mathrm{CT},\lambda}$, and hence the parity of $\dim_{\mathbb{F}_2}\Sha_{\textup{nd}}(A/K)[2]$, may be expressed as a sum of local terms whose definition requires no global choices (this is, however, known to be true if $A/K$ is the Jacobian of a curve: see \cite[Corollary 12]{MR1740984}).

When considering $A$ along with its quadratic twist $A^\chi$, we eliminate the global choices as follows.
 Associated to $\lambda_\chi$ is the class $\mathfrak{c}_{\lambda_\chi}\in \textup{Sel}^2(A^\chi/K)$ (viewed inside $H^1(K,A[2])$ as in \Cref{notation for quadratic twist cohomology}). By \Cref{twisting the torsor} we have $\mathfrak{c}_{\lambda_\chi}=\mathfrak{c}_\lambda$ and in particular, $\mathfrak{c}_\lambda$ lies in $\textup{Sel}^2(A/K)\cap \textup{Sel}^2(A^\chi/K)$. Now the sum  of the pairings $\left \langle ~,~\right \rangle_{\textup{CT},\lambda}$ and $\left \langle ~,~\right \rangle_{\textup{CT},\lambda_\chi}$ gives a new pairing on $\textup{Sel}^2(A/K)\cap \textup{Sel}^2(A^\chi/K)$. By \Cref{poonen stoll},  $\dim_{\mathbb{F}_2}\Sha_{\textup{nd}}(A/K)[2]+\dim_{\mathbb{F}_2}\Sha_{\textup{nd}}(A^\chi/K)[2]$ is even if and only if $\mathfrak{c}_\lambda$ pairs trivially with itself under this new pairing. 

We show in \Cref{identification of pairings} that the global choices involved in computing the sum of the two Cassels--Tate pairings  are milder than those for the individual pairings (we remark that this simplification of the Cassels--Tate pairing under quadratic twist has also been observed in the recent preprint of Smith \cite[Proof of Theorem 3.2]{Smith16}). Specifically, the global choices involved in computing  
\[\left \langle \mathfrak{c}_\lambda,\mathfrak{c}_\lambda\right \rangle_{\textup{CT},\lambda}+\left \langle \mathfrak{c}_\lambda,\mathfrak{c}_\lambda\right \rangle_{\textup{CT},\lambda_\chi}\]
are: a choice of cocycle $c_\lambda \in Z^1(K,A[2])$ representing $\mathfrak{c}_\lambda$, and a choice of cochain $F_c:G_K\rightarrow \boldsymbol \mu_2$ such that $dF_c=c_\lambda\cup c_\lambda \in Z^2(K,\boldsymbol \mu_2)$.

By \Cref{identification of torsors: Kestutis lemma},  $\mathfrak{c}_\lambda\in H^1(K,A[2])$ is the cohomology class parameterizing quadratic refinements of the Weil pairing. In particular, a choice of cocycle representing $\mathfrak{c}_\lambda$ amounts to a choice of quadratic refinement $q$. For each such $q$ we have already constructed a canonical choice for the function $F_c$ above, namely that given by \Cref{extension of dickson invariant}.  Thus  the only global choice remaining is that of $q$. \Cref{change of form}  shows how $F_c$ changes upon changing $q$, allowing us to prove that the local terms then arising do not, in fact, depend on the choice of quadratic refinement either.

\subsection{Pairings on $\text{Sel}^2(A/K)\cap \text{Sel}^2(A^{\chi}/K)$} \label{new pairing}

Define $\mathcal{S}_{\chi}:=\text{Sel}^2(A/K)\cap \text{Sel}^2(A^{\chi}/K)$. Here we define a pairing $\left \langle ~,~ \right \rangle_{S_\chi}$ on $\mathcal{S}_{\chi}$ with values in $\mathbb{Q}/\mathbb{Z}$ which we shall  see is  the sum of the Cassels--Tate pairings for $A$ and its twist $A^\chi$. However, for clarity when using this pairing later, we define it separately. 

\begin{defi}[The pairing $\left \langle ~,~ \right \rangle_{S_\chi}$] \label{definition of our altered pairing}
Let $\mathfrak{a}$,$\mathfrak{b}\in \mathcal{S}_{\chi}=\textup{Sel}^2(A/K)\cap \textup{Sel}^2(A^{\chi}/K)$. As with the definition of the Cassels--Tate pairing, we begin with the global choices. We first claim that $\mathfrak{a}\cup \mathfrak{b}=0\in H^2(K,\boldsymbol\mu_2)=\textup{Br}(K)[2]$. Indeed, for each place $v$ of $K$ both $\mathfrak{a}_v$ and $\mathfrak{b}_v$  are in the image of $A(K_v)/2A(K_v)$ under the connecting homomorphism associated to the multiplication-by-2 Kummer sequence.  Since this image is its own orthogonal complement under the cup-product pairing \[H^1(K_v,A[2])\times H^1(K_v,A[2])\rightarrow H^2(K_v,\bar{K_v}^{\times})=\textup{Br}(K_v)\] (this is well known and results from Tate local duality, see e.g. \cite[I.3.4]{MR2261462}) we have $(\mathfrak{a}\cup \mathfrak{b})_v=0\in \textup{Br}(K_v)$ for each place $v$ of $K$. Reciprocity for the Brauer group now gives the claim.  

Now represent $\mathfrak{a}$ and $\mathfrak{b}$ by cocycles $a$ and $b$ respectively and, as is possible by the above discussion, pick $f\in C^1(K,\boldsymbol\mu_2)$ with $df=a\cup b\in \text{Z}^2(K,\boldsymbol \mu_2)$.  

We now turn to the local choices. Fix a place $v$ of $K$. Since $\mathfrak{a}\in \textup{Sel}^2(A/K)$ there is $P_v\in A(\bar{K_v})$ with $dP_v=a_v$. Pick $Q_v\in  A(\bar{K_v})$ with $2Q_v=P_v$. Then $\rho_v:=dQ_v$ is an element of $Z^1(K_v,A[4])$. Since $\mathfrak{a}$ is also in $\textup{Sel}^2(A^{\chi}/K)$  we can similarly (i.e. by replacing $d$ by $d_\chi$ throughout) define $P_{v,\chi}$, $Q_{v,\chi}$ and $\rho_{v,\chi}=d_{\chi}Q_{v,\chi}\in C^1(K_v,A[4])$. Then $\rho_{v}+\rho_{v,\chi}$ takes values in $A[2]$. One checks that $d(\rho_{v}+\rho_{v,\chi})=\chi_v\cup a_v\in Z^2(K_v,A[2])$. Thus the difference
\[(\rho_{v}+\rho_{v,\chi})\cup b_v ~~ - ~~  \chi_v \cup f_v\]
is a $2$-cocycle with values in $\boldsymbol\mu_2$. Denote by $\mathfrak{c}_v$ its class in $\textup{Br}(K_v)[2]$. 

Now define 
\[\left \langle a,b \right \rangle_{S_\chi} :=\sum_{v\in M_K} \textup{inv}_v(\mathfrak{c}_v)\in \mathbb{Q}/\mathbb{Z}.\]
One easily checks that once the initial global choices are made the cocycle class $\mathfrak{c}_v\in \textup{Br}(K_v)$ is independent of the local choices. That the resulting sum is independent of all choices follows from reciprocity for the Brauer group.
\end{defi}

\begin{remark}
If a place $v$ of $K$ splits in the quadratic extension $L/K$ associated to $\chi$ then $\chi_v$ is trivial and $\psi$ gives an isomorphism between $A$ and $A^\chi$ over $K_v$. It follows easily that the local terms $\textup{inv}_v(\mathfrak{c}_v)$ are trivial at all such $v$. Thus in the definition of $\left \langle ~,~ \right \rangle_{S_\chi}$ we may replace the sum over all places of $K$ by the sum over all places of $K$ non-split in $L/K$.
\end{remark}

\begin{lemma} \label{identification of pairings}
The pairing $\left \langle ~,~\right \rangle_{S_\chi}$ is the sum of the Cassels--Tate pairings for $A$ and $A^{\chi}$:
\[\left \langle ~,~\right \rangle_{S_\chi}=\left \langle ~,~\right \rangle_{CT,\lambda}+\left \langle ~,~\right \rangle_{CT,\lambda_\chi}.\]
In particular, it is (anti-)symmetric.
\end{lemma}

\begin{remark}
This lemma is implicit in the recent arxiv preprint of Smith \cite[Proof of Theorem 3.2]{Smith16}.
\end{remark}

\begin{proof}
Fix $\mathfrak{a},\mathfrak{b}\in S_\chi$. We begin by making the global choices involved in computing
$\left \langle \mathfrak{a},\mathfrak{b}\right \rangle _{CT,\lambda}$.
We pick cocycles $a$ and $b$ representing $\mathfrak{a}$ and $\mathfrak{b}$ respectively and pick $\sigma \in C^1(K,A[4])$ with $2\sigma =a$. Next, we pick $\epsilon \in C^2(K,\bar{K}^{\times})$ with $d\epsilon=d\sigma \cup b$. 

We now make the corresponding choices involved in computing $\left \langle \mathfrak{a},\mathfrak{b}\right \rangle _{CT,\lambda_\chi}$.  As we are at liberty to do, we pick the same cocycle representitves $a$ and $b$ chosen above. We similarly pick the same element $\sigma$  of $C^1(K,A[4])$ satisfying $2\sigma =a$ (here using the identification of $A[4]$ with $A^\chi[4]$ via $\psi$ as discussed). We then pick $\epsilon_\chi\in C^2(K,\bar{K}^{\times})$ such that $d\epsilon_\chi=d_\chi\sigma \cup b$. Note that we cannot chose $\epsilon=\epsilon_\chi$ in general due to the difference between the differentials $d$ and $d_\chi$. However, we have 
\[d(\epsilon+\epsilon_\chi)=(d\sigma+d_\chi \sigma)\cup b= (\chi \cup a) \cup b,\]
the last equality following from the definition of $d_\chi$ and a simple computation. 

Now let $f\in C^1(K,\boldsymbol\mu_2)$ be such that $df=a\cup b$.  By \cref{cup product on cochains} and associativity of the cup-product we have $d(\chi \cup f)=d(\epsilon+\epsilon_\chi)$ whence $\chi \cup f=\epsilon+\epsilon_\chi + g$ for some cocycle $g\in Z^2(K,\bar{K}^{ \times})$.

We now make the local choices involved in computing $\left \langle \mathfrak{a},\mathfrak{b}\right \rangle _{CT,\lambda}$. We choose $P_v\in A(\bar{K_v})$ with $dP_v=a_v$ and then pick $Q_v \in A(\bar{K_v})$ with $2Q_v=P_v$. Next, set $\rho_v:=dQ_v \in C^1(K_v,A[4])$ 
and  define $\mathfrak{c}_v$ to be the class of $(\sigma_v -\rho_v)\cup b_v -\epsilon_v$ in $H^2(K_v,\bar{K_v}^{\times})$. 

Finally, we make the local choices involved in computing $\left \langle \mathfrak{a},\mathfrak{b}\right \rangle _{CT,\lambda_\chi}$. Thus we pick $P_{v,\chi}$ with $d_\chi P_{v,\chi}=a_v$, $Q_{v,\chi}$ with $2Q_{v,\chi}=P_{v,\chi}$, set $\rho_{v,\chi}=d_\chi Q_{v,\chi}$ and define $\mathfrak{c}_{v,\chi}$ to be the class of $(\sigma_v -\rho_{v,\chi})\cup b_v -\epsilon_{\chi,v}$ in $H^2(K_v,\bar{K_v}^{\times})$. 

We these choices in place, $\mathfrak{c}_{v}+\mathfrak{c}_{v,\chi}$ is the class in $\textup{Br}(K_v)$ of
\[(a_v-(\rho_v+\rho_{v,\chi}))\cup b_v  - \chi_v \cup f_v + g_v.\]
Noting that $\rho_v+\rho_{v,\chi}$ takes values in $A[2]$ and that $\mathfrak{a}_v\cup \mathfrak{b}_v=0$ (as discussed previously) we see that
\[\text{inv}_v(\mathfrak{c}_v)+\text{inv}_v(\mathfrak{c}_{v,\chi})=\text{inv}_v\left((\rho_v+\rho_{v,\chi})\cup b_v - \chi_v \cup f_v \right)+\text{inv}_v(g_v).\]
Summing over all places and noting that by reciprocity for the Brauer group we have
\[\sum_{v\in M_K}\text{inv}_v (g_v)=0\in \mathbb{Q}/\mathbb{Z},\]
we have
\[\left \langle \mathfrak{a},\mathfrak{b}\right \rangle _{CT}+\left \langle \mathfrak{a},\mathfrak{b}\right \rangle _{CT,\chi}=\sum_{v\in M_K}\text{inv}_v \left((\rho_v+\rho_{v,\chi})\cup b_v - \chi_v \cup f_v \right).\]
But this is precisely how the quantity $\left \langle \mathfrak{a},\mathfrak{b}\right \rangle_{S_\chi}$ was defined.
\end{proof}

\subsection{The local terms $\mathfrak{g}(A,\lambda,\chi)$}

In this subsection we study the local terms which arise in computing $\left \langle \mathfrak{c}_\lambda,\mathfrak{c}_\lambda\right \rangle_{\mathcal{S}_\chi}$, and show in particular that they are independent of certain choices involved. We work purely locally and take $F$ to be a local field of characteristic $0$. Let $(A/F,\lambda)$ be a principally polarised abelian variety. Let $\chi\in\textup{Hom}_{\textup{cnts}}(G_F,\boldsymbol \mu_2)$ be a quadratic character of $F$ and  $(A^\chi/F,\psi)$ be the quadratic twist of $A$ by $\chi$.  We use the same conventions and  notation as in \Cref{notation for quadratic twist cohomology} when talking about objects associated to $A^\chi$. 

Denote by $\mathfrak{c}_\lambda\in H^1(F,A[2])$ the cohomology class associated to $\lambda$ as in \Cref{identification of torsors: Kestutis remark}. By \cite[Lemma 1]{MR1740984} its image in $H^1(F,A)[2]$ is trivial. By \Cref{twisting the torsor}, it follows also that the image of $\mathfrak{c}_\lambda$ in $H^1(F,A^\chi)[2]$ is trivial too (here the map $H^1(F,A[2])\rightarrow H^1(F,A^\chi)[2]$ comes from identifying $A[2]$ with $A^\chi[2]$ via $\psi$). 

\begin{remark} \label{representative of c remark}
By \Cref{identification of torsors: Kestutis lemma} $\mathfrak{c}_\lambda$ is equal to the cohomology class associated to the set of quadratic refinements of the Weil pairing $(~,~)_\lambda$ on $A[2]$. In particular, for each quadratic refinement $q$ of $(~,~)_\lambda$, the function $c_q:G_F\rightarrow A[2]$ sending $\sigma \in G_F$ to the unique element $c_q(\sigma)\in A[2]$ such that 
\[ q(\sigma^{-1}v)-q(v)=(v,c_q(\sigma))_\lambda\] for all $v\in A[2]$, is a cocycle in $Z^1(F, A[2])$ representing the class $\mathfrak{c}_\lambda$. 
\end{remark}

 \begin{defi}
 Let $q:A[2]\rightarrow \boldsymbol \mu_2$ be a quadratic refinement of the Weil pairing $(~,~)_{\lambda}$. Then we define the function 
 $F_q:G_F\rightarrow \boldsymbol \mu_2$ as the composition
 \[F_q:G_F\longrightarrow \textup{Sp}(A[2])\stackrel{f_q}{\longrightarrow}\mathbb{F}_2\cong \boldsymbol \mu_2,\]
 where the map $G_F\rightarrow \textup{Sp}(A[2])$ is the homomorphism coming from the action of $G_F$ on $A[2]$ and  $f_q:\textup{Sp}(A[2])\rightarrow \mathbb{F}_2$ is the map afforded by \Cref{extension of dickson invariant}. 
 \end{defi}
 
 \begin{remark} \label{d of cochain remark}
 For each quadratic refinement $q$ of $(~,~)_\lambda$ it follows from \Cref{extension of dickson invariant} that we have $dF_q=c_q\cup c_q \in Z^2(F,\boldsymbol \mu_2)$.
 \end{remark}
 
 \begin{defi} \label{local terms controlling sha}
Let $\chi\in \textup{Hom}_{\textup{cnts}}(G_F,\boldsymbol \mu_2)$ be a quadratic character, let $q$ be a quadratic refinement of $\left( ~,~\right)_\lambda$ and let $c_q$ be the associated cocycle representing $\mathfrak{c}_\lambda$. As in the definition of the local choices for the pairing $\left \langle ~,~\right \rangle_{S_\chi}$, pick $P_{q}\in A(\bar{F})$ with $dP_{q}=c_q$, let $Q_{q}\in A(\bar{F})$ be such that $2Q_{q}=P_{q}$ and set $\rho_{q}=dQ_{q}$. Similarly, pick  $P_{\chi,q}\in A(\bar{F})$ with $d_\chi P_{\chi,q}=c_q$, let $Q_{\chi,q}\in A(\bar{F})$ be such that $2Q_{\chi,q}=P_{\chi,q}$ and set $\rho_{\chi,q}=d_\chi Q_{\chi,q}$. 

We then define $\mathfrak{g}(A,\lambda,\chi,q)$ to be the class of the cocycle \[g(A,\lambda,\chi,q):=(\rho_{q}+\rho_{\chi,q})\cup c_{q}-\chi \cup F_{q}\] in $\textup{Br}(F)[2]$. 
 As in \Cref{new pairing},  $\mathfrak{g}(A,\lambda,\chi,q)$ does not depend on the choices of $P_{q}, Q_{q}, P_{\chi,q}$ or $Q_{\chi,q}$. 
 \end{defi}
 
 The following lemma is key to the proof of \Cref{local sha twist decomp}.
 
\begin{lemma} \label{independence of refinement}
The quantity $\mathfrak{g}(A,\lambda,\chi,q)\in \textup{Br}(F)[2]$ is independent of the choice of quadratic refinement $q$. 
\end{lemma}

\begin{proof}
Keep the notation of \Cref{local terms controlling sha} in what follows. Let $q$ and $q'$ be two quadratic refinements. Then $q-q'=\left( -,v\right)_\lambda$ for some $v\in A[2]$ and $c_{q'}=c_q+dv$. By \Cref{change of form} we have 
\[F_{q'}=F_q+c_q\cup v+v\cup c_q+v\cup dv.\]
Now fix choices for $P_{q}, Q_{q}, P_{\chi,q}$ and $Q_{\chi,q}$ as in \Cref{local terms controlling sha}. Then we may take $P_{q'}=P_{q}+v$ and $P_{\chi,q'}=P_{\chi,q}+v$. Pick $T\in A[4]$ with $2T=v$. Then we may take $Q_{q'}=Q_{q}+T$ and $Q_{\chi,q'}=Q_{\chi,q}+T$. Thus
\[\rho_{q'}+\rho_{\chi,q'}=\rho_{q}+\rho_{\chi,q}+dT+d_\chi T.\] 
An easy computation gives $dT+d_\chi T=dv+\chi \cup v$. 
Combining this with the expressions for $F_{q'}$ and $c_{q'}$ in terms of $F_{q}$ and $c_{q}$ respectively, we see that we have an equality of cocycles 
\[g(A,\lambda,\chi,q')=g(A,\lambda,\chi,q)+(\rho_{q}+\rho_{\chi,q})\cup dv+(dv+\chi \cup v)\cup(c_q+dv)-\chi \cup(c_q\cup v+v\cup c_q+v\cup dv)\]
 inside $Z^2(F,\boldsymbol \mu_2)$.

Now $c_q+dv\in C^1(F,A[2])$ is a cocycle whilst $dv\in C^1(F,A[2])$ is a coboundary. Thus the class of $dv\cup (c_q+dv)$ is trivial in $\text{Br}(F)[2]$. Using this observation, cancelling like terms in the previous expression, and passing to classes in the Brauer group, one has
\[\mathfrak{g}(A,\lambda,\chi,q')=\mathfrak{g}(A,\lambda,\chi,q)+[(\rho_{q}+\rho_{\chi,q})\cup dv-\chi \cup c_q\cup v]\]
(where here `$[~]$' denotes the operation of taking classes in the Brauer group).

Now, as remarked in the definition of the pairing $\left \langle ~,~\right \rangle _{S_\chi}$, we have $d(\rho_{q}+\rho_{\chi,q})=\chi \cup c_q$. Thus by standard properties of cup product on cochains (see \Cref{cup product of cochains}) we have
\[d\left((\rho_{q}+\rho_{\chi,q})\cup v\right)=(\rho_{q}+\rho_{\chi,q})\cup dv-\chi \cup c_q \cup v.\] 
In particular, the class of $(\rho_{q}+\rho_{\chi,q})\cup dv-\chi \cup c_q \cup v$ is trivial in $\textup{Br}(F)$, whence $\mathfrak{g}(A,\lambda,\chi,q')=\mathfrak{g}(A,\lambda,\chi,q)$
as desired.
\end{proof}

As a consequence of \Cref{independence of refinement}, we make the folllowing refinement of \Cref{local terms controlling sha}.

\begin{defi} \label{the local terms part 2}
Define $\mathfrak{g}(A,\lambda,\chi)\in \textup{Br}(F)[2]$ to be the quantity $\mathfrak{g}(A,\lambda,\chi,q)$ for any choice of quadratic refinement $q$ of $\left( ~,~\right)_\lambda$.
\end{defi}

The following proposition computes explicitly the terms $\mathfrak{g}(A,\lambda,\chi)$ in certain cases.

\begin{proposition} \label{local terms in sha computation}
Let $\mathfrak{g}(A,\lambda,\chi)\in \textup{Br}(F)[2]$ be as in \Cref{the local terms part 2}. 

\begin{itemize}
\item[(i)] We have $\mathfrak{g}(A,\lambda,\mathbbm{1})=0$ where $\mathbbm{1}$ is the trivial character of $F$.
\item[(ii)] Suppose that $q$ is a $G_F$-invariant quadratic refinement of the Weil pairing $\left( ~,~\right)_\lambda$ on $A[2]$ and let $\alpha:G_{F}\rightarrow \boldsymbol \mu_2$ be the quadratic character corresponding to the homomorphism 
\[G_{F}\longrightarrow O(q)/SO(q)\cong \mathbb{Z}/2\mathbb{Z}\cong \boldsymbol \mu_2\] 
coming from the action of $G_{F}$ on $A[2]$. Then 
\[\mathfrak{g}(A,\lambda,\chi)=\alpha\cup \chi\in \textup{Br}(F)[2].\]
\item[(iii)] Suppose that $F$ is nonarchimedean with odd residue characteristic and that $A$ has good reduction. 
Then we have
\[\textup{inv}_F~ \mathfrak{g}(A,\lambda,\chi)=\begin{cases}0~~&~~\chi~\textup{unramified} \\ \frac{1}{2}\dim_{\mathbb{F}_2}A(F)[2] \in \mathbb{Q}/\mathbb{Z}~~&~~\chi~\textup{ramified}. \end{cases}\]
\item[(iv)] Suppose that $F$ is archimedean. Then we have
\[\textup{inv}_F~ \mathfrak{g}(A,\lambda,\chi)=\begin{cases}0~~&~~F=\mathbb{C}~~\textup{or}~~\chi~~\textup{trivial,} \\ \frac{1}{2}\dim_{\mathbb{F}_2}A(F)[2] \in \mathbb{Q}/\mathbb{Z}~~&~~F=\mathbb{R}~~\textup{ and }\chi~\textup{non-trivial.}\end{cases}\]
\end{itemize}
\end{proposition}

\begin{proof}
(i): Clear.

(ii): If there is an $F$-rational quadratic refinement $q$ then $c_q$ is identically zero and it follows immediately from \Cref{independence of refinement} and the definition of $\mathfrak{g}(A,\lambda,\chi,q)$ that 
\[\mathfrak{g}(A,\lambda,\chi)=\mathfrak{g}(A,\lambda,\chi,q)=-\chi\cup F_{q}=\chi\cup \alpha\]
where for the last equality we use that the restriction of $F_{q}$ to elements of $O(q)$ agrees with the Dickson homomorphism $d_q$ (see \Cref{extension of dickson invariant}).

(iii): By \cite[Proposition 3.6 (d)]{MR2915483}, our assumptions on $F$ and the reduction of $A$ imply that there is a $G_F$-invariant quadratic refinement $q$ of the Weil pairing on $A[2]$. Let $\alpha$ be the associated quadratic character so that $\mathfrak{g}(A,\lambda,\chi)=\alpha \cup \chi$ by (ii). Now by definition, $\alpha$ factors through $\textup{Gal}(F(A[2])/F)$ and our assumptions on $F$ and $A$ mean that $F(A[2])/F$ is unramified. Consequently, $\alpha$ is unramified. In fact, let $\sigma$ denote the Frobenius element in $F(A[2])/F$. Then by \Cref{dimension proposition} we have
\[\alpha(\sigma)=(-1)^{\textup{dim}_{\mathbb{F}_2}A[2]^{\sigma}}=(-1)^{\dim_{\mathbb{F}_2}A(F)[2]}.\] 
In particular, we see that if $\dim_{\mathbb{F}_2}A(F)[2]$ is even then $\alpha$ is the trivial character, whilst if $\dim_{\mathbb{F}_2}A(F)[2]$ is odd then $\alpha$ is the unique non-trivial unramified quadratic character of $F$. Since $F$ is assumed to have odd residue characteristic, standard properties of the cup-product of two quadratic characters gives the result (we review these later in \Cref{local symbols subsection}: see, in particular, \Cref{properties of local symbol}).

(iv): The argument here is similar to that of (iii). First note that if $\chi$ is trivial then $\mathfrak{g}(A,\lambda,\chi)=0$ by $(i)$. In particular, the only case we have not already covered is when $F=\mathbb{R}$ and $\chi$ is the unique non-trivial quadratic character corresponding to the extension $\mathbb{C}/\mathbb{R}$. By \cite[Proposition 3.6 (d)]{MR2915483} there is an $\mathbb{R}$-rational quadratic refinement $q$ of the Weil pairing $(~,~)_\lambda$. Let $\alpha$ be the associated quadratic character and write $\sigma$ for the unique non-trivial element of $\textup{Gal}(\mathbb{C}/\mathbb{R})$. By \Cref{dimension proposition} we see that $\alpha$ is trivial if $\dim_{\mathbb{F}_2}A[2]^\sigma=\dim_{\mathbb{F}_2}A(\mathbb{R})[2]$ is even, and is the quadratic character corresponding to $\mathbb{C}/\mathbb{R}$ otherwise. The result now follows from $(ii)$.
\end{proof}

\begin{remark}
As in \Cref{properties of quadratic forms associated to line bundles}, if the polarisation $\lambda$ is of the form $\phi_\mathcal{L}$ for an $F$-rational symmetric line bundle $\mathcal{L}$ then there is an associated $G_F$-invariant quadratic refinement  of the Weil pairing on $A[2]$. Thus combined with  \Cref{local terms in sha computation} (ii) this gives a geometric condition for when the local terms $\mathfrak{g}(A,\lambda,\chi)$ may be evaluated.
\end{remark}

\begin{remark}
It is natural to ask if the terms $\mathfrak{g}(A,\lambda,\chi)$ are independent of the choice of principal polarisation $\lambda$. 
The above proposition shows that this is true when $\chi$ is trivial, when $A/F$ has good reduction and $F$ has odd residue characteristic, or when $F$ is archimedean. We have been unable to prove this in general however. 
\end{remark}

\subsection{Controlling the parity of $\dim_{\mathbb{F}_2}\Sha_{\mathrm{nd}}(A/K)[2]+\dim_{\mathbb{F}_2}\Sha_{\mathrm{nd}}(A^{\chi}/K)[2]$ via local contributions}

We return to the notation of Sections 5.1-5.7 so that, in particular, $K$ is a number field and $(A/K,\lambda)$ a principally polarised abelian variety.


 \begin{theorem}[=\Cref{local sha twist decomp}] \label{main sha decomp theorem}
Let $\chi$ be a quadratic character of $K$ and for each place $v$ of $K$ write $\chi_v$ for the restriction of $\chi$ to $G_{K_v}$,  $A/K_v$ for the base change of  $A$ to $K_v$, and $\lambda_v$ for the principal polarisation on $A/K_v$ corresponding to $\lambda$. 

Then $\textup{dim}_{\mathbb{F}_2}\Sha_\textup{nd}(A/K)[2]+\textup{dim}_{\mathbb{F}_2}\Sha_\textup{nd}(A^\chi/K)[2]\equiv 0 ~~\textup{ (mod 2)}$ if and only if 
\[\sum_{v\in M_K} \textup{inv}_v~\mathfrak{g}(A/K_v,\lambda_v,\chi_v) =0\in \mathbb{Q}/\mathbb{Z}.\]
\end{theorem}

\begin{remark}
Before proving \Cref{main sha decomp theorem} we remark that if $v$ is a nonarchimedean place of $K$, not dividing $2$ and such that both $A$ has good reduction and $\chi$ is unramified at $v$, then $\mathfrak{g}(A/K_v,\lambda_v,\chi_v)=0$ by \Cref{local terms in sha computation} (iii). In  particular, the sum in the statement of \Cref{main sha decomp theorem} is finite.
\end{remark}

\begin{proof}[Proof of \Cref{main sha decomp theorem}]
By \Cref{twisting the torsor}, \Cref{poonen stoll} applied to both $A$ and $A^\chi$ (along with their principal polarisations $\lambda$ and $\lambda_\chi$), and \Cref{identification of pairings}, we see that $\textup{dim}_{\mathbb{F}_2}\Sha_\textup{nd}(A/K)[2]+\textup{dim}_{\mathbb{F}_2}\Sha_\textup{nd}(A_\chi/K)[2]$ is even if and only if $\left \langle \mathfrak{c}_\lambda,\mathfrak{c}_\lambda\right \rangle_{\mathcal{S}_\chi}=0$.

We now follow \Cref{definition of our altered pairing} to compute $\left \langle \mathfrak{c}_\lambda,\mathfrak{c}_\lambda\right \rangle_{\mathcal{S}_\chi}$. For the global choices, fix a quadratic refinement $q$ of the Weil pairing $(~,~)_\lambda$ on $A[2]$. Then as in the local case (\Cref{representative of c remark}) the function $c_q:G_K\rightarrow A[2]$ sending $\sigma \in G_K$ to the unique element $c_q(\sigma)\in A[2]$ such that 
\[ q(\sigma^{-1}v)-q(v)=(v,c_q(\sigma))_\lambda\] for all $v\in A[2]$, is a cocycle in $Z^1(F, A[2])$ representing the class $\mathfrak{c}_\lambda$. Similarly, the function  $F_q:G_K\rightarrow \boldsymbol \mu_2$ defined as the composition
 \[F_q:G_K\longrightarrow \textup{Sp}(A[2])\stackrel{f_q}{\longrightarrow}\mathbb{F}_2\cong \boldsymbol \mu_2,\]
(where the map $G_K\rightarrow \textup{Sp}(A[2])$ is the homomorphism coming from the action of $G_K$ on $A[2]$  and  $f_q:\textup{Sp}(A[2])\rightarrow \mathbb{F}_2$ is the map afforded by \Cref{extension of dickson invariant}) is an element of $C^1(K,\boldsymbol \mu_2)$ satisfying $dF_q=c_q\cup c_q \in Z^2(K,\bar{K}^{\times})$.

With these global choices in place, the local terms arising in the definition of $\left \langle \mathfrak{c}_\lambda,\mathfrak{c}_\lambda\right \rangle_{\mathcal{S}_\chi}$ are precisely the terms $\mathfrak{g}(A/K_v,\lambda_v,\chi_v,q)$ of \Cref{local terms controlling sha}. By \Cref{independence of refinement} (for fixed $v$) they are independent of $q$, their common value being by definition $\mathfrak{g}(A/K_v,\lambda_v,\chi_v)$.

Thus  
\[\left \langle \mathfrak{c}_\lambda,\mathfrak{c}_\lambda\right \rangle_{\mathcal{S}_\chi}=\sum_{v\in M_K}\textup{inv}_v~\mathfrak{g}(A/K_v,\lambda_v,\chi_v)\]
and the result follows.
\end{proof}

%




\section{Disparity in Selmer ranks: definitions and recollections}

In the next four sections we prove \Cref{main twisting theorem combined intro} concerning the parity of certain Selmer groups defined in terms of abstract twisting data. Our aproach follows closely the strategy of \cite{MR3043582}, which proves the result for Galois modules of dimension $2$ (whilst we handle arbitrary (even) dimension). Many of the statements of op. cit. go through with some minor changes however in order to highlight the diferences it is necessary to recall much of their setup and basic results. Thus in this section we recall the setup of \cite{MR3043582}. Where notions need to be generalised or slightly adapted we state the differences in a remark immediately following the definition.

\subsection{Notation} Here we fix some notation which will remain in place for the entirety of Sections 6-9. Fix first a prime $p$  and number field $K$. Following \cite{MR3043582}, for a field $L$ (either $K$ or $K_v$ for some $v\in M_K$) we define $\mathcal{C}(L):=\text{Hom}_{\textup{cnts}}(G_L,\boldsymbol\mu_p)$, the group of characters of order dividing $p$. We denote the trivial character by $\mathbbm{1}_L$. Further, we define $\mathcal{F}(L)$ to be the quotient of $\mathcal{C}(L)$ by the action of $\text{Aut}(\boldsymbol\mu_p)$ (the action given by post-composition). The set $\mathcal{F}(L)$ is naturally identified with the set of cyclic extensions of $L$ of degree dividing $p$, the map being given by sending the equivalence class of $\chi \in \mathcal{C}(L)$ to the fixed field $\bar{K}^{\text{ker}(\chi)}$. When $L$ is a local field we write $\mathcal{C}_\text{ram}(L)$ (resp. $\mathcal{C}_{\text{ur}}(L)$) for the subset of $\mathcal{C}(L)$ consisting of ramified (resp. unramified) characters, and similarly write $\mathcal{F}_{\text{ram}}(L)$ (resp. $\mathcal{F}_{\text{ur}}(L)$) for the subset of $\mathcal{F}(L)$ corresponding to ramified (resp. unramified) extensions. Note that if $L$ has residue characteristic coprime to $p$ then $\mathcal{C}_{\text{ram}}(L)$ (and hence also $\mathcal{F}_{\text{ram}}(L)$) is non-empty if and only if $\boldsymbol \mu_p \subseteq L$.

For an finite dimensional $\mathbb{F}_p$-vector space $M$ we say that a map $q:M\rightarrow \mathbb{Q}/\mathbb{Z}$ is a \emph{quadratic form} if $q(nx)=n^2q(x)$ for all $n\in \mathbb{Z}$ and $x\in M$, and if the map $(x,y)\mapsto q(x+y)-q(x)-q(y)$ is a symmetric bilinear pairing on $M$. We say that $q$ is \emph{non-degenerate} if the associated pairing is (i.e. if it has trivial kernel).  If $q$ is a quadratic form on $M$ and $\left \langle ~,~\right \rangle$ the associated pairing then for a subgroup $W$ of $M$ we write
\[W^\perp=\{m\in M~~:~~\left \langle w,m\right \rangle =0~~\forall w\in W\}\]
for the orthogonal complement of $W$ and say that $W$ is a \emph{Lagrangian}  subspace of $(M,q)$ if $W=W^\perp$ and $q(W)=0$. We call $(M,q)$ a \emph{metabolic space} if $q$ is non-degenerate and if $M$ has a Lagrangian subspace.  

\subsection{The module $T$ and the finite set of places $\Sigma$} \label{the module T}

Fix, once and for all, a finite dimensional $\mathbb{F}_p$-vector space $T$  equipped with a continuous $G_K$-action and a non-degenerate $G_K$-equivariant alternating pairing 
\[(~,~):T\times T \longrightarrow \boldsymbol \mu_p\]
(so that, in particular, $\dim_{\mathbb{F}_p}T$ is necessarily even).
For $v\in M_K$, if the inertia subgroup of $G_{K_v}$ acts trivially on $T$ then we say that $T$ is \textit{unramified} at $v$, and \textit{ramified} at $v$ otherwise. 
We denote by $K(T)$ the field of definition of the elements of $T$, i.e. the fixed field of the kernel of the action of $G_K$ on $T$.

We also fix  a finite set $\Sigma$ of places of $K$ containing all archimedean places, all places over $p$ and all places where $T$ is ramified (and possibly some more to be specified later). 

\subsection{The Local Tate Pairing and Tate quadratic forms} \label{local tate pairing subsection}

For each place $v\in M_K$, we write $\left \langle ~,~\right \rangle_v$ for the local Tate pairing 
\[H^1(K_v,T)\times H^1(K_v,T)\longrightarrow \mathbb{Q}/\mathbb{Z}\]
given by the composition
\[H^1(K_v,T)\times H^1(K_v,T)\stackrel{\cup}{\longrightarrow} H^2(K_v,\boldsymbol \mu_p)\stackrel{\text{inv}_v}{\longrightarrow}\mathbb{Q}/\mathbb{Z},\]
where the first map is induced by cup-product and the pairing $(~,~)$. It is non-degenerate, bilinear and symmetric. 

\begin{defi} \label{tate form defi}
Let $v$ be a place of $K$. We say a quadratic form $q_v:H^1(K_v,T)\rightarrow \mathbb{Q}/\mathbb{Z}$ is a \emph{Tate quadratic form} if its associated bilinear form is the local Tate pairing $\left \langle ~,~\right \rangle_v$. If $v\notin \Sigma$, then we say that $q_v$ is \emph{unramified} if it vanishes on $H^1_{\text{ur}}(K_v,T)$ (in which case $H^1_{\text{ur}}(K_v,T)$ is a Lagrangian subspace for $q_v$). 
\end{defi}

\begin{remark}
If $p=2$ then our definition differs slightly from \cite[Definition 3.2]{MR3043582} of Klagsbrun--Mazur--Rubin since it allows quadratic forms valued in $\frac{1}{4}\mathbb{Z}/\mathbb{Z}$ whilst their definition only allows quadratic forms taking values in $\frac{1}{2}\mathbb{Z}/\mathbb{Z}$. This extra generality is necessary when $\dim_{\mathbb{F}_p}T>2$ in order to allow $T=A[2]$ for a principally polarised abelian variety $A/K$ (see \Cref{valued in Z/4Z}). 
\end{remark}

As in \cite[Lemma 3.4]{MR3043582}, if $p>2$ then there is a unique Tate quadratic form $q_v$ on $H^1(K_v,T)$ given by
\[q_v=\frac{1}{2}\left\langle ~,~\right \rangle_v.\]

\subsection{Global metabolic structures} \label{global metabolic structures} 
With our slightly modified definition of a Tate quadratic form in hand we can define a global metabolic structure on $T$ in an identical way to \cite[Definition 3.3]{MR3043582}.

\begin{defi} \label{global metabolic structure defi}
A \textit{global metabolic structure} $\textbf{q}$ on $T$ consists of a collection $\textbf{q}=(q_v)_v$ ($v\in M_K$) of Tate quadratic forms such that 
\begin{itemize}
\item[(i)] for each $v\in M_K$ the pair $\left(H^1(K_v,T),q_v\right)$ is a metabolic space, 
\item[(ii)] the quadratic form $q_v$ is unramified at each place $v\notin \Sigma$,
\item[(iii)] if $c\in H^1(K,T)$ then $\sum_v q_v(c_v)=0$.
\end{itemize}
\end{defi}

As in \cite[Lemma 3.4]{MR3043582}, if $p>2$ then the unique Tate quadratic forms on $H^1(K_v,T)$ defined above do indeed give a global metabolic structure on $T$, so specifying a global metabolic structure is only necessary when $p=2$. 

\subsection{Selmer structures and Selmer groups}

We define  Selmer structures for $(T,\textbf{q})$, along with the associated Selmer groups as in \cite[Definition 3.8]{MR3043582}. 

\begin{defi}
A Selmer structure $\mathcal{S}$ for $(T,\textbf{q})$ is the data of 
\begin{itemize}
\item[(i)] a finite set $\Sigma_{\mathcal{S}}$ of places of $K$ containing $\Sigma$,
\item[(ii)] for each $v\in \Sigma_{\mathcal{S}}$, a Lagrangian subspace $H_{\mathcal{S}}(K_v,T)$ of $\left(H^1(K_v,T),q_v\right)$.  
\end{itemize}
\end{defi}

\begin{defi}
Let $\mathcal{S}$ be a Selmer structure for $(T,\textbf{q})$. For each $v\notin \Sigma_{\mathcal{S}}$ we set $H^1_{\mathcal{S}}(K_v,T)=H^1_{\textup{ur}}(K_v,T)$ and define the \emph{Selmer group} associated to $\mathcal{S}$ as
\[H^1_{\mathcal{S}}(K,T):=\ker\left(H^1(K,T)\longrightarrow \bigoplus_{v\in M_K} H^1(K_v,T)/H^1_{\mathcal{S}}(K_v,T)\right).\]
\end{defi}

The following theorem, which is a very slight generalisation of \cite[Theorem 3.9]{MR3043582} allows us to compare the dimensions of two Selmer groups modulo 2. 

\begin{theorem} \label{comparison of Selmer structures}
Let $\mathcal{S}$ and $\mathcal{S}'$ be two Selmer structures for $(T,\textbf{q})$. Then
\[\dim_{\mathbb{F}_p}H^1_{\mathcal{S}}(K,T)-\dim_{\mathbb{F}_p}H^1_{\mathcal{S}'}(K,T)\equiv \]
\[~~~~~~~~~~~~~~~~~~~~~~~~~~~~~~~~~\sum_{{\Sigma}_\mathcal{S}\cup{\Sigma}_{\mathcal{S}'}}\dim_{\mathbb{F}_p}H^1_{\mathcal{S}}(K_v,T)/(H^1_{\mathcal{S}}(K_v,T)\cap H^1_{\mathcal{S}'}(K_v,T))~~\textup{ (mod 2)}.\]
\end{theorem}

\begin{proof}
This is proven for $\text{dim}_{\mathbb{F}_p}T=2$ in \cite[Theorem 3.9]{MR3043582} and the proof generalises verbatim to the case where $T$ has arbitrary (even) dimension with one subtlety: their proof relies on \cite[Proposition 2.4]{MR3043582} which is a general result concerning the dimension of the intersection of Lagrangian subspaces of a finite dimensional metabolic space. The one difference from the case there is that now our quadratic forms (in general) take values in $\mathbb{Q}/\mathbb{Z}$ rather than just $\mathbb{F}_p$ as they assume. However, one readily verifies that this assumption is not used in the proof of the cited result. Alternatively, see \cite[Theorem 5.9]{KES14} which gives a further generalisation of \cite[Theorem 3.9]{MR3043582} which includes our case.
\end{proof}

\subsection{Twisting data and twisted Selmer groups} 

Fix from now on a global metabolic structure $\textbf{q}$ on $T$. 

\begin{defi} For each place $v\in M_K$, write $\mathcal{H}(q_v)$ for the set of Lagrangian subspaces for $q_v$ and, for $v\notin \Sigma$, write $\mathcal{H}_{\textup{ram}}(q_v)$ for the subset of $\mathcal{H}(q_v)$ consisting of Lagrangian subspaces $X$ for which $X\cap H^1_{\textup{ur}}(K_v,T)=0$. 
\end{defi}

\begin{defi}[Twisting data]  \label{twisting data}
We define \textit{twisting data} $\boldsymbol \alpha$ for $(T,\textbf{q},\Sigma)$ to consist of
\begin{itemize}
\item[(i)] for each $v\in \Sigma$, a map
\[\alpha_v:\mathcal{F}(K_v)\longrightarrow \mathcal{H}(q_v),\]
\item[(ii)] for each $v\notin \Sigma$ for which $\boldsymbol \mu_p \subseteq K_v$, a map
\[\alpha_v:\mathcal{F}_\textup{ram}(K_v)\longrightarrow \mathcal{H}_{\textup{ram}}(q_v).\]
\end{itemize}
\end{defi}

\begin{remark}
Our definition of twisting data is slightly different to that of  \cite[Definition 4.4]{MR3043582}. In their case, since $T$ has dimension $2$, for $v\notin \Sigma$ and with $\boldsymbol \mu_p \subseteq K_v$,  $\mathcal{H}_\textup{ram}$ has cardinality $0$,$1$, or $p$ according to $\textup{dim}T^{G_{K_v}}=0,1$ or $2$ respectively. In the first two cases they do not  specify a map $\alpha_v$ as there is a unique such. In the final case, they additionally insist that $\alpha_v$ is a bijection, as is possible since $\mathcal{F}_\textup{ram}(K_v)$ has order $p$. 

Since $T$ is allowed to have dimension greater that $2$ we in general have $|\mathcal{H}_\textup{ram}|>p$ and so we cannot insist that $\alpha_v$ is a bijection once the map ceases to be unique. Although omitting this condition does not impact what follows, and is in fact not used in the main results of \cite{MR3043582}, we remark that it is used crucially in a follow up paper to op. cit. (\cite{Klagsbrun_Mazur_Rubin_2014}).
\end{remark}

\begin{defi}[Twisted Selmer groups] \label{twisted selmer groups defi}
Let ($T$, $\textbf{q}$, $\Sigma$, $\boldsymbol \alpha$) as above be fixed, and let $\chi\in \mathcal{C}(K)$. Let $\P_\chi$ denote the set of primes of $K$ for which $\chi$ ramifies. Then we define a Selmer structure $\mathcal{S}(\chi)$ by taking $\Sigma_{\mathcal{S}(\chi)}$ to be $\Sigma \cup \P_{\chi}$ and setting $H^1_{\mathcal{S}(\chi)}(K_v,T):=\alpha_v(\chi_v) $ for $ v\in \Sigma \cup \P_\chi$. We write $\textup{Sel}(T,\chi)$ for the associated Selmer group: 
\[\textup{Sel}(T,\chi):=H^1_{\mathcal{S}(\chi)}(K,T).\]
\end{defi}

\subsection{Comparing the parity of dimensions of twisted Selmer groups}

From now on we fix $T$, the set of places $\Sigma$, a global metabolic structure $\textbf{q}$, and twisting data $\boldsymbol \alpha$.

The following theorem, which is a slight variant of \cite[Theorem 4.11]{MR3043582} allows us to compare the parity of the dimensions of the Selmer groups $\text{Sel}(T,\chi)$ as we vary $\chi$. We first make one further definition. 

\begin{defi} \label{notation for intersection of twisted subspaces}
Let $v$ be a place of $K$ and $\chi_1$ and $\chi_2$ be elements of $\mathcal{C}(K_v)$. Then we set
\[h_v(\chi_1,\chi_2):=\dim_{\mathbb{F}_p}\left(\alpha_v(\chi_1)/(\alpha_v(\chi_1)\cap \alpha_v(\chi_2))\right).\]
Note that since any two Lagrangian subspaces of $H^1(K_v,T)$ have the same dimension this is symmetric in $\chi_1$ and $\chi_2$.
\end{defi}

\begin{theorem} \label{local decomposition}
For any $\chi \in \mathcal{C}(K)$ we have
\[\dim_{\mathbb{F}_p}\textup{Sel}(T,\chi)-\dim_{\mathbb{F}_p}\textup{Sel}(T,\mathbbm{1}_K)\equiv \sum_{v\in \Sigma}h_v(\mathbbm{1}_{K_v},\chi_v)+\sum_{v\notin \Sigma,\chi_v~\textup{ram}}\dim_{\mathbb{F}_p}T^{G_{K_v}}~~~\textup{ (mod 2)}.\]
\end{theorem}

\begin{proof}
This is essentially \cite[Theorem 4.11]{MR3043582}. Let $\mathcal{S}(\chi)$ and $\mathcal{S}(\mathbbm{1}_K)$ be the Selmer structures associated to the characters $\chi$ and $\mathbbm{1}_K$ respectively. Then 
\[\Sigma_{\mathcal{S}(\chi)}\cup \Sigma_{\mathcal{S}(\mathbbm{1}_K)}=\Sigma \sqcup\{v\notin \Sigma ~~:~~\chi_v ~~\text{ramified}\}.\]
Applying \Cref{comparison of Selmer structures} to $\mathcal{S}(\chi)$ and $\mathcal{S}(\mathbbm{1}_K)$ and noting that, by the definition of the twisting data, $H^1_{\text{ur}}(K_v,T)\cap \alpha_v(\chi_v)=0$ for all $v\notin \Sigma$ for which $\chi_v$ is ramified, we obtain
\[\dim_{\mathbb{F}_p}\textup{Sel}(T,\chi)-\dim_{\mathbb{F}_p}\textup{Sel}(T,\mathbbm{1}_K)\equiv \sum_{v\in \Sigma}h_v(\mathbbm{1}_{K_v},\chi_v)+\sum_{v\notin \Sigma,\chi_v~\textup{ram}}\text{dim}_{\mathbb{F}_p} H^1_{\textup{ur}}(K_v,T)~~~\textup{ (mod 2)}.\]
The result now follows since for each $v\notin \Sigma$ we have $\text{dim}_{\mathbb{F}_p} H^1_{\textup{ur}}(K_v,T)=\dim_{\mathbb{F}_p}T^{G_{K_v}}$. This is shown in (the proof of) \cite[Lemma 3.7]{MR3043582} in the case that $T$ has dimension $2$. The general case is identical. 
\end{proof}

\section{Disparity in Selmer ranks: statement and first cases}

In this section we fix $(T,\Sigma,\textbf{q},\boldsymbol \alpha)$ as in the previous section and consider the proportion of characters $\chi$ for which the associated Selmer groups $\text{Sel}(T,\chi)$ have odd (resp. even) $\mathbb{F}_p$-dimension. To make this precise, one has to order the elements of $\mathcal{C}(K)$.

\subsection{Ordering twists}

We use the same ordering as in \cite[Definition 7.3]{MR3043582}.  

\begin{defi} 
For $\chi \in \mathcal{C}(K)$, set
\[||\chi||=\textup{max}\{N(\mathfrak{p})~~:~~\chi ~~\text{is ramified at}~\mathfrak{p}\}\]
(where here for a prime $\mathfrak{p}\vartriangleleft \mathcal{O}_K$, $N(\mathfrak{p})$ denotes the norm of $\mathfrak{p}$). If this set is empty, our convention is that $||\chi||=1$.

Then for each $X>0$ define
\[\mathcal{C}(K,X)=\{\chi \in \mathcal{C}(K)~~:~~||\chi||<X\}.\]
\end{defi}

For each $X>1$, $\mathcal{C}(K,X)$ is a finite subgroup of $\mathcal{C}(K)$ and each element of $\mathcal{C}(K)$ appears in $\mathcal{C}(K,X)$ for some $X$. We will make crucial use of the group structure on the $\mathcal{C}(K,X)$ to facilitate with counting problems. 

We will repeatedly use the following fact.

\begin{lemma} \label{surjectivity of restriction}
For all sufficiently large $X>0$ the restriction homomorphism
\[\mathcal{C}(K,X)\longrightarrow \prod_{v\in \Sigma}\mathcal{C}(K_v)\]
sending $\chi$ to $(\chi_v)_{v\in \Sigma}$ is surjective.  
\end{lemma}

\begin{proof}
This follows immediately from the Grunwald--Wang theorem. See for example \cite[Theorem 9.2.3(ii)]{MR2392026}. See also \cite[Proposition 6.8 (i)]{MR3043582} but note that they have a running hypothesis on the set of places $\Sigma$ which we do not wish to impose at this stage.
\end{proof}

\subsection{Statement of the result}

The proportion of characters for which $\dim_{\mathbb{F}_p}\textup{Sel}(T,\chi)$ is even (resp. odd) will depend heavily on the action of $G_K$ on $T$. More specifically, it will depend on the behaviour of the following function. Recall that $K(T)$ denotes the field of definition of the elements of $T$.

\begin{defi}
Write $G:=\textup{Gal}(K(T)/K)$ and define the function 
$\epsilon:G\rightarrow \{\pm 1\}$
by $\sigma \mapsto (-1)^{\dim_{\mathbb{F}_p}T^{\sigma}}$.
\end{defi}

The result is then the following. 

\begin{theorem}[=\Cref{main twisting theorem combined intro}] \label{main twisting theorem combined}
We have:
\begin{itemize}
\item[(i)] if either $p=2$ and $\epsilon$ fails to be a homomorphism, or $p>2$ and $\epsilon$ is non-trivial when restricted to $\textup{Gal}(K(T)/K(\boldsymbol \mu_p))$, then  
\[\lim_{X\rightarrow \infty}\frac{|\{\chi \in \mathcal{C}(K,X)~:~\textup{dim}_{\mathbb{F}_p}\textup{Sel}(T,\chi)~~\textup{is even}\}|}{|\mathcal{C}(K,X)|}=1/2.\]
Moreover, if $p=2$ then it suffices to take $X$ sufficiently large as opposed to taking the limit $X\rightarrow \infty$. 
\item[(ii)] if either $p=2$ and $\epsilon$ is a homomorphism, or $p>2$ and $\epsilon$ is trivial when restricted to $\textup{Gal}(K(T)/K(\boldsymbol \mu_p))$, then for all sufficiently large $X$ we have  \[\frac{|\{\chi \in \mathcal{C}(K,X)~:~\textup{dim}_{\mathbb{F}_p}\textup{Sel}(T,\chi)~~\textup{is even}\}|}{|\mathcal{C}(K,X)|}=\frac{1+(-1)^{\textup{dim}_{\mathbb{F}_p}\textup{Sel}(T,\mathbbm{1}_K)}\cdot\delta}{2}\] with $\delta=\prod_{v\in \Sigma}\delta_v$  given in \Cref{definition of delta}.
\end{itemize}
\end{theorem}

The proof of \Cref{main twisting theorem combined}, which  is a combination of \Cref{epsilon homomorphism thm,epsilon non-trivial twisting theorem}, will occupy the remainder of Sections 7-9.

\begin{remark} \label{remark on the function epsilon}
Here we briefly discuss the function $\epsilon$. For convenience, we identify $\boldsymbol \mu_p$ with the additive group of $\mathbb{F}_p$ and think of the pairing $(~,~)$ as landing in  $\mathbb{F}_p$. Due to this pairing, the group $G=\textup{Gal}(K(T)/K)$ is a subgroup of the general symplectic group \[\textup{GSp}(T)=\{g\in \textup{GL}(V)~~:~~\forall v,w\in T, ~(gv,gw)=\lambda(g)(v,w)~\textup{for some}~\lambda(g)\in \mathbb{F}_p^{\times}\}.\] 

First suppose $p=2$, so that $\textup{GSp}(T)=\textup{Sp}(T)$ is the symplectic group associated to $(~,~)$. If $\textup{dim}_{\mathbb{F}_2}T>4$ then  $\textup{Sp}(T)$ is simple and since any symplectic transvection $\sigma$ (i.e. element of $\textup{Sp}(T)$ of the form $v\mapsto v+(v,w)w$
for fixed $0\neq w\in T$) has $\textup{dim}_{\mathbb{F}_2}T^{\sigma}$ odd, if $G$ is isomorphic to $\textup{Sp}(T)$ (i.e. is as large as possible) then $\epsilon$ is not a homomorphism. Thus Case (i) of \Cref{main twisting theorem combined} is, in some sense, the `generic' case. When $\dim_{\mathbb{F}_2}T=2$ one can check that $\epsilon$ is always a homomorphism, whilst if $\textup{dim}_{\mathbb{F}_2}T=4$ then $\textup{Sp}(T)$ is isomorphic to the symmetric group $S_6$. One can check  (see \Cref{hyperelliptic curves example} later) that when $G$ is either the whole of $S_6$ or the alternating group $A_6$ then $\epsilon$ is not a homomorphism, so again Case (i) of  \Cref{main twisting theorem combined} holds for $G$ `large enough'. On the other hand, \Cref{dimension proposition} gives a supply of examples where $\epsilon$ is a homomorphism. Namely, if $G$ fixes a quadratic refinement $q$ of $(~,~)$ then $G$ is a subgroup of the orthogonal group $O(q)$, in which case $\epsilon$ is the Dickson homomorphism. 

Now suppose that $p>2$. The subgroup $\textup{Gal}(K(T)/K(\boldsymbol \mu_p))$ consists of those elements $g\in G$ for which $\lambda(g)=1$. That is, it is the intersection of $G$ with the symplectic group $\textup{Sp}(T)$. If $G$ contains a symplectic transvection $\sigma$ (which as now $p>2$ is an element of $\textup{Sp}(T)$ of the form
$v\mapsto v+\beta\cdot (v,w)w$ for $\beta \in \mathbb{F}_p^{\times},0\neq w\in T$) then one sees easily that $\epsilon(\sigma)=-1$, so that $\epsilon$ is non-trivial when restricted to $\textup{Gal}(K(T)/K(\boldsymbol \mu_p))$. Thus again Case (i) of \Cref{main twisting theorem combined} holds for $G$ `large enough'.
\end{remark}

\subsection{The cases $p=2$ and $\epsilon$ a homomorphism, and $p>2$ and $\epsilon$ trivial when restricted to $\textup{Gal}(K(T)/K(\boldsymbol \mu_p))$}

Suppose now that either $p=2$ and $\epsilon$ is a homomorphism, or $p>2$ and $\epsilon$ is trivial for all $\sigma\in \text{Gal}(K(T)/K(\boldsymbol \mu_p))$. 

\begin{defi}
Let $v\in \Sigma$ and $\chi \in \mathcal{C}(K_v)$. If $p>2$ we define
\[\omega_v(\chi):=(-1)^{h_v(\mathbbm{1}_{K_v},\chi_v)}.\]

If $p=2$ view $\epsilon$ as a quadratic character of $K$ and let $\Delta\in K^{\times}/K^{\times 2}$ be such that the corresponding quadratic extension is given by $K(\sqrt{\Delta})/K$. We then define
\[\omega_v(\chi):=\chi(\Delta)(-1)^{h_v(\mathbbm{1}_{K_v},\chi_v)}\]
where here for a place $v$ of $K$  we evaluate $\chi_v$ at $\Delta$ via local class field theory.  
\end{defi}

\begin{lemma}  \label{first parity lemma}
For any $\chi\in \mathcal{C}(K)$ we have
\[(-1)^{\textup{dim}_{\mathbb{F}_p}\textup{Sel}(T,\chi)}=(-1)^{\textup{dim}_{\mathbb{F}_p}\textup{Sel}(T,\mathbbm{1}_K)}\prod_{v\in \Sigma}\omega_v(\chi_v).\]
\end{lemma}

\begin{proof}
 Fix $v\notin \Sigma$ with $\boldsymbol \mu_p \subseteq K_v$, and let $\text{Frob}_v\in G$ denote the Frobenius element at $v$ in $K(T)/K$. Then as $T$ is unramified at $v$ we have 
\[(-1)^{\text{dim}_{\mathbb{F}_p}T^{G_{K_v}}}=\epsilon(\text{Frob}_v).\]

If $p>2$ then $\epsilon(\text{Frob}_v)=1$ for all $v\notin \Sigma$ by assumption, whence the result follows from \Cref{local decomposition}.

Now suppose that $p=2$. As above, view $\epsilon$ as a quadratic character of $K$. Since $\epsilon$ factors through $\text{Gal}(K(T)/K)$ it is unramified outside $\Sigma$. In particular, if $v\notin \Sigma$ is such that $\chi_v$ is unramified, then both $\epsilon_v$ and $\chi_v$ are unramified at $v$ and so $\chi_v(\Delta)=1$. On the other hand, if $v\notin \Sigma$ is such that $\chi_v$ ramifies at $\Sigma$ then since $K_v$ has odd residue characteristic, we have $\chi_v(\Delta)=\epsilon(\text{Frob}_v)$ (see \Cref{properties of local symbol} (ii)). Global class field theory gives $\prod_{v\in M_K}\chi_v(\Delta)=1$ from which it follows that
\[\prod_{v\notin \Sigma,~ \chi_v\text{ ram}}(-1)^{\dim_{\mathbb{F}_p}T^{G_{K_v}}}=\prod_{v\in \Sigma}\chi_v(\Delta).\]
We now conclude by \Cref{local decomposition}. 
\end{proof}

The argument now proceeds as in \cite[Section 7]{MR3043582}. 

\begin{defi} \label{definition of delta}
For each $v\in \Sigma$, define 
\[\delta_v:=\frac{1}{|\mathcal{C}(K_v)|}\sum_{\chi \in \mathcal{C}(K_v)}\omega(\chi) \phantom{hello} \textup{and} \phantom{hello} \delta:=\prod_{v\in \Sigma}\delta_v.\]
\end{defi}

\begin{remark}
We have decided to define $\delta$ slightly differently to  \cite[Section 7]{MR3043582} so that it is a product of local terms. Our definition of the $\delta_v$ is consistent with theirs however.
\end{remark}

\begin{theorem} \label{epsilon homomorphism thm}
Suppose that either $p=2$ and $\epsilon$ is a homomorphism, or $p>2$ and $\epsilon$ is trivial when restricted to $\textup{Gal}(K(T)/K(\boldsymbol \mu_p))$. Then for all sufficiently large $X>0$ we have
\[\frac{|\{\chi \in \mathcal{C}(K,X)~~:~~\dim_{\mathbb{F}_p}\textup{Sel}(T,\chi)~~\textup{is even}\}|}{|\mathcal{C}(K,X)|}=\frac{1+(-1)^{\textup{dim}_{\mathbb{F}_p}\textup{Sel}(T,\mathbbm{1}_K)}\delta}{2}.\]
\end{theorem}

\begin{proof}
The argument is the same as in \cite[Theorem 7.6]{MR3043582}. We repeat it for convenience. Write $\Gamma=\prod_{v\in \Sigma}\mathcal{C}(K_v)$ and for $\chi \in \mathcal{C}(K)$, write $\chi|_\Gamma$ for the image of $\chi$ under the natural restriction homomorphism $\mathcal{C}(K)\rightarrow \Gamma$ sending $\chi$ to $(\chi_v)_{v\in \Sigma}$. From \Cref{first parity lemma} we see that the parity of $\text{dim}_{\mathbb{F}_p}\text{Sel}(T,\chi)$ depends only on $\chi|_\Gamma$ and that
$\text{dim}_{\mathbb{F}_p}\text{Sel}(T,\chi)$ is even if and only if \[\prod_{v\in \Sigma}\omega(\chi_v)=(-1)^{\text{dim}_{\mathbb{F}_p}\text{Sel}(T,\mathbbm{1}_K)}.\]

Now fix $X$ sufficiently large so that $\mathcal{C}(K,X)$ surjects onto $\Gamma$ under restriction. Since restriction is a group homomorphism, its fibres all have the same size (being cosets of the kernel) and, in particular, we have 
\[\frac{|\{\chi \in \mathcal{C}(K,X)~~:~~\dim_{\mathbb{F}_p}\textup{Sel}(T,\chi)~~\textup{is even}\}|}{|\mathcal{C}(K,X)|}=\frac{|\{\gamma \in \Gamma ~~:~~\prod_{v\in \Sigma} \omega(\gamma_v)=(-1)^{\text{dim}_{\mathbb{F}_p}\text{Sel}(T,\mathbbm{1}_K)}\}|}{|\Gamma|}\]
where here, for $\gamma\in \Gamma$ we denote by $\gamma_v$ its projection onto $\mathcal{C}(K_v)$.

To evaluate the right hand side of the above expression, define
\[N:=|\{\gamma \in \Gamma~~:~~\prod_{v\in \Sigma}\omega(\gamma_v)=1\}|.\]
Then we have
\[N-(|\Gamma|-N)=\sum_{\gamma \in \Gamma}\prod_{v\in \Sigma}\omega(\gamma_v)=\prod_{v\in \Sigma}\sum_{\chi_v\in \mathcal{C}(K_v)}\omega(\chi_v).\]
Dividing the above expression through by $2|\Gamma|$ gives
\[\frac{|\{\gamma \in \Gamma~~:~~\prod_{v\in \Sigma}\omega(\gamma_v)=1\}|}{|\Gamma|}=\frac{1+\delta}{2}\]
and the result follows immediately. 
\end{proof}

\section{Disparity in Selmer ranks: local symbols and global characters} \label{local global}

In order to prove the remaining cases of \Cref{main twisting theorem combined} we now recall and slightly generalise (as well as rephrase for convenience in Section 9) the results of \cite[Section 6]{MR3043582}, which uses class field theory to analyse which collections of local characters arise from a global character.  

\subsection{Local symbols} \label{local symbols subsection}

For each nonarchimedean place $v$ of $K$, Tate local duality gives a non-degenerate pairing
\begin{equation} \label{local hilbert symbol}
H^1(K_v,\boldsymbol \mu_p) \times H^1(K_v,\mathbb{Z}/p\mathbb{Z})\longrightarrow \mathbb{Q}/\mathbb{Z},
\end{equation}
defined as the composition
\[H^1(K_v,\boldsymbol \mu_p) \times H^1(K_v,\mathbb{Z}/p\mathbb{Z})\stackrel{\cup}{\longrightarrow}H^2(K,\boldsymbol \mu_p)\hookrightarrow \text{Br}(K_v)\stackrel{\textup{inv}_v}{\rightarrow} \mathbb{Q}/\mathbb{Z}\]
(here the map `$\cup$' is the cup product map on cohomology combined with the canonical isomorphism $\mathbb{Z}/p\mathbb{Z}\otimes \boldsymbol \mu_p \cong \boldsymbol \mu_p$).

We now slightly modify this  pairing. As the Galois action on $\mathbb{Z}/p\mathbb{Z}$ is trivial we have
$H^1(K_v,\mathbb{Z}/p\mathbb{Z})=\text{Hom}_{\text{cnts}}(G_{K_v},\mathbb{Z}/p\mathbb{Z})$. Picking an isomorphism of abstract groups $\theta:\boldsymbol \mu_p \stackrel{\sim}{\rightarrow} \mathbb{Z}/p\mathbb{Z}$ induces isommorphisms
\begin{equation} \label{iso1}
\mathcal{C}(K_v)\cong H^1(K_v,\mathbb{Z}/p\mathbb{Z})~~~\phantom{\text{ghost}}~~~\textup{and}~~~~\phantom{\text{ghost}}~~\frac{1}{p}\mathbb{Z}/\mathbb{Z}\cong \boldsymbol \mu_p
\end{equation}
where for the latter we identify $\mathbb{Z}/p\mathbb{Z}$ with $\frac{1}{p}\mathbb{Z}/\mathbb{Z}$ by sending $1\in \mathbb{Z}/p\mathbb{Z}$ to $\frac{1}{p}$. Noting that  $H^2(K_v,\boldsymbol \mu_p)\subseteq \text{Br}(K_v)$ is mapped by $\textup{inv}_v$ into $\frac{1}{p}\mathbb{Z}/\mathbb{Z}$, combining the pairing \cref{local hilbert symbol} with the isomorphisms of \cref{iso1} yields a non-degenerate pairing
\begin{equation} \label{altered pairing}
[~,~]_v:H^1(K_v,\boldsymbol \mu_p) \times \mathcal{C}(K_v)\longrightarrow \boldsymbol \mu_p
\end{equation}
which is easily seen to be independent of the choice of $\theta$.

The following well known lemma summarises the properties of this local pairing. 

\begin{lemma} \label{properties of local symbol} 
Let $v$ be a nonarchimedean place of $K$. Then 
\begin{itemize}
\item[(i)] if $v\nmid p$ then the groups $H^1_{\textup{ur}}(K_v,\boldsymbol \mu_p)$ and  $\mathcal{C}_{\textup{ur}}(K_v)$ are orthogonal complements with respect to the pairing $[~,~]_v$.
\item[(ii)] let $x\in K_v^{\times}$ and write $\phi_x\in H^1(K_v,\boldsymbol \mu_p)$ for the image of $x$ under the boundary map associated to the Kummer sequence 
\[
1\rightarrow \boldsymbol \mu_p \longrightarrow \bar{K}_v^{\times}\stackrel{x\mapsto x^p}{\longrightarrow}\bar{K}_v^\times\rightarrow 1.
\]
Then for any $\chi\in \mathcal{C}(K_v)$ we have 
\[[\phi_x,\chi]_v=\chi(\textup{Art}_{K_v}(x))^{-1},\]
where here $\textup{Art}_{K_v}:K_v^\times \rightarrow G_{K_v}^{\textup{ab}}$ denotes the local Artin map.
\item[(iii)] suppose $v$ is such that $\boldsymbol \mu_p\subseteq K_v$ so that $H^1(K_v,\boldsymbol \mu_p)=\mathcal{C}(K_v)$. Then the resulting pairing 
\[[~,~]_v:\mathcal{C}(K_v)\times \mathcal{C}(K_v)\longrightarrow \boldsymbol \mu_p\]
is antisymmetric.
\end{itemize} 
\end{lemma}

\begin{proof}
Part (i) is \cite[Theorem 7.2.15]{MR2392026} whilst part (ii) is Corollary 7.2.13 of op cit. (The cited results are stated for the pairing of \cref{local hilbert symbol} rather than the altered pairing $[~,~]_v$ but in each case that they imply the claimed results is immediate.) Finally, antisymmetry of the cup product 
\[H^1(K_v,\mathbb{Z}/p\mathbb{Z})\times H^1(K_v,\mathbb{Z}/p\mathbb{Z})\longrightarrow H^2(K_v,\mathbb{Z}/p\mathbb{Z}\otimes \mathbb{Z}/p\mathbb{Z})\]
gives part (iii).
\end{proof}

\subsection{Existence of global characters with specified restriction and ramification}

We will need the following lemma which is the analogue of \cite[Proposition 6.8 (iii)]{MR3043582} in the case that the dimension of $T$ is allowed to be larger than $2$.

\begin{notation} 
\label{hilbert symbol}
Writing $\Gamma:=\prod_{v\in \Sigma}\mathcal{C}(K_v)$, we denote by $[~,~]_\Sigma$ the non-degnerate bilinear pairing
\[[~,~]_\Sigma:\left(\prod_{v\in \Sigma}H^1(K_v,\boldsymbol \mu_p) \right) \times \Gamma  \longrightarrow \boldsymbol \mu_p\]
defined as the sum (or rather product) over $v\in \Sigma$ of the pairings $[~,~]_v$ of \Cref{altered pairing}.
\end{notation}

\begin{lemma} \label{existence of characters}
Let $\P$ denote the set of primes of $K$ not in $\Sigma$ which split completely in $K(T)/K$, and fix $\gamma \in \Gamma$. Then there is a character $\chi\in \mathcal{C}(K)$ unramified outside $\Sigma \cup \P$ and with $\chi|_\Gamma =\gamma$, if and only if $[c,\gamma]_\Sigma=0$ for each $c$ in the image of the restriction homomorphism \[H^1(K(T)/K,\boldsymbol \mu_p) \longrightarrow \prod_{v\in \Sigma}H^1(K_v,\boldsymbol \mu_p).\]
\end{lemma}

\begin{proof}
Exactness at the middle term of the Poitou--Tate exact sequence (see, for example, \cite[Theorem I.4.10]{MR2261462}) applied to the set $\Sigma \cup \P$ of places and the $G_K$-module $\mathbb{Z}/p\mathbb{Z}$ (and its dual $\boldsymbol \mu_p$), shows that 
\[\text{im}\left(H^1(K_{\Sigma \cup \P}/K,\mathbb{Z}/p\mathbb{Z}) \longrightarrow \sideset{}{'}\prod_{v\in \Sigma \cup \P} H^1(K_v,\mathbb{Z}/p\mathbb{Z})\right)\]  is the orthogonal complement of  \[\text{im}\left(H^1(K_{\Sigma \cup \P}/K,\boldsymbol \mu_p)\longrightarrow \sideset{}{'}\prod_{v\in \Sigma \cup \P}H^1(K_v,\boldsymbol \mu_p)\right)\]
under the sum of the local pairings of $\cref{local hilbert symbol}$, where here $K_{\Sigma \cup \P}$ denotes the maximal extension of $K$ unramified outside $\Sigma \cup \P$ and the restricted direct products are taken with respect to unramified classes. 

Now fix any choice of isomorphism $\boldsymbol \mu_p\cong \mathbb{Z}/p\mathbb{Z}$ and use it to identify $\mathcal{C}(K)$ with $H^1(K,\mathbb{Z}/p\mathbb{Z})$, and $\mathcal{C}(K_v)$ with $H^1(K_v,\mathbb{Z}/p\mathbb{Z})$ for each $v$ similarly. Then the group $H^1(K_{\Sigma \cup \P}/K,\mathbb{Z}/p\mathbb{Z})$ corresponds to the group of characters unramified outside $\Sigma\cup \P$, which we denote by $\mathcal{C}(K)_{\Sigma \cup \P}$. Making these identifications and projecting onto $\prod_{v\in \Sigma}\mathcal{C}(K_v)$, it follows formally that the image of $\mathcal{C}(K)_{\Sigma \cup \P}$  in $\prod_{v\in \Sigma}\mathcal{C}(K_v)$ is the orthogonal complement with respect to the pairing $[~.~]_{\Sigma}$ of the image of 
\[
\ker\left(H^1(K_{\Sigma \cup \P}/K,\boldsymbol \mu_p)\longrightarrow \sideset{}{'}\prod_{v\in \P}H^1(K_v,\boldsymbol \mu_p)\right)
\]
 in $\prod_{v\in \Sigma}H^1(K_v,\boldsymbol \mu_p)$. We now conclude by the following lemma.
\end{proof}

\begin{lemma} \label{splitting primes lemma}
Let $\P$ denote the set of primes of $K$ not in $\Sigma$ and which split completely in $K(T)/K$, and let $K_{\Sigma \cup \P}$ denote the maximal extension of $K$ unramified outside $\Sigma \cup \P$. Then we have
\[H^1(K(T)/K,\boldsymbol \mu_p)=
\ker\left(H^1(K_{\Sigma \cup \P}/K,\boldsymbol \mu_p)\longrightarrow \sideset{}{'}\prod_{v\in \P}H^1(K_v,\boldsymbol \mu_p)\right),
\]
the groups being compared inside $H^1(K,\boldsymbol \mu_p)$ (and the restricted direct product being taken with respect to unramified classes as above).
\end{lemma}

\begin{proof}
Since $K(T)$ is unramified outside $\Sigma$, we have $K(T)\subseteq K_{\Sigma \cup \P}$. Thus it suffices to show that we have
\[H^1(K(T)/K,\boldsymbol \mu_p)=\ker\left(H^1(K,\boldsymbol \mu_p)\stackrel{\text{res}}{\longrightarrow}\sideset{}{'}\prod_{v\in \P}H^1(K_v,\boldsymbol \mu_p)\right).\]

Since each prime in $\P$ splits completely in $K(T)/K$, the restriction map above factors as
\[H^1(K,\boldsymbol \mu_p)\stackrel{f_1}{\longrightarrow} H^1(K(T),\boldsymbol \mu_p) \stackrel{f_2}{\longrightarrow} \sideset{}{'}\prod_{v\in \mathcal{P}}H^1(K_v,\boldsymbol \mu_p), \] both maps being given by restriction. Since the inflation-restriction exact sequence identifies $H^1(K(T)/K,\boldsymbol \mu_p)$ with $\ker(f_1)$, it suffices to show that $f_2$ is injective. Since $K(T)$ and each $K_v$ ($v\in \P$) contain $\boldsymbol \mu_p$, we may reinterpret $f_2$ as the restriction map on characters
\[\mathcal{C}(K(T))\longrightarrow \sideset{}{'}\prod_{v\in \P}\mathcal{C}(K_v).\]
Suppose $\chi \in \mathcal{C}(K(T))$ is a character of $K(T)$ which is trivial in $\mathcal{C}(K_v)$ for each $v\in \P$, let $L/K(T)$ denote the extension corresponding to the fixed field of the kernel of $\chi$ and let $L'/K$ denote the Galois closure of $L/K$. Then our assumption on $\chi$ means that every prime $v\in \mathcal{P}$ splits completely in $L'/K$. By the Chebotarev density theorem, this gives $[L':K]\leq [K(T):K]$. Since we already know that $K(T)\subseteq L'$ we must have $L'=K(T)$ whence $\chi$ is the trivial character. 
\end{proof}

\subsection{Assumptions on the set of places $\Sigma$}

We now impose conditions on the finite set of places $\Sigma$ (in addition to containing all archimedean places, all primes over $p$ and all places for which $T$ is ramified) which will be necessary for the proof of the remaining cases of \Cref{main twisting theorem combined}.

\begin{assumption} \label{assumptions on places}
We henceforth impose the following conditions on the finite set of places $\Sigma$:

\begin{itemize}
\item[(i)] the restriction homomorphism
\[H^1(K(T)/K,\boldsymbol \mu_p)\longrightarrow \prod_{v\in \Sigma}H^1(K_v,\boldsymbol \mu_p)\] is injective,
\item[(ii)]  $\textup{Pic}(\mathcal{O}_{K,\Sigma})=0$, 
\item[(iii)] the map \[\mathcal{O}^{\times}_{K,\Sigma}/(\mathcal{O}_{K,\Sigma}^{\times})^p\longrightarrow \prod_{v\in \Sigma}K_v^{\times}/(K_v^{\times})^p\]
is injective.
\end{itemize}
(In (ii) and (iii), $\mathcal{O}_{K,\Sigma}$ denotes the elements of $K$ integral outside $\Sigma$.)
\end{assumption}

\begin{lemma}
A set of places $\Sigma$ satisfying \Cref{assumptions on places} exists.
\end{lemma}

\begin{proof}
We begin by taking $\Sigma$ large enough that it contains all archimedean places, all primes over $p$ and all places where $T$ ramifies. By the Grunwald--Wang theorem \cite[Theorem 9.1.9(ii)]{MR2392026}  the map
\[H^1(K,\boldsymbol \mu_p)\longrightarrow \prod_{v\in M_K}H^1(K_v,\boldsymbol \mu_p)\]
is injective. In particular, as $H^1(K(T)/K,\boldsymbol \mu_p)$ is a finite subgroup of $H^1(K,\boldsymbol \mu_p)$, we see that by enlarging $\Sigma$ if necessary we may additionally ensure that (ii) holds.  

Finally, \cite[Lemma 6.1]{MR3043582} shows that any finite set of places may be further enlarged so that (iii) and (iv) hold.
\end{proof}

We will have need of the following consequence of \Cref{assumptions on places}. 


\begin{lemma} \label{existence of generators}
Suppose \Cref{assumptions on places} is satisfied and let $\mathfrak{p}$ be a prime of $K$ with $\mathfrak{p}\notin \Sigma$ and $\boldsymbol \mu_p\subseteq K^{\times}_{\mathfrak{p}}$. Write \[\delta_{\mathfrak{p}}:K_{\mathfrak{p}}^{\times}/K_{\mathfrak{p}}^{{\times}p}\stackrel{\sim}{\longrightarrow}\mathcal{C}(K_\mathfrak{p})\] 
for the isomorphism (coming from the Kummer sequence) sending $x\in K_{\mathfrak{p}}^{\times}$ to the character $\sigma \mapsto \sigma(y)/y$
where $y\in \bar{K}_{\mathfrak{p}}^{\times}$ is such that $y^p=x$ (any two choices for $y$ yield the same character since $\boldsymbol \mu_p \subseteq K_{\mathfrak{p}}$).

Then there is a (global) character $\varphi(\mathfrak{p})\in \mathcal{C}(K)$ satisfying the following three conditions: 
\begin{equation} \label{our global generators}
\phantom{hello} \begin{matrix*}[l]  \varphi(\mathfrak{p})\textup{ ramifies at }\mathfrak{p}, \\ \varphi(\mathfrak{p})\textup{ is unramified outside }\Sigma \cup \{\mathfrak{p}\},\\ \textup{the restriction of }\varphi({\mathfrak{p}}) \textup{ to }\mathcal{C}(K_{\mathfrak{p}}) \textup{ is equal to }\delta_{\mathfrak{p}}(\varpi) \textup{  for some uniformiser }\varpi \textup{ of }K_{\mathfrak{p}}. \end{matrix*}
\end{equation}
\end{lemma}

\begin{proof}
Given the assmuptions on $\Sigma$, the existence of a character $\varphi(\mathfrak{p})$ which ramifies at $\mathfrak{p}$ and is unramified outside $\Sigma\cup  \{\mathfrak{p}\}$ follows from  \cite[Proposition 6.8 (ii)]{MR3043582}. Fix one such and pick $x\in K_{\mathfrak{p}}^{\times}$ such that the restriction of $\varphi(\mathfrak{p})$ is equal to $\delta_{\mathfrak{p}}(x)$. Since $\varphi(\mathfrak{p})$ ramifies at $\mathfrak{p}$, the extension of  $K_{\mathfrak{p}}^{\times}$ obtained by adjoining a $p$-th root of $x$ ramifies. In particular, since  $K_{\mathfrak{p}}$ has  residue characteristic coprime to $p$ (as $\mathfrak{p}\notin \Sigma$), the valuation $v_\mathfrak{p}(x)$ of $x$ is coprime to $p$. Noting that replacing $\varphi({\mathfrak{p}})$ with $\varphi(\mathfrak{p})^m$ for any $m$ coprime to $p$ yields another character which ramifies at $\mathfrak{p}$ and is unramified outside $\Sigma\cup \{\mathfrak{p}\}$, we may suppose that $v_\mathfrak{p}(x)$ is congruent to $1$ modulo $p$. Finally, since $ K_{\mathfrak{p}}^{\times p}$ is in the kernel of $\delta_{\mathfrak{p}}$ we may now shift $x$ by a $p$-th power of a uniformiser to suppose that $x$ has valuation $1$ as desired. 
\end{proof}

The following lemma evaluates the pairing $[~,~]_\Sigma$ of \Cref{hilbert symbol} between the characters $\varphi({\mathfrak{p}})$ of \Cref{existence of generators} and elements of $H^1(K(T)/K,\boldsymbol \mu_p)$. 

\begin{lemma} \label{preevaluation of alpha}
Let $\mathfrak{p}$ be a prime of $K$ not in $\Sigma$, let $\varphi(\mathfrak{p})$ satisfy \cref{our global generators}, and let $c\in H^1(K(T)/K,\boldsymbol \mu_p)$. Then writing $\textup{Frob}_{\mathfrak{p}}$ for the Frobenius element at $\mathfrak{p}$ in $\textup{Gal}(K(T)/K)$, we have 
\[[c,\varphi(\mathfrak{p})]_\Sigma=c(\textup{Frob}_{\mathfrak{p}}).\]
\end{lemma}

\begin{proof}
By global class field theory the product of $[c,\varphi({\mathfrak{p}})]_v$ over all places of $K$ is equal to 1. In particular, we have
\[[c,\varphi(\mathfrak{p})]_\Sigma=\prod_{v\notin \Sigma}[c,\varphi({\mathfrak{p}})]_v.\]
If $\mathfrak{q}$ is a prime of $K$ not in $\Sigma$  then $\mathfrak{q}\nmid p$ and, additionally, $K(T)/K$ is unramified at $\mathfrak{q}$ whence the restriction of $c$ to $H^1(K_\mathfrak{q},\boldsymbol \mu_p)$ is in the unramified subgroup $H^1_{\text{ur}}(K_\mathfrak{q},\boldsymbol \mu_p)$. If $\mathfrak{q}\neq \mathfrak{p}$ then $\varphi({\mathfrak{p}})$ is also unramified at $\mathfrak{q}$ whence $[c,\varphi(\mathfrak{p})]_\mathfrak{q}=1$ by \Cref{properties of local symbol} (i).  

It now follows that  $[c,\varphi(\mathfrak{p})]_\Sigma=[c,\varphi(\mathfrak{p})]_\mathfrak{p}$ and to conclude we must show that $[c,\varphi(\mathfrak{p})]_\mathfrak{p}=c(\text{Frob}_{\mathfrak{p}})$. Since $\boldsymbol \mu_p\subseteq K_{\mathfrak{p}}$ and we've chosen $\varphi(\mathfrak{p})$ so that its restriction to $\mathcal{C}(K_\mathfrak{p})$ agrees with $\delta_\mathfrak{p}(\varpi)$ for some uniformiser $\varpi$ of $K_{\mathfrak{p}}$, parts (ii) and (iii) of \Cref{properties of local symbol} combine to give
\[[c,\varphi({\mathfrak{p}})]_{\mathfrak{p}}=c(\textup{Art}_{K_{\mathfrak{p}}}(\varpi)).\] Now $c$ is unramified at $\mathfrak{p}$ and  by standard properties of the local Artin map, $\textup{Art}_{K_{\mathfrak{p}}}(\varpi)|_{{K_\mathfrak{p}}^{\text{nr}}}=\text{Frob}_{K_{\mathfrak{p}}}$. On the other hand, since $c$ came from $H^1(K(T)/K,\boldsymbol \mu_p)$, its restriction to $H^1(K_\mathfrak{p},\boldsymbol \mu_p)$ factors through $\text{Gal}(K_\mathfrak{p}(T)/K_{\mathfrak{p}})$. As the restriction of $\text{Frob}_{K_{\mathfrak{p}}}$ to $\text{Gal}(K_\mathfrak{p}(T)/K_{\mathfrak{p}})$ is precisely $\text{Frob}_{\mathfrak{p}}$, we have the result. 
\end{proof}

\section{Disparity in Selmer ranks: remaining cases}

We now treat the remaining cases of \Cref{main twisting theorem combined}, namely when $p=2$ and $\epsilon$ fails to be  a homomorphism, or when $p>2$ and $\epsilon$ is non-trivial when restricted to $\text{Gal}(K(T)/K(\boldsymbol \mu_p))$. Our strategy is broadly based on that of \cite[Section 8]{MR3043582}, although the arguments are more involved in order to allow the dimension of $T$ to be arbitrary.

As before, let $G$ denote the Galois group of $K(T)/K$ and write $\Gamma:=\prod_{v\in \Sigma}\mathcal{C}(K_v)$. For $\chi \in \mathcal{C}(K)$ we denote by $\chi|_\Gamma$ the image of $\chi$ in $\Gamma$ under the (product of the) natural restriction map(s).
 
 We fix a finite set of places $\Sigma$ satisfying \Cref{assumptions on places}.

\begin{defi} \label{definition of w}
Define a map $w:\mathcal{C}(K)\rightarrow \{\pm 1\}$ by 
\[w(\chi):=\prod_{v\notin \Sigma,~\chi_v~\text{ram}}(-1)^{\dim_{\mathbb{F}_p}T^{G_{K_v}}}=\prod_{v\notin \Sigma,~\chi_v~\text{ram}} \epsilon(\textup{Frob}_v),\]
where here $\textup{Frob}_v\in G$ denotes the Frobenius element at $v$ in $K(T)/K$. 
\end{defi}

\begin{remark} \label{local decomp remark}
By \Cref{local decomposition}, for each $\chi \in \mathcal{C}(K)$ we have
\[(-1)^{\dim_{\mathbb{F}_2}\textup{Sel}(T,\chi)}=w(\chi)(-1)^{\dim_{\mathbb{F}_2}\textup{Sel}(T,\mathbbm{1}_K)}\prod_{v\in \Sigma}(-1)^{h_v(\mathbbm{1}_K,\chi_v)}.\]
\end{remark}

We now examine the extent to which $w(\chi)$ behaves `independently' of the restriction of $\chi$ to $\Gamma$. To this end, we make the following definition.

\begin{defi}
For each $X\geq 1$ and $\gamma \in \Gamma$, define 
\[s_X(\gamma)=\frac{|\{\chi \in \mathcal{C}(K,X)~~:~~\chi|_\Gamma=\gamma,~~w(\chi)=1\}|}{|\{\chi \in \mathcal{C}(K,X)~~:~~\chi|_\Gamma=\gamma\}|}.\]
\end{defi}

The rest of the section is occupied with the proof of the following Theorem. 

\begin{theorem} \label{independence theorem}
We have
\begin{itemize}
\item[(i)] if $p=2$ and $\epsilon$ fails to be a homomorphism then for all sufficiently large $X$, $s_X(\gamma)=\frac{1}{2}$ for all $\gamma \in \Gamma$,
\item[(ii)] if $p>2$ and $\epsilon$ is non-trivial when restricted to $\textup{Gal}(K(T)/K(\boldsymbol \mu_p))$ then $\lim_{X\rightarrow \infty} s_X(\gamma)=\frac{1}{2}$ for all $\gamma \in \Gamma$.
\end{itemize} 
\end{theorem}

Assuming this for the moment, we get as a corollary the remaining cases of \Cref{main twisting theorem combined}.

\begin{theorem} \label{epsilon non-trivial twisting theorem} 
We have
\begin{itemize}
\item[(i)] if $p=2$ and $\epsilon$ fails to be a homomorphism then for all sufficiently large $X$, \[\frac{|\{\chi \in \mathcal{C}(K,X)~~:~~\dim_{\mathbb{F}_p}\textup{Sel}(T,\chi)~~\textup{is even}\}|}{|\mathcal{C}(K,X)|}=\frac{1}{2},\]
\item[(ii)] if $p>2$ and $\epsilon$ is non-trivial when restricted to $\textup{Gal}(K(T)/K(\boldsymbol \mu_p))$ then \[\lim_{X\rightarrow \infty} \frac{|\{\chi \in \mathcal{C}(K,X)~~:~~\dim_{\mathbb{F}_p}\textup{Sel}(T,\chi)~~\textup{is even}\}|}{|\mathcal{C}(K,X)|}=\frac{1}{2}.\]
\end{itemize} 
\end{theorem}

\begin{proof}
Fix $\gamma \in \Gamma$ and suppose that $\chi \in \mathcal{C}(K,X)$ is such that $\chi|_\Gamma=\gamma$. Then by \Cref{local decomp remark} we have 
\[\text{dim}_{\mathbb{F}_2}\textup{Sel}(T,\chi) \textup{ is even}~~\Leftrightarrow ~~ w(\chi)=(-1)^{\dim_{\mathbb{F}_2}\textup{Sel}(T,\mathbbm{1}_K)}\prod_{v\in \Sigma}(-1)^{h_v(\mathbbm{1}_K,\gamma_v)},\]
and the right hand side depends only on $\gamma$. In particular, by \Cref{independence theorem} we have
\[\lim_{X\rightarrow \infty}\frac{|\{\chi \in \mathcal{C}(K,X)~~:~~\chi|_\Gamma=\gamma\text{ and }\text{dim}_{\mathbb{F}_2}\textup{Sel}(T,\chi) \textup{ is even}\}|}{|\{\chi \in \mathcal{C}(K,X)~~:~~\chi|_\Gamma=\gamma\}|}=\frac{1}{2},\]
and if $p=2$ then this is in fact an equality for all sufficiently large $X$  rather than a limit. 
Averaging over all $\gamma \in \Gamma$  gives the result (note that the sets 
$\{\chi \in \mathcal{C}(K,X)~~:~~\chi|_\Gamma=\gamma\}$  all have the same size for sufficiently large $X$ as the restriction map $\chi\mapsto \chi|_\Gamma$ is a homomorphism and is surjective for $X$ sufficiently large by \Cref{surjectivity of restriction}).
\end{proof}

%

We now turn to the proof of \Cref{independence theorem}. 

\begin{defi} \label{definition of beta 2}
Fix an $\mathbb{F}_p$-basis $\{\phi_1,...,\phi_r\}$ for $H^1(K(T)/K,\boldsymbol \mu_p)$. Further, define the homomorphism $f:\Gamma \rightarrow \boldsymbol \mu_p^r$ by setting
\[f(\gamma)= \left([\phi_i,\gamma]_\Sigma\right)_{i=1}^r\]
where here we view the $\phi_i$ inside $\prod_{v\in \Sigma}H^1(K_v,\boldsymbol \mu_p)$ via the product of the natural restriction maps, and $[~,~]_\Sigma$ is the pairing of \Cref{hilbert symbol} (we allow the case $r=0$ in  which case $\boldsymbol \mu_p^r$ is the trivial group).
\end{defi}

\begin{remark} \label{surjectivity of beta 2}
Since we have taken $\Sigma$ large enough that the map
\[H^1(K(T)/K,\boldsymbol \mu_p)\longrightarrow \prod_{v\in \Sigma}H^1(K_v,\boldsymbol \mu_p)\] is injective, it follows from the non-degeneracy of the pairing $[~,~]_\Sigma$  that $f$ is surjective.
\end{remark}

\begin{defi} \label{defi of t}
For each $n\geq 1$ and $\eta \in \boldsymbol \mu_p^r$, define
\[t_X(\eta)=\frac{|\{\chi \in \mathcal{C}(K,X)~~:~~f(\chi|_\Gamma)=\eta,~w(\chi)=1\}|}{|\{\chi \in \mathcal{C}(K,X)~~:~~f(\chi|_\Gamma)=\eta\}|}.\]
\end{defi}

The following lemma reduces the problem of understanding $s_X(\gamma)$ as $\gamma$ ranges over the elements of $\Gamma$, to understanding $t_X(\eta)$ as $\eta$ ranges over the elements of $\boldsymbol \mu_p^r$.

\begin{lemma}\textup{(c.f. }\cite[Lemma 8.4]{MR3043582}\textup{)} \label{reduction to beta}
Let $\gamma \in \Gamma$. Then for $X$ sufficiently large we have
\[s_X(\gamma)=t_X(f(\gamma)).\] 
\end{lemma}

\begin{proof}
Let $\P$ denote the set of primes of $K$ not in $\Sigma$ and which split completely in $K(T)/K$, and let $\gamma'\in \Gamma$ be such that $f(\gamma')=f(\gamma)$. Then $\gamma'\gamma^{-1}$ is in the kernel of $f$ so by \Cref{existence of characters} there is $\chi_{\gamma,\gamma'}\in \mathcal{C}(K)$ with $\chi_{\gamma,\gamma'}|_\Gamma=\gamma'\gamma^{-1}$ and such that $\chi_{\gamma,\gamma'}$ is unramified outside $\Sigma \cup \P$. Now for any $\chi\in \mathcal{C}(K)$ we have $w(\chi)=w(\chi \chi_{\gamma,\gamma'})$ since the sets of primes not in $\Sigma$ where $\chi$ and $\chi_{\gamma,\gamma'}$ ramify differ only at primes $\mathfrak{p}\in \P$, and at such primes we have \[\epsilon(\text{Frob}_{\mathfrak{p}})=\epsilon(1)=(-1)^{\text{dim}T}=1\] (where as usual $\text{Frob}_{\mathfrak{p}}$ denotes the Frobenius element at $\mathfrak{p}$ in $K(T)/K$). Thus if $X$ is sufficiently large that $\chi_{\gamma,\gamma'}$ is in $\mathcal{C}(K,X)$, multiplication by $\chi_{\gamma,\gamma'}$ gives a bijection between the set 
\[\{\chi \in \mathcal{C}(K,X)~~:~~\chi|_\Gamma=\gamma,~w(\chi)=1\}\]
and the set
\[\{\chi \in \mathcal{C}(K,X)~~:~~\chi|_\Gamma=\gamma',~w(\chi)=1\},\]
as well as between the same two sets with the conditions on $w(\chi)$ removed. 

Writing $\eta=f(\gamma)$, it follows that for $X$ sufficiently large we have
\[t_X(\eta)=\frac{\sum_{\gamma' \in f^{-1}(\{\eta\})}|\{\chi \in \mathcal{C}(K,X)~~:~~\chi|_\Gamma=\gamma',~w(\chi)=1\}|}{\sum_{\gamma' \in f^{-1}(\{\eta\})}|\{\chi \in \mathcal{C}(K,X)~~:~~\chi|_\Gamma=\gamma'\}|}\]
\[~~~~~~~=\frac{|f^{-1}(\{\eta\})|\cdot |\{\chi \in \mathcal{C}(K,X)~~:~~\chi|_\Gamma=\gamma,~w(\chi)=1\}|}{|f^{-1}(\{\eta\})|\cdot |\{\chi \in \mathcal{C}(K,X)~~:~~\chi|_\Gamma=\gamma\}|}=s_X(\gamma)\]
as desired.
\end{proof}

We now study the quantities $t_X(\eta)$ as $\eta$ ranges over $\boldsymbol \mu_p^r$, splitting into cases  according to $p=2$ or $p>2$.

\subsection{The case where $p=2$ and $\epsilon$ fails to be a homomorphism}

Suppose now that $p=2$ and $\epsilon$ fails to be a homomorphism.  

\begin{defi} \label{definition of alpha}
Define the map $\theta:\mathcal{C}(K)\rightarrow \boldsymbol \mu_2^r \times \{\pm 1\}$ by setting
\[\theta(\chi):=\left(f(\chi|_\Gamma),w(\chi)\right).\]
\end{defi}

The following observation will be crucial to our method. We remark that it fails for $p>2$.

\begin{lemma}
The map $\theta$   is a homomorphism.
\end{lemma}

\begin{proof}
Since both the  restriction map $\mathcal{C}(K)\rightarrow \Gamma$ and the map $f:\Gamma \rightarrow \boldsymbol\mu_2^r$ are homomorphisms, it suffices to show that $w:\mathcal{C}(K)\rightarrow \{\pm 1\}$ is a homomorphism.

For each $v\notin \Sigma$, define a map $w_v:\mathcal{C}(K_v)\rightarrow \{\pm 1\}$ by
\[w_v(\chi)=\begin{cases} (-1)^{\dim_{\mathbb{F}_2}T^{G_{K_v}}} ~~&~~\chi~~\text{ramified}\\ 1~~&~~\text{else}.\end{cases}\]
Since $w$ is the product of the $w_v$ over $v\notin \Sigma$, it suffices to show that each $w_v$ is a homomorphism. To see this, note that as $v\notin \Sigma$, $K_v$ has odd residue characteristic. In particular, the product of any two ramified characters of $K_v$ is unramified, and the product of a ramified  character with an unramified character is again ramified.
\end{proof}

\begin{remark} \label{strategy}
 For $X>0$, write $\theta_X$ for the restriction of $\theta$ to $\mathcal{C}(K,X)$. Then for each $\eta \in \boldsymbol \mu_2^r$, we have
\[t_X(\eta)=\frac{|\theta_X^{-1}\left((\eta,1)\right)|}{|\theta_X^{-1}\left((\eta,1)\right)|+|\theta_X^{-1}\left((\eta,-1)\right)|}.\]
Now (for $X> 1$), $\mathcal{C}(K,X)$ is a group and $\theta_X$ is a homomorphism.  Thus the fibres over points in the image of $\theta_X$ have the same size, being cosets of the kernel. In light of \Cref{reduction to beta}, \Cref{independence theorem} (i) is equivalent to the statement that, if $\epsilon$ fails to be a homomorphism, then $\theta_X$ is surjective for sufficiently large $X>0$. Since $\boldsymbol \mu_2^r\times \{\pm 1\}$ is a finite group this is, in turn, equivalent to the statement that if $\epsilon$ fails to be a homomorphism then $\theta$ is surjective. This is the statement we now study, and prove in \Cref{surjectivity of alpha}.
\end{remark}

We now fix a collection of global characters $\{\varphi(\mathfrak{p})\}_{\mathfrak{p}\notin \Sigma}$ satisfying \cref{our global generators}. Each $\varphi(\mathfrak{p})$ is ramified at $\mathfrak{p}$, yet unramified outside $\Sigma \cup \{\mathfrak{p}\}$.
%
\Cref{preevaluation of alpha} allows us to evaluate the map $\theta$ on the $\varphi(\mathfrak{p})$. 

\begin{lemma} \label{explicit evaluation of composition}
For each $\mathfrak{p}\notin \Sigma$ we have
\[\theta(\varphi(\mathfrak{p}))=\left(\left(\phi_i(\textup{Frob}_{\mathfrak{p}})\right)_{i=1}^r,\epsilon(\textup{Frob}_{\mathfrak{p}})\right)\]
where here $\textup{Frob}_{\mathfrak{p}}\in G$ denotes the Frobenius element at $\mathfrak{p}$ in $K(T)/K$. 
\end{lemma}

\begin{proof}
Since amongst the primes not in $\Sigma$ the character $\varphi(\mathfrak{p})$ only ramifies at $\mathfrak{p}$, we have $w(\varphi(\mathfrak{p}))=\epsilon(\text{Frob}_{\mathfrak{p}})$ by definition. We have $f(\varphi(\mathfrak{p})|_\Gamma)=(\phi_i(\text{Frob}_{\mathfrak{p}}))_{i=1}^r$  by \Cref{preevaluation of alpha}.
\end{proof}


\begin{proposition} \label{surjectivity of alpha}
The map $\theta:\mathcal{C}(K)\longrightarrow \boldsymbol \mu_2^r \times \{\pm 1\}$ of \Cref{definition of alpha} is surjective if and only if $\epsilon$ fails to be a homomorphism.  
\end{proposition}

\begin{proof}
Note that the subgroup $\mathcal{U}$ of $\mathcal{C}(K)$ consisting of characters unramified outside $\Sigma$ is in the kernel of $\theta$, and the quotient $\mathcal{C}(K)/\mathcal{U}$ is generated by the $\varphi({\mathfrak{p}})$ as $\mathfrak{p}$ ranges over primes not in $\Sigma$.

 By the Chebotarev density theorem, each conjugacy class in $G=\text{Gal}(K(T)/K)$ arises as $\text{Frob}_\mathfrak{p}$ for some $\mathfrak{p}\notin \Sigma$ and so by \Cref{explicit evaluation of composition} it follows that the image of $\theta$ is the subgroup of $\boldsymbol \mu_2^r\times\{\pm 1\}$ generated by the set 
\[\{\left((\phi_i(\sigma))_{i=1}^r,\epsilon(\sigma)\right)~~:~~\sigma \in G\}\]
(note that for $\sigma \in G$, both $\epsilon(\sigma)$ and the $\phi_i(\sigma)$ depend only on the conjugacy class of $\sigma$ in $G$). 

Recall that the set $\{\phi_i:1\leq i \leq r\}$ is a basis for $H^1(K(T)/K,\boldsymbol \mu_2)=\text{Hom}(G,\boldsymbol \mu_2)$. To make this more explicit, denote by $G^2$ the subgroup of $G$ generated by the squares of all the elements of $G$. It's a normal subgroup and the quotient $G/G^2$ is an abelian group of exponent $2$. That is, $G/G^2$ is a finite dimensional $\mathbb{F}_2$-vector space. Since every homomorphism from $G$ to $\boldsymbol \mu_2$ factors through G/$G^2$, we have
\[\text{Hom}(G,\boldsymbol \mu_2)=\text{Hom}(G/G^2,\boldsymbol \mu_2)\]
and the right hand group is just the dual of $G/G^2$ as an $\mathbb{F}_2$-vector space. In particular, the map $G/G^2\rightarrow \boldsymbol \mu_2^r$ sending $\sigma$ to $(\phi_i(\sigma))_{i=1}^r$ is an isomorphism.

Combining the above, we arrive at a purely group theoretic criterion: $\theta$ is surjective if and only if the set 
\[S:=\{\left(\bar{\sigma},\epsilon(\sigma)\right)~:~\sigma \in G\}\] 
generates $G/G^2\times \{\pm 1\}$, where here for $\sigma \in G$ we write $\bar{\sigma}$ for the image of $\sigma$ in $G/G^2$. 

Suppose now that $\epsilon$ is a homomorphism. Then $\epsilon$ necessarily factors through $G/G^2$ and we see that $S$ generates an index $2$ subgroup of $G/G^2\times \{\pm 1\}$, so that $\alpha$ is not surjective in this case.

Conversely, suppose that $\epsilon$ fails to be a homomorphism and write $H$ for the subgroup of $G/G^2\times \{\pm 1\}$ generated by $S$. By assumption, we may find $\sigma$ and $\tau$ in $G$ with $\epsilon(\sigma \tau)=-\epsilon(\sigma)\epsilon(\tau)$. Then
\[(\bar{\sigma},\epsilon(\sigma))\cdot (\bar{\tau},\epsilon(\tau))\cdot (\overline{\sigma \tau},\epsilon(\sigma \tau))=\left(\overline{(\sigma \tau)^2},-1\right)=(1,-1)\]
is in $H$ (here the first $1$ denotes the identity in $G/G^2$). Then for any $\sigma \in G$, both $(\bar{\sigma},\epsilon(\sigma))$ and \[(\bar{\sigma},-\epsilon(\sigma))=(1,-1) \cdot(\bar{\sigma},\epsilon(\sigma))\] are in $H$. Thus $H=G/G^2\times \{\pm 1\}$ whence $\alpha$ is surjective as desired. 
\end{proof}

\begin{proof}[Proof of \Cref{independence theorem} (i)]
By \Cref{strategy} we see that \Cref{independence theorem} (i) holds if and only if $\theta$ is surjective whenever $\epsilon$ fails to be a homomorphism. The result now follows from \Cref{surjectivity of alpha}.
\end{proof}

\subsection{The case where $p>2$ and $\epsilon$ is non-trivial when restricted to $\text{Gal}(K(T)/K(\boldsymbol \mu_p))$}

Suppose now that $p>2$ and that the restriction of $\epsilon$ to $\text{Gal}(K(T)/K(\boldsymbol \mu_p))$ is non-trivial. 

We begin by defining a slight refinement of the quantity $t_X(\eta)$.

\begin{defi} \label{defi of hat t}
Fix an enumeration of the primes $\mathfrak{p}\notin \Sigma$ such that if $i\leq j$ then $N(\mathfrak{p}_i)\leq N(\mathfrak{p}_j)$, and for each $n\geq 1$ define the subgroup $\mathcal{C}_n(K)$ of $\mathcal{C}(K)$ by
\[\mathcal{C}_n(K):=\{\chi \in \mathcal{C}(K)~~:~~\chi \text{ is unramified outside }\Sigma \cup\{\mathfrak{p}_1,...,\mathfrak{p}_n\}\}.\] Further, for each $n\geq 1$ and $\eta\in \boldsymbol \mu_p^r$, define 
\[\hat{t}_n(\eta):=\frac{|\{\chi \in \mathcal{C}_n(K)~~:~~f(\chi|_\Gamma)=\eta,~~w(\chi)=1\}|}{|\{\chi \in \mathcal{C}_n(K)~~:~~f(\chi|_\Gamma)=\eta\}|}-\frac{1}{2}.\]
\end{defi}

\begin{remark}
Note that we subtract $1/2$ in the definition of $\hat{t}_n(\eta)$ whilst we did not in the definition of $t_X(\eta)$. This will neaten the statement of some results in the rest of the section. Clearly for any $\eta \in \boldsymbol \mu_p^r$, to show that $\lim_{X\rightarrow \infty}t_X(\eta)=\frac{1}{2}$ it suffices to show that $\lim_{n\rightarrow \infty}\hat{t}_n(\eta)=0$.
\end{remark}

As in the case $p=2$ we now fix a collection of global characters \[\{\varphi({\mathfrak{p}})\}_{\mathfrak{p}\notin \Sigma,~\boldsymbol \mu_p\subseteq K_{\mathfrak{p}}}\] satisfying \cref{our global generators}.



\begin{lemma} \label{structure of characters}
Fix $n\geq 1$. Then if $\boldsymbol \mu_p \subsetneq K_{\mathfrak{p}_{n+1}}$  we have $\mathcal{C}_{n+1}(K)=\mathcal{C}_n(K)$. On the other hand, if $\boldsymbol \mu_p \subseteq K_{\mathfrak{p}_{n+1}}$ then we have
\[\mathcal{C}_{n+1}(K)=\bigsqcup_{i=0}^{p-1} \varphi({\mathfrak{p}_{n+1}})^i \cdot \mathcal{C}_n(K).\]
\end{lemma}

\begin{proof}
In each case this follows from the structure of $\mathcal{C}(K_{\mathfrak{p}_{n+1}})$; see \cite[Lemma 8.3]{MR3043582}.
\end{proof}

\begin{defi} \label{defi of matrix}
Let $V$ be the regular representation of $\boldsymbol \mu_p^r$ over $\mathbb{C}$, so that $V$ has basis $\{e_\eta~~:~~\eta \in \boldsymbol \mu_p^r\}$ on which $\boldsymbol \mu_p^r$ acts via $\eta'\cdot e_\eta=e_{\eta'\eta}$.
For each $n\geq 1$, define
\[\hat{\textbf{t}}_n:=\sum_{\eta \in \boldsymbol \mu_p^r}\hat{t}_n(\eta)e_\eta\in V\]
and for $\sigma\in \textup{Gal}(K(T)/K(\boldsymbol \mu_p))$, define $\rho(\sigma):=\left(\phi_i(\sigma)\right)_{i=1}^r \in \boldsymbol \mu_p^r$
and \[M(\sigma):=\frac{1}{p}\left(1+\epsilon(\sigma)\sum_{i=1}^{p-1}\rho(\sigma)^i\right)\in \textup{End}(V).\] 
\end{defi}

\begin{remark} \label{conjugacy class remark}
For $\sigma \in \textup{Gal}(K(T)/K(\boldsymbol{\mu}_p))$, the element $M(\sigma)$ depends only on the conjugacy class of $\sigma$ in $G$. Indeed, for  each $1\leq i \leq r$ and $g\in G$, the cocycle relation for $\phi_i$ gives $\phi_i(g\sigma g^{-1})=g\phi_i(\sigma)$. It now follows that for each $i$, $\sum_{j=1}^{p-1}\phi_i(\sigma)^j$
depends only on the conjugacy class of $\sigma$ in $G$. Since the same is true for $\epsilon(\sigma)$ we are done.
\end{remark}

\begin{lemma} \label{recurrence for t}
Fix $n\geq 1$. If $\boldsymbol \mu_p\subsetneq K_{\mathfrak{p}_{n+1}}$ then we have $\hat{\textbf{t}}_{n+1}=\hat{\textbf{t}}_n$. On the other hand, if $\boldsymbol \mu_p\subseteq K_{\mathfrak{p}_{n+1}}$ then we have the following recurrence relation for $\hat{\textbf{t}}_n$:
\[\hat{\textbf{t}}_{n+1}=M(\textup{Frob}_{\mathfrak{p}_{n+1}})\hat{\textbf{t}}_n,\]
where here $\textup{Frob}_{\mathfrak{p}_{n+1}}\in G$ denotes the Frobenius element at $\mathfrak{p}_{n+1}$ in $K(T)/K$.
\end{lemma}

\begin{proof}
If $\boldsymbol \mu_p\subsetneq K_{\mathfrak{p}_{n+1}}$ then $\mathcal{C}_{n+1}(K)=\mathcal{C}_n(K)$ and the result is clear. 

Suppose now that $\boldsymbol \mu_p\subseteq K_{\mathfrak{p}_{n+1}}$ and define the map $\theta:\mathcal{C}(K)\rightarrow \boldsymbol \mu_2^r \times \{\pm 1\}$ by 
\[\theta(\chi):=\left(f(\chi|_\Gamma),w(\chi)\right)\]
(note that, unlike the case $p=2$ this is not a homomorphism). Then \Cref{preevaluation of alpha} gives
\[\theta(\varphi({\mathfrak{p}_n}))=\left(\rho(\textup{Frob}_{\mathfrak{p}_n}),\epsilon(\textup{Frob}_{\mathfrak{p}_n})\right).\]
 Moreover, if $\chi_0\in \mathcal{C}_n(K)$ then we have
\[\theta(\chi_0\cdot\varphi({\mathfrak{p}_{n+1}})^i)=\theta(\chi_0)\cdot \theta(\varphi({\mathfrak{p}_{n+1}})^i)\] 
since the sets of primes not in $\Sigma$ at which $\chi_0$ and $\varphi({\mathfrak{p}_{n+1}})^i$ ramify are disjoint. 
Writing $\sigma$ for $\text{Frob}_{\mathfrak{p}_{n+1}}$, this gives
\[\theta(\chi_0\cdot \varphi({\mathfrak{p}_{n+1}})^i)=\begin{cases}\theta(\chi_0)~~&~~i=0\\  \theta(\chi_0)\cdot \left(\rho(\sigma)^i,\epsilon(\sigma)\right)~~&~~1\leq i\leq p-1.\end{cases}\]
It now follows from \Cref{structure of characters} that for each $\eta\in \boldsymbol \mu_p^r$ we have
\begin{eqnarray*}
\left|\{\chi \in \mathcal{C}_{n+1}(K)~~:~~\theta(\chi)=(\eta,1)\}\right|&=&\sum_{i=0}^{p-1}\left|\{\chi \in \varphi({\mathfrak{p}_{n+1}})^i\cdot \mathcal{C}_{n}(K)~~:~~\theta(\chi)=(\eta,1)\}\right|
\end{eqnarray*}
\[=\left|\left\{\chi_0 \in \mathcal{C}_{n}(K)~~:~~\theta(\chi_0)=(\eta,1)\right\}\right|\]
\[+\sum_{i=1}^{p-1}\left|\left\{\chi_0 \in \mathcal{C}_{n}(K)~~:~~\theta(\chi_0)=\left(\eta \cdot \rho(\sigma)^{-i},\epsilon(\sigma)\right)\right\}\right|.\]
Dividing through by $|\mathcal{C}_{n+1}(K)|=p|\mathcal{C}_{n}(K)|$ gives
\[\hat{t}_{n+1}(\eta)=\frac{1}{p}\left(\hat{t}_n(\eta)+\epsilon(\sigma)\sum_{i=1}^{p-1} \hat{t}_n\left(\rho(\sigma)^{-i}\cdot \eta \right)\right)\]
and the result now follows from the definition of $M(\sigma)$.
\end{proof}

\begin{lemma} \label{limit of operator norm}
For any $m\geq 1$ and $\sigma \in \textup{Gal}(K(T)/K(\boldsymbol \mu_p))$ we have
\[M(\sigma)^m=\begin{cases}M(\sigma) ~~&~~\epsilon(\sigma)=1,\\\left(\frac{2}{p}\left(\frac{2-p}{p}\right)^m-\frac{2-p}{p}\left(\frac{2}{p}\right)^m\right)\textup{id}_V+\left(\left(\frac{2}{p}\right)^m-\left(\frac{2-p}{p}\right)^m\right)M(\sigma)~~&~~\epsilon(\sigma)=-1. \end{cases} \] 
In particular, for $p>2$ and writing $||\cdot ||$ for the operator norm on $\textup{End}(V)$, we have
\[\lim_{m\rightarrow \infty}||M(\sigma)^m||=\begin{cases}||M(\sigma)||~~&~~\epsilon(\sigma)=1,\\ 0~~&~~\epsilon(\sigma)=-1. \end{cases}\]
\end{lemma}

\begin{proof}
Fix  $\sigma \in \textup{Gal}(K(T)/K(\boldsymbol \mu_p))$ and define
\[T(\sigma):=\frac{1}{p}\sum_{i=0}^{p-1}\rho(\sigma)^i.\]
Then $T(\sigma)$ is an idempotent in $\textup{End}(V)$ (e.g. by orthogonality of characters of $\boldsymbol \mu_p$ or by explicit computation) so that $T(\sigma)^m=T(\sigma)$ for each $m\geq 1$. 
Note that we have
\[M(\sigma)=\begin{cases}T(\sigma)~~&~~\epsilon(\sigma)=1\\ \frac{2}{p}-T(\sigma)~~&~~\epsilon(\sigma)=-1.\end{cases}\]
If $\epsilon(\sigma)=1$ this immediately gives $M(\sigma)^m=M(\sigma)$, whilst if $\epsilon(\sigma)=-1$ 
the result now follows easily either by induction on $m$ or by expanding $(\frac{2}{p}-T(\sigma))^m$ with the binomial theorem. 
%

Since $p>2$ we have both
\[\lim_{m\rightarrow \infty}\left(\frac{2}{p}\right)^m=0~~\text{ and }~~\lim_{m\rightarrow \infty}\left(\frac{2-p}{p}\right)^m=0,\]
from which the statement about $\lim_{m\rightarrow \infty}||M(\sigma)^m||$ follows immediately.
\end{proof}

\begin{proposition} \label{limit of t}
Suppose that $p>2$ and $\epsilon$ is non-trivial when restricted to $\textup{Gal}(K(T)/K(\boldsymbol \mu_p))$. Then for each $\eta \in \boldsymbol \mu_p^r$ we have
\[\lim_{n\rightarrow \infty}\hat{t}_n(\eta)=0.\]
\end{proposition}

\begin{proof}
Write $H:=\text{Gal}(K(T)/K(\boldsymbol \mu_p))$, and note that this is a normal subgroup of $G$. For each $n\geq 1$ we have $\boldsymbol \mu_p\subseteq K_{\mathfrak{p}_n}$ if and only if $\text{Frob}_{\mathfrak{p}_n}\in H$. By \Cref{recurrence for t}, for each $n\geq 1$ we have
\begin{equation} \label{formula for t}
\hat{\textbf{t}}_n=\left(\prod_{\substack{i=2\\\text{Frob}_{\mathfrak{p}_i}\in H}}^n M(\text{Frob}_{\mathfrak{p}_i})\right)\hat{\textbf{t}}_1.
\end{equation}

Write $C_1,...,C_l$ for the conjugacy classes in $G$ that are contained in $H$ and for each $i$, fix a representative $\sigma_i$ for $C_i$. Further, for each $1\leq i \leq l$, define
\[m_i(n):=|\{2\leq j \leq n~~:~~\text{Frob}_{\mathfrak{p}_j}\in C_i\}|.\]
Since the group ring $\mathbb{C}[\boldsymbol \mu_p^r]$ is  commutative, the matrices $M(\sigma_i)$ all mutually commute and so we may group like terms in \Cref{formula for t} to obtain (cf. \Cref{conjugacy class remark})
\[\hat{\textbf{t}}_n=\left(\prod_{i=1}^lM(\sigma_i)^{m_{i}(n)}\right)\hat{\textbf{t}}_1.\]
Writing $||\cdot ||$ for the usual Euclidean norm on $V$ (with respect to the basis $\{e_\eta:\eta \in \boldsymbol \mu_p^r\}$), we have
\[||\hat{\textbf{t}}_n||=\left\|\left(\prod_{i=1}^lM(\sigma_i)^{m_{i}(n)}\right)\hat{\textbf{t}}_1\right\|\leq \left(\prod_{i=1}^l\left\|M(\sigma_i)^{m_i(n)}\right\|\right)||\hat{\textbf{t}}_1||. \]
By the Chebotarev density theorem each of the $m_i(n)$ tend to infinity with $n$, and since we have assumed there is at least one $i$ with $\epsilon(\sigma_i)=-1$, it follows from \Cref{limit of operator norm} that
\[\lim_{n\rightarrow \infty}||\hat{\textbf{t}}_n||=0.\]
That is, $\lim_{n\rightarrow \infty}\hat{t}_n(\eta)=0$ for each $\eta \in \boldsymbol \mu_p^r$.
\end{proof}

\begin{proof}[Proof of \Cref{independence theorem} (ii)]
Fix $\gamma \in \Gamma$ and write $\eta=f(\gamma)$. Then by \Cref{reduction to beta}, for all $X$ sufficiently large we have $s_X(\gamma)=t_X(\eta)$. It follows from \Cref{limit of t} that $\lim_{X\rightarrow \infty}t_X(\eta)=\frac{1}{2}$, from which the result follows.
\end{proof}

\section{Twisting data for abelian varieties (p=2)} \label{twisting abelian varieties section}

In this section let $K$ be a number field and $(A/K,\lambda)$ a principally polarised abelian variety. In the notation of sections 6-9, we take $p=2$ and $T=A[2]$ endowed with the Weil pairing $(~,~)_\lambda$.  Let $\Sigma$ be a finite set of places of $K$ containing all archimedean places, all places dividing $2$, and all places at which $A$ has bad reduction, so that $T$ is unramified outside $\Sigma$. 

We now  endow $T$ with a global metabolic structure and twisting data in such a way that for $\chi\in \mathcal{C}(K)$ the associated Selmer group $\textup{Sel}(A[2],\chi)$ agrees with the $2$-Selmer group $\textup{Sel}^2(A^\chi/K)$ of the quadratic twist of $A$ by $\chi$. For elliptic curves this is done in \cite[Section 5]{MR3043582}. Our definition of the global metabolic structure and twisting data will be a direct generalisation of theirs. The main difficulty is establishing \Cref{identification of quadratic forms under twisting} which for elliptic curves is \cite[Lemma 5.2]{MR3043582} and for Jacobians of odd degree hyperelliptic curves is \cite[Theorem 5.10]{YU15}.  We will deduce the general case from the results of \Cref{quadratic twisting theta groups} concerning the behaviour of certain Theta groups under quadratic twist.

\subsection{A global metabolic structure on $A[2]$}


%


For a place $v$ of $K$ write 
\[\delta_{v}:A(K_v)/2A(K_v)\hookrightarrow H^1(K_v,A[2])\]
for the connecting homomorphism in the multiplication-by-2 Kummer sequence.

\begin{defi} \label{defi of tate quadratic form}
Let $\mathcal{P}$ denote the Poincare line bundle on $A\times A^{\vee}$.
For each place $v$ of $K$ write $\mathcal{L}_v$ for the pull back of $\mathcal{L}=(1,\lambda)^*\mathcal{P}$ to a line bundle on $A/K_v$ and let $\mathscr{G}(\mathcal{L}_v)$  denote the associated Theta group. Then we define $q_{A,\lambda,v}$ to be the map
\[q_{A,\lambda,v}:H^1(K_v,A[2])\longrightarrow H^2(K_v,\bar{K}_v^{\times})=\textup{Br}(K_v)\stackrel{\textup{inv}_v}{\longrightarrow}\mathbb{Q}/\mathbb{Z}\] where the first map is the connecting homomorphism associated to the short exact sequence of $G_{K_v}$-modules
\begin{equation} \label{quad form seq}
0\rightarrow \bar{K}_v^{\times}\longrightarrow \mathscr{G}(\mathcal{L}_v)\longrightarrow A[2]\rightarrow 0
\end{equation} 
of \Cref{Galois action}.
\end{defi}

\begin{lemma} \label{props of local form}
Let $v$ be a place of $K$. Then
\begin{itemize}
\item[(i)] $q_{A,\lambda,v}$ is a quadratic form on $H^1(K_v,A[2])$ whose associated bilinear pairing is the local Tate pairing corresponding to $(~,~)_\lambda$,
\item[(ii)] the image  of $A(K_v)/2A(K_v)$ under $\delta_v$ is  a Lagrangian subspace of $H^1(K_v,A[2])$ with respect to $q_{A,\lambda,v}$.
\end{itemize}
In particular, $q_{A,\lambda,v}$ is a Tate quadratic form on $H^1(K_v,A[2])$ in the sense of \Cref{tate form defi}.
\end{lemma}

\begin{proof}
Part (i) is \cite[Corollary 4.7]{MR2833483} whilst Proposition 4.9 of op. cit. gives (ii).
\end{proof}

\begin{remark} \label{valued in Z/4Z}
In contrast to the case of elliptic curves, the quadratic form $q_{A,\lambda,v}$ in general takes values in $\frac{1}{4}\mathbb{Z}/\mathbb{Z}$ rather than just $\frac{1}{2}\mathbb{Z}/\mathbb{Z}$, which is the reason for allowing $\mathbb{Q}/\mathbb{Z}$-valued quadratic forms  in \Cref{tate form defi} rather than just those valued in $\mathbb{F}_2$. See \cite[Remark 4.16]{MR2833483} for an example of this phenomenon. 
\end{remark}

\begin{cor}
The collection $\mathbf{q}=(q_{A,\lambda,v})_v$ defines a global metabolic structure on $A[2]$.
\end{cor}

\begin{proof}
By \Cref{props of local form} (ii) $q_{A,\lambda,v}$ admits a Lagrangian subspace making $\left(H^1(K_v,T),q_{A,\lambda,v}\right)$ into a metabolic space for each place $v$ of $K$. Moreover,  if $v\notin \Sigma$ then $\textup{im}(\delta_{v})=H^1_{\textup{ur}}(K_v,A[2])$ (see e.g. \cite[Proposition 4.12]{MR2833483} and the preceeding remark). In particular, by \Cref{props of local form} (ii), $q_{A,\lambda,v}$ is unramified at each such place. 

Finally, let $\mathfrak{a}\in H^1(K,A[2])$. Write \[q:H^1(K,A[2])\rightarrow H^2(K,\bar{K}^\times)=\textup{Br}(K)\] for the connecting homomorphism associated to the sequence \cref{quad form seq} viewed over $K$ instead of $K_v$ (with $\mathcal{L}_v$ replaced by $\mathcal{L}:=(1,\lambda)^*\mathcal{P}$). Then $q(\mathfrak{a})\in \textup{Br}(K)$ and we have 
\[\sum_{v\in M_K}q_v(\mathfrak{a}_v)=\sum_{v\in M_K}\textup{inv}_vq(\mathfrak{a})=0,\]
 the last equality following from reciprocity for the Brauer group of $K$.
\end{proof}

\subsection{Twisting data associated to $A/K$}

We now define the twisting data $\boldsymbol{\alpha}$.

Fix a place $v$ of $K$ and $\chi\in \mathcal{C}(K_v)$, and let $(A^\chi,\psi)$ denote the quadratic twist of $A$ by $\chi$. By \Cref{descent of polarisation}, $\lambda_\chi:=(\psi^\vee)^{-1}\lambda \psi^{-1}$ is a principal polarisation on $A^\chi$, defined over $K_v$. In particular, associated to the pair $(A^\chi,\lambda_\chi)$ we have a quadratic form $q_{A^\chi,\lambda_\chi,v}$ on $H^1(K_v,A^\chi[2])$. 


\begin{lemma} \label{identification of quadratic forms under twisting}  \label{is a lagrangian subspace}
The isomorphism $H^1(K_v,A^\chi[2])\cong H^1(K_v,A[2])$ induced by $\psi$ identifies the quadratic forms $q_{A,\lambda,v}$ and $q_{A^\chi,\lambda_\chi,v}$. 
\end{lemma}

\begin{proof}
Take the long exact sequences for Galois cohomology associated to the commutative diagram of \Cref{main twist of theta groups lemma}.
\end{proof}

\begin{defi}
For $\chi \in \mathcal{C}(K_v)$ define $\alpha_v(\chi)\subseteq H^1(K_v,A[2])$ to be the image of the map  
\[A^\chi(K_v)/2A^\chi(K_v)\hookrightarrow H^1(K_v,A^\chi[2])\stackrel{\sim}{\longrightarrow}H^1(K_v,A[2])\] the first map arising from the multiplication-by-2 Kummer sequence for $A^\chi$ and the latter being induced by $\psi^{-1}$. 

By combining \Cref{identification of quadratic forms under twisting} with \Cref{props of local form} (ii) applied to $A^\chi/K_v$ we see that $\alpha_v(\chi)$ is a Lagrangian subspace of $H^1(K_v,A[2])$.
\end{defi}

As in \Cref{notation for intersection of twisted subspaces},
for $\chi_1$ and $\chi_2$ elements of $\mathcal{C}(K_v)$ we set
\[h_v(\chi_1,\chi_2)=\dim_{\mathbb{F}_p}\left(\alpha_v(\chi_1)/(\alpha_v(\chi_1)\cap \alpha_v(\chi_2))\right).\]

\begin{lemma} \label{intersection is norm}
For each quadratic character $\chi \in \mathcal{C}(K_v)$, let $L_\chi$ denote the extension of $K_v$ cut out by $\chi$. Then  \[h_v(\mathbbm{1},\chi_v)=\dim_{\mathbb{F}_2}A(K_v)/N_{L_\chi/K_v}A(L_\chi)\] where here $N_{L_\chi/K_v}:A(L_{\chi})\rightarrow A(K_v)$ is the  `local norm map' sending $P\in A(L_\chi)$ to \[N_{L_\chi/K_v}(P):=\sum_{\sigma \in \textup{Gal}(L_\chi/K_v)}\sigma(P).\]
\end{lemma}

\begin{proof}
This is shown in \cite[Proposition 5.2]{MR2373150}. Whilst loc. cit. is stated for the case of elliptic curves and for twists by characters of order $p>2$, the proof carries over unchanged to our case. See also \cite[Proposition 7]{MR597871}.
\end{proof}


The following Lemma evaluates the cokernel of the local norm map in certain cases.

\begin{lemma} \label{first cases of norm}
Let $v$ be a place of $K$ and $\chi\in \mathcal{C}(K_v)$. As above, let $L_\chi$ denote the extension of $K_v$ cut out by $\chi$.
\begin{itemize}
\item[(i)]  Suppose $v\nmid 2$ is nonarchimedean and that $A$ has good reduction at $v$. If $\chi$  is unramified  then 
\[\dim_{\mathbb{F}_2}A(K_v)/N_{L_\chi/K_v}A(L_\chi)=0.\] 
On the other hand, if $\chi$  is ramified then $N_{L_\chi/K_v}A(L_\chi)=2A(K_v)$ and, in particular, we have
\[\dim_{\mathbb{F}_2}A(K_v)/N_{L_\chi/K_v}A(L_\chi)=\dim_{\mathbb{F}_2}A(K_v)[2].\]
\item[(ii)] Suppose $v$ is archimedean and $\chi$ non-trivial. Then
\[\dim_{\mathbb{F}_2}A(K_v)/N_{L_\chi/K_v}A(L_\chi)=\dim_{\mathbb{F}_2}A(K_v)[2]-g\]
where $g=\textup{dim}A$ is the dimension of $A$. 
\end{itemize}
\end{lemma}

\begin{proof}
(i). The case where $\chi$ is unramified is a result of Mazur \cite[Corollary 4.4]{MR0444670}. For $\chi$ ramified the case where $A$ is an elliptic curve is \cite[Lemma 5.5 (ii)]{MR2373150} and the argument for general abelian varieties is identical. 

(ii). By assumption $L_\chi/K_v$ is the extension $\mathbb{C}/\mathbb{R}$. Now since $A/K_v$ is an abelian variety of dimension $g$ over the reals, we have an isomorphism of real Lie groups
\begin{equation} \label{real lie groups}
A(K_v)\cong\left(\mathbb{R}/\mathbb{Z}\right)^{g}\times\left(\mathbb{Z}/2\mathbb{Z}\right)^{m}
\end{equation}
for some $0\leq m\leq g$ (see, for example, \cite[Proposition 1.9~and~Remark 1.12]{Silhol}). Now $N_{L_\chi/K_v}$ is a continuous map from the connected group $A(L_\chi)$ to $A(K_v)$ (for the complex and real topologies respectively) and it follows that the image of $N_{L_\chi/K_v}$ is contained in the connected component of the identity in $A(K_v)$, which we denote $A^{0}(K_v)$. Under the isomorphism \cref{real lie groups}, $A^{0}(K_v)$ is the factor corresponding to $\left(\mathbb{R}/\mathbb{Z}\right)^{g}$. On the other hand, we have $2A(K_v)\subseteq N_{L_\chi/K_v} A(L_\chi)$ and we see again from \cref{real lie groups} that multiplication by $2$ is surjective on $A^{0}(K_v)$. Thus  $N_{L_\chi/K_v} A(L_\chi)=A^{0}(K_v)$.  Appealing to \cref{real lie groups} one last time we obtain $|A(K_v)/N_{L_\chi/K_v} A(L_\chi) |=2^{-g}|A(K_v)[2]|$.
\end{proof}

\begin{proposition} \label{equivalence of selmer group definitions}
The collection of maps $\boldsymbol{\alpha}=(\alpha_v)_v$ defines twisting data with respect to $(A[2],\textbf{q},\Sigma)$. Moreover,  we have
\[\textup{Sel}(A[2],\chi)\cong\textup{Sel}^2(A^\chi/K)\]
where $\textup{Sel}(A[2],\chi)$ is defined with respect to  $(A[2],\textbf{q},\Sigma,\boldsymbol \alpha)$ as in \Cref{twisted selmer groups defi}.
\end{proposition}

\begin{proof}
Note that since $p=2$, the group $\mathcal{F}(K_v)$ appearing in the definition of twisting data (\Cref{twisting data}) is equal to $\mathcal{C}(K_v)$. For each place $v$ of $K$ and $\chi_v\in \mathcal{C}(K_v)$, the subspace $\alpha_v(\chi_v)$ of $H^1(K_v,A[2])$ is Lagrangian by \Cref{is a lagrangian subspace}. Moreover, if $v\notin \Sigma$ and $\chi_v$ is ramified then $\alpha_v(\chi_v)$ is an element of $\mathcal{H}_\textup{ram}(q_v)$. Indeed, by definition we need to show that $\alpha_v(\chi_v)\cap H^1_{\textup{ur}}(K_v,A[2])=0$. As before, as $v\notin \Sigma$ we have \[H^1_{\textup{ur}}(K_v,A[2])=\delta_v(A(K_v)/2A(K_v))=\alpha_v(\mathbbm{1}_v).\] Combining \Cref{intersection is norm} with \Cref{first cases of norm} gives
\[\dim_{\mathbb{F}_2}\left(\alpha(\mathbbm{1}_v)/\alpha_v(\chi_v)\cap \alpha_v(\mathbbm{1}_v)\right)=\dim_{\mathbb{F}_2}A(K_v)/2A(K_v)=\dim_{\mathbb{F}_2}\alpha(\mathbbm{1}_v)\] whence $\alpha_v(\chi_v)\cap \alpha_v(\mathbbm{1}_v)=0$ as desired. Thus $\boldsymbol \alpha$ defines twisting data. 

Finally, we will show that for $\chi\in \mathcal{C}(K)$ the associated Selmer group $\textup{Sel}(A[2],\chi)$ agrees with the classical Selmer group $\textup{Sel}^2(A^\chi/K)$. 

Now by the definition of $\textup{Sel}^2(A^\chi/K)$ and the maps $\alpha_v$, we have
\[\textup{Sel}^2(A^\chi/K) =\{\mathfrak{a}\in H^1(K,A[2])~~:~~\mathfrak{a}_v\in \alpha_v(\chi_v)~~\textup{for all }v\in M_K\}.\] On the other hand, we have 
\[\textup{Sel}(A[2],\chi) = \{\mathfrak{a}\in H^1(K,A[2])~~:~~\mathfrak{a}_v\in H^1_{\mathcal{S}(\chi)}(K_v,A[2])~~\textup{for all }v\in M_K\}\]
where, as in \Cref{twisted selmer groups defi}, $H^1_{\mathcal{S}(\chi)}(K_v,A[2])=\alpha(\chi_v)$ if $v\in \Sigma$ or $\chi_v$ is ramified at $v$, and is equal to $H^1_{\textup{ur}}(K_v,A[2])$ otherwise.

In particular, to show that $\textup{Sel}(A[2],\chi)=\textup{Sel}^2(A^\chi/K)$ it suffices to show that $\alpha(\chi_v)=H^1_{\textup{ur}}(K_v,A[2])$ whenever $v\notin \Sigma$ and $\chi_v$ is unramified. But for such places we have $\alpha(\mathbbm{1}_v)=H^1_{\textup{ur}}(K_v,A[2])$ and since $\chi_v$ is unramified \Cref{first cases of norm} (i) gives $h(\mathbbm{1}_v,\chi_v)=0$. It now follows immediately that $\alpha(\chi_v)=H^1_{\textup{ur}}(K_v,A[2])$ as desired.
\end{proof}

\subsection{Main theorems for $2$-Selmer ranks} \label{Main theorems for $2$-Selmer ranks}

Having interpreted the groups $\textup{Sel}^2(A^\chi/K)$ as those arising from twisting data we apply the results of the previous  sections to deduce results about abelian varieties. 

 The following generalises a theorem of Kramer \cite[Theorem 1]{MR597871} for elliptic curves and Yu \cite[Theorem 5.11]{YU15}  for odd degree hyperelliptic curves. 

\begin{theorem} \label{difference of selmer groups for abelian varieties}
Let $K$ be a number field, $\chi$ a quadratic character of $K$ corresponding to the extension $L/K$, and $A/K$ a principally polarised abelian variety. Then 
\[\dim_{\mathbb{F}_2}\textup{Sel}^2(A^\chi/K)\equiv \dim_{\mathbb{F}_2}\textup{Sel}^2(A/K)+\sum_{v\in M_K}\dim_{\mathbb{F}_2}A(K_v)/N_{L_w/K_v}A(L_w)~~\textup{ (mod }2\textup{)}\] (here $w$ denotes any place of $L$ extending $v$). 
\end{theorem}

\begin{proof}
Combine \Cref{local decomposition}, \Cref{equivalence of selmer group definitions} and \Cref{intersection is norm}.
\end{proof}

\begin{theorem}[=\Cref{theorem}] \label{main selmer ranks explicit theorem}
Let $K$ be a number field, $A/K$ a principally polarised abelian variety and  $\Sigma$ the set consisting of all archimedean places of $K$, all places of bad reduction for $A$, and all places dividing $2$. Define $\epsilon:\textup{Gal}(K(A[2])/K)\rightarrow \{\pm 1\}$ by $\sigma \mapsto (-1)^{\dim_{\mathbb{F}_2}A[2]^\sigma}$. 

\begin{itemize}
\item[(i)] If $\epsilon$ fails to be a homomorphism then for all sufficiently large $X$ \[\frac{|\{\chi \in \mathcal{C}(K,X)~:~\textup{dim}_{\mathbb{F}_2}\textup{Sel}^2(A^\chi/K)~~\textup{is even}\}|}{|\mathcal{C}(K,X)|}=1/2.\]
\item[(ii)] If $\epsilon$ is a homomorphism, let $K(\sqrt{\Delta})/K$ be the fixed field of the kernel of $\epsilon$. For each $v\in \Sigma$ and quadratic character $\chi \in \mathcal{C}(K_v)$ write $L_\chi/K_v$ for the extension cut out by $\chi$ and define
\[\omega_v(\chi):=\chi(\Delta)(-1)^{\dim_{\mathbb{F}_2}A(L_\chi)/N_{L_\chi/K_v}A(L_\chi)}.\] 
Finally, define
\[\delta_v:=\frac{1}{|\mathcal{C}(K_v)|}\sum_{\chi \in \mathcal{C}(K_v)}\omega(\chi)\phantom{hello}\textup{and}\phantom{hello}\delta:=\prod_{v\in \Sigma}\delta_v.\]
Then for all sufficiently large $X$, 
\[\frac{|\{\chi \in \mathcal{C}(K,X)~:~\textup{dim}_{\mathbb{F}_2}\textup{Sel}^2(A^\chi/K)~~\textup{is even}\}|}{|\mathcal{C}(K,X)|}=\frac{1+(-1)^{\textup{dim}_{\mathbb{F}_2}\textup{Sel}^2(A/K)}\cdot\delta}{2}.\]
\end{itemize}
\end{theorem}

\begin{proof}
Combine \Cref{equivalence of selmer group definitions},  \Cref{main twisting theorem combined} and \Cref{intersection is norm}.
\end{proof}

\begin{remark}
\Cref{first cases of norm} (ii) enables one to evaluate the local terms $\delta_v$ for archimedean places. For nonarchimedean places of odd residue characteristic, the dimension of the cokernel of the norm map may be expressed in terms of Tamagawa numbers: see  \cite[Lemma 2.5]{morgan}.
\end{remark}

In the following examples we examine when $\epsilon$ is (resp. is not) a homomorphism for certain families of abelian varieties. 

\begin{example}[Generic $2$-torsion]
For any principally polarised abelian variety $A/K$ of dimension $g$, $\textup{Gal}(K(A[2])/K)$ is a subgroup of the symplectic group $\textup{Sp}_{2g}(\mathbb{F}_2)$. As in \Cref{remark on the function epsilon}, if $g\geq 2$ and $\textup{Gal}(K(A[2])/K)\cong \textup{Sp}_{2g}(\mathbb{F}_2)$ then $\epsilon$ is not a homomorphism.
\end{example}

\begin{example}[Elliptic curves]
Suppose that $A/K$ is an elliptic curve, say given by a Weierstrass equation of the form $y^2=f(x)$ for some monic (separable) cubic polynomial $f(x)$. Then $\textup{Gal}(K(A[2])/K)=\textup{Gal}(f)$ is the Galois group of the splitting field of $f(x)$ and as such may be viewed as a subgroup of the symmetric group $S_3$. One readily checks that the map $\sigma \mapsto (-1)^{\dim_{\mathbb{F}_2}A[2]^\sigma}$ is  the sign homomorphism. Thus $\epsilon$ is always a homomorphism and we may take $\Delta$ to be the discriminant  of the elliptic curve. Thus \Cref{main selmer ranks explicit theorem} recovers \cite[Theorem A]{MR3043582}. See Proposition 7.9 of op. cit. for a table computing the local terms $\delta_v$ as a function of the reduction of the elliptic curve.
\end{example}

\begin{example}[Hyperelliptic curves] \label{hyperelliptic curves example}
Let $C/K$ be a hyperelliptic curve of genus $g\geq 2$, say given by a Weierstrass equation $y^2=f(x)$ for a (separable, not necessarily monic) polynomial $f(x)$ with $\textup{deg}(f)\in \{2g+1,2g+2\}$. Take $A/K$ to be the Jacobian of $C$ so that $A/K$ is a principally polarised abelian variety of dimension $g$. Then again $\textup{Gal}(K(A[2])/K)=\textup{Gal}(f)$ which we view as a subgroup of the symmetric group $S_{\textup{deg}(f)}$. Write $\textup{sgn}:\textup{S}_n\rightarrow \{\pm 1\}$ for the sign homomorphism and fix $\sigma \in \textup{Gal}(f)$ with cycle type $(d_1~...~d_s)$. Then we have
\begin{equation}\label{2-torsion in hyp curve}
\epsilon(\sigma)=\begin{cases} -\textup{sgn}(\sigma)~~&~~\textup{all }d_i\textup{ even and }\textup{deg}(f)\equiv 2~~\textup{(mod }4\textup{)},\\\textup{sgn}(\sigma)~~&~~\textup{else.}\end{cases}
\end{equation}
Indeed, this follows from \cite[Theorem 1.4]{MR1865865} (whilst loc. cit. is stated for hyperelliptic curves over finite fields of odd residue characteristic, the proof yields the above statement for all fields of characteristic not $2$; note also the erratum \cite{MR2169307}). 

Suppose now that either $g$ is odd or $\textup{deg}(f)$ is odd. Then by \cref{2-torsion in hyp curve} $\epsilon$ is always a homomorphism and again we may take $\Delta$ to be the discriminant of the hyperelliptic curve $C$. In particular, the case $\textup{deg}(f)$ odd recovers \cite[Theorem 1]{YU15}. 

Now suppose that both $g$  and $\textup{deg}(f)$ are even, or equivalently $\textup{deg}(f)\equiv 2~\textup{(mod }4\textup{)}$. Suppose further that either $\textup{Gal}(f)\cong S_{2g+2}$ or $\textup{Gal}(f)\cong A_{2g+2}$. Then by \cref{2-torsion in hyp curve} we see that $\epsilon$ is not a homomorphism (indeed, the only non-trivial homomorphism from $S_{2g+2}$ to $\{\pm 1\}$ is sgn  yet \cref{2-torsion in hyp curve} shows that $\epsilon$ is non-trivial when restricted to $A_{2g+2}$). 
\end{example}

\begin{example}[Abelian varieties with principal polarisation induced by a rational symmetric line bundle] \label{symm line bundle example}
Suppose that $(A/K,\lambda)$ is a principally polarised abelian variety and that the polarisation $\lambda$ is induced by a rational (i.e. $G_K$-invariant) symmetric line bundle $\mathcal{L}$. Then the associated quadratic refinement $q_\mathcal{L}$ of the Weil-pairing $(~,~)_\lambda$ on $A[2]$ (as in \Cref{defi of quadratic refinement cor line bundle}) is $G_K$-invariant also, whence $\textup{Gal}(K(A[2])/K)$ acts on $A[2]$ through the orthogonal group $O(q_\mathcal{L})$. Then $\epsilon$ is the Dickson homomorphism $d_{q_\mathcal{L}}$ (\Cref{dimension proposition}). We remark that this case includes both elliptic curves and Jacobians of hyperelliptic curves of either odd degree or odd genus: see \cite[Proposition 3.11]{MR2915483}.
\end{example}

\subsection{Main theorems for $2^\infty$-Selmer ranks}

We now incorporate the results of Section $5$ to move from $2$-Selmer ranks  to $2^\infty$-Selmer ranks. 


\begin{theorem} \label{2-infinity selmer decomp}
Let $K$ be a number field and $(A/K,\lambda)$ a principally polarised abelian variety. Let $\Sigma$ be the set consisting of all archimedean places of $K$, all places of bad reduction for $A$, and all places dividing $2$, and let $L/K$ be a quadratic extension with associated quadratic character $\chi$. 

Then 

\[\textup{rk}_2(A/L)\equiv ~~~~~~\phantom{hi}\sum_{\mathclap{\substack{v\in \Sigma \\v~\textup{ non-split in }L/K }}}\phantom{hi}~~~~~~\left(2~\textup{inv}_v~\mathfrak{g}(A/K_v,\lambda_v,\chi_v)+\dim_{\mathbb{F}_2}A(K_v)/N_{L_w/K_v}A(L_w)\right)~~\textup{ (mod }2\textup{)}\]
where the local terms $\mathfrak{g}(A/K_v,\lambda_v,\chi_v)\in \textup{Br}(K_v)[2]$ are given in \Cref{the local terms part 2} and $w$ denotes any place of $L$ extending $v$. 
\end{theorem}

\begin{proof}
First note that $\text{rk}_2(A/L)=\text{rk}_2(A/K)+\text{rk}_2(A^\chi/K)$. Moreover, we have
\[\text{dim}_{\mathbb{F}_2}\text{Sel}^2(A/K)=\text{rk}_2(A/K)+\text{dim}_{\mathbb{F}_2}A(K)[2]+\text{dim}_{\mathbb{F}_2}\Sha_\textup{nd}(A/K)[2]\] and the analagous equality for $A^\chi/K$. Noting that $\text{dim}_{\mathbb{F}_2}A(K)[2]=\text{dim}_{\mathbb{F}_2}A^\chi(K)[2]$ the above observations combine to give
\[\textup{rk}_2(A/L)~~\equiv~~ \text{dim}_{\mathbb{F}_2}\text{Sel}^2(A/K)+\text{dim}_{\mathbb{F}_2}\text{Sel}^2(A^\chi/K)\]
\[\textup{ }\textup{ }\textup{ }\textup{ }\textup{ }\textup{ }\textup{ }\textup{ }\textup{ }\textup{ }\textup{ }\textup{ }\textup{ }\textup{ }\textup{ }\textup{ }\textup{ }\textup{ }\textup{ }\textup{ }\textup{ }\textup{ }\textup{ }\textup{ }\textup{ }\textup{ }\textup{ }\textup{ }\textup{ }\textup{ }\textup{ }\textup{ }\textup{ }\textup{ }\textup{ }\textup{ }\textup{ }\textup{ }\textup{ }\textup{ }\textup{ }\textup{ }+\text{dim}_{\mathbb{F}_2}\Sha_\textup{nd}(A/K)[2]+\text{dim}_{\mathbb{F}_2}\Sha_\textup{nd}(A^\chi/K)[2]~~\textup{ (mod }2\textup{).}\]

Combining \Cref{difference of selmer groups for abelian varieties} with \Cref{main sha decomp theorem} then gives
\[\textup{rk}_2(A/L)\equiv \sum_{v\in M_K}\left(2~\textup{inv}_v\mathfrak{g}(A/K_v,\lambda_v,\chi_v)+\dim_{\mathbb{F}_2}A(K_v)/N_{L_w/K_v}A(L_w)\right)~~\textup{ (mod }2\textup{)}.\]

Finally, combining \Cref{local terms in sha computation} (iii) with \Cref{first cases of norm} shows that \[2~\textup{inv}_v\mathfrak{g}(A/K_v,\lambda_v,\chi_v)+\dim_{\mathbb{F}_2}A(K_v)/N_{L_w/K_v}A(L_w)\equiv 0 ~~\textup{ (mod }2\textup{)}\] for each place $v\notin \Sigma$, and similarly for each place $v\in \Sigma$ which split in $L/K$.
\end{proof}

We now prove \Cref{2-infinity selmer rank main theorem}, after first defining the local terms appearing in the statement.

\begin{defi} \label{delta for 2-infinity}
Let $K$ be a number field, $(A/K,\lambda)$ a principally polarised abelian variety and let $\Sigma$ denote the set consisting of all archimedean places of $K$, all places of bad reduction for $A$, and all places dividing $2$. 

For each $v\in \Sigma$ and $\chi \in \mathcal{C}(K_v)$ define
\[\Omega_v(\chi):=(-1)^{2~\textup{inv}_v\mathfrak{g}(A/K_v,\lambda_v,\chi)+\dim_{\mathbb{F}_2}A(K_v)/N_{L_\chi/K_v}A(L_\chi) }\]
where here $L_\chi$ is the extension of $K_v$ cut out by $\chi$. Further, we define (for each $v\in \Sigma$)
\[\kappa_v=\frac{1}{|\mathcal{C}(K_v)|}\sum_{\chi \in \mathcal{C}(K_v)}\Omega_v(\chi)\phantom{hello}\textup{and}\phantom{hello}\kappa=\prod_{v\in \Sigma}\kappa_v.\]
\end{defi}

\begin{remark} \label{archimedean places evaluation of delta}
If $v$ is archimedean then by \Cref{2-infinity selmer decomp} we have 
\[\Omega_v(\chi_v)=\begin{cases} 1~~&~~\chi_v ~\textup{trivial} \\ (-1)^{\textup{dim}A}~~&~~\textup{else}. \end{cases}\]
In particular, if $v$ is a real place and $\textup{dim}A$ is odd then $\kappa_v=0$ (hence also $\kappa=0$), whilst if $v$ is complex or $\textup{dim}A$ is even, we have $\kappa_v=1$.
\end{remark}

\begin{theorem} \label{2-infinity selmer rank main theorem expanded}
Let  $A/K$ be a principally polarised abelian variety. Then for all sufficiently large $X>0$, 
\[\frac{|\{\chi \in \mathcal{C}(K,X)~:~\textup{rk}_2(A^\chi/K)~~\textup{is even}\}|}{|\mathcal{C}(K,X)|}=\frac{1+(-1)^{\textup{rk}_2(A/K)}\cdot\kappa}{2}\]
\end{theorem}

\begin{proof}
As noted previously, for any $\chi \in \mathcal{C}(K)$, corresponding to the quadratic extension $L/K$, we have
\[\textup{rk}_2(A/L)=\textup{rk}_2(A/K)+\textup{rk}_2(A^\chi/K).\]
Thus for each $\chi \in \mathcal{C}(K)$, \Cref{2-infinity selmer decomp} gives
\[(-1)^{\textup{rk}_2(A^\chi/K)}=(-1)^{\text{rk}_2(A/K)}\prod_{v\in \Sigma}\Omega(\chi_v)\]
with $\Omega(\chi_v)\in \{\pm 1\}$ depending only on the restriction of $\chi$ to $K_v$. The argument is now identical to that in the proof of \Cref{epsilon homomorphism thm}. As is the case there, `sufficiently large $X>0$' means that we require only that $X$ is large enough that the restriction homomorphism from $\mathcal{C}(K,X)$ to $\prod_{v\in \Sigma}\mathcal{C}(K_v)$ is surjective. 
\end{proof}

The following example shows that the proportion of twists having even $2$-Selmer rank can differ from the proportion having even $2^\infty$-Selmer rank.

\begin{example} \label{main example of parities}
Consider the genus $2$ hyperelliptic curve $C:y^2=x^6+x^4+x+3$ over $\mathbb{Q}$. The polynomial $f(x)=x^6+x^4+x+3$ has Galois group $S_6$. By \Cref{main selmer ranks explicit theorem} (see also \Cref{hyperelliptic curves example}) the $2$-Selmer ranks are distributed 1/2-1/2 amongst even/odd in the quadratic twist family of the Jacobian $J/K$ of $C$. 

On the other hand, we claim that $\kappa=3/16$ so that $19/32$ of the twists of $J$ have even $2^{\infty}$-Selmer rank whilst $13/32$ have odd $2^{\infty}$-Selmer rank. The discriminant of $f(x)$ is $-5\cdot 2670719$, so $J/K$ has good reduction away from $2,5$ and $2670719$. Thus we have $\Sigma=\{2,5,2670719,\infty\}$. Using the computer algebra package \textsc{magma} (\cite{MR1484478}), one computes that $\textup{rk}_2(J/K)$ is odd. By \Cref{archimedean places evaluation of delta}, $\kappa_\infty=1$. To compute $\kappa_2$, $\kappa_5$ and $\kappa_{2670719}$, one may use the following trick. By \Cref{2-infinity selmer decomp} and the above discussion, for a quadratic character $\chi$ of $\mathbb{Q}$ corresponding to the extension $L/\mathbb{Q}$, one has
\smallskip
\begin{equation} \label{trick} (-1)^{\textup{rk}_2(J^\chi/\mathbb{Q})}=-~~~~~~\phantom{hi}\prod_{\mathclap{\substack{v\in \{2,5,2670719\} \\v~\textup{ non-split in }L }}}\phantom{hi}~~~~~~\Omega_v(\chi_v).
\end{equation}

Now for $0\neq n\in \mathbb{Z}$, the quadratic twist of $J$ by $\mathbb{Q}(\sqrt{n})/\mathbb{Q}$ is the Jacobian of the hyperelliptic curve $y^2=nf(x)$. Thus one may use  \textsc{magma} to compute $(-1)^{\textup{rk}_2(J^\chi/\mathbb{Q})}$ for various (finitely many) quadratic characters $\chi$, from which one may then determine all the  $\Omega_v(\chi_v)$ by \cref{trick}. Upon doing this one obtains $\kappa_2=3/4, \kappa_5=-1/2$ and $\kappa_{2670719}=1/2$ and the claim follows. 
\end{example}

\begin{remark} \label{independent variables remark}
Since by \Cref{2-infinity selmer decomp} the parity of $\textup{rk}_2(A^\chi/K)$ depends only on the restriction of $\chi$ to the archimedean places, the places of bad reduction for $A$, and the places over $2$, it follows from \Cref{independence theorem} (i)  (along with \Cref{equivalence of selmer group definitions}) that when $\epsilon$ fails to be a homomorphism we in fact have 
\[\frac{|\{\chi \in \mathcal{C}(K,X)~:~\textup{dim}_{\mathbb{F}_2}\textup{Sel}^2(A^\chi/K)~~\textup{is even and }\textup{rk}_2(A^\chi/K)~~\textup{is even}\}|}{|\{\chi \in \mathcal{C}(K,X)~:~\textup{rk}_2(A^\chi/K)~~\textup{is even}\}|}=1/2\]
for all sufficiently large $X$, assuming the denominator is non-zero, and that the same holds when we condition on $\textup{rk}_2(A^\chi/K)$ being odd also. Thus when $\epsilon$ fails to be a homomoprhism the parities of Selmer ranks and the parities of $2$-infinity Selmer ranks behave `independently'. 
\end{remark}

\subsection{The proportion of twists having non-square Shafarevich--Tate group}

We end this section by giving an analogue of \Cref{theorem} for $\textup{dim}_{\mathbb{F}_2}\Sha_{\textup{nd}}(A/K)[2]$ rather than for $\textup{dim}_{\mathbb{F}_2}\textup{Sel}^2(A/K)$. Since the Shafarevich--Tate group of a principally polarised abelian variety, if finite, has square order if and only if $\textup{dim}_{\mathbb{F}_2}\Sha_{\textup{nd}}(A/K)[2]$ is even  (see e.g. \cite[Theorem 8]{MR1740984}), this may be viewed as quantiftying the failure of the Shafarevich--Tate group to have square order in quadratic twist families. The proof of the theorem is identical to its analogue for $2$-Selmer ranks, so we only sketch the proof.

\begin{theorem}\label{sha statistics theorem}
Let $K$ be a number field, $(A/K,\lambda)$ a principally polarised abelian variety and  $\Sigma$ the set consisting of all archimedean places of $K$, all places of bad reduction for $A$, and all places dividing $2$. Define $\epsilon:\textup{Gal}(K(A[2])/K)\rightarrow \{\pm 1\}$ by $\sigma \mapsto (-1)^{\dim_{\mathbb{F}_2}A[2]^\sigma}$. 

\begin{itemize}
\item[(i)] If $\epsilon$ fails to be a homomorphism then for all sufficiently large $X$ \[\frac{|\{\chi \in \mathcal{C}(K,X)~:~\textup{dim}_{\mathbb{F}_2}\Sha_{\textup{nd}}(A^\chi/K)[2]~~\textup{is even}\}|}{|\mathcal{C}(K,X)|}=1/2.\]
\item[(ii)] If $\epsilon$ is a homomorphism, let $K(\sqrt{\Delta})/K$ be the fixed field of the kernel of $\epsilon$. For each $v\in \Sigma$ and quadratic character $\chi \in \mathcal{C}(K_v)$ write $L_\chi/K_v$ for the extension cut out by $\chi$ and define
\[\Upsilon_v(\chi):=\chi(\Delta)(-1)^{2~\textup{inv}_v\mathfrak{g}(A/K_v,\lambda_v,\chi_v)}.\] 
Finally, define
\[\rho_v:=\frac{1}{|\mathcal{C}(K_v)|}\sum_{\chi \in \mathcal{C}(K_v)}\Upsilon(\chi)\phantom{hello}\textup{and}\phantom{hello}\rho:=\prod_{v\in \Sigma}\rho_v.\]
Then for all sufficiently large $X$, 
\[\frac{|\{\chi \in \mathcal{C}(K,X)~:~\textup{dim}_{\mathbb{F}_2}\Sha_{\textup{nd}}(A^\chi/K)[2]~~\textup{is even}\}|}{|\mathcal{C}(K,X)|}=\frac{1+(-1)^{\textup{dim}_{\mathbb{F}_2}\Sha_{\textup{nd}}(A/K)[2]}\cdot\rho}{2}.\]
\end{itemize}
\end{theorem}

\begin{proof}
Fix a quadratic character $\chi$ of $K$. As in \Cref{definition of w}, set \[w(\chi):=\prod_{v\notin \Sigma,~\chi_v~\text{ram}}(-1)^{\dim_{\mathbb{F}_p}A(K_v)[2]}=\prod_{v\notin \Sigma,~\chi_v~\text{ram}} \epsilon(\textup{Frob}_v).\]

Combining \Cref{main sha decomp theorem} with \Cref{local terms in sha computation} we obtain 
\begin{equation} \label{random eq}
(-1)^{\textup{dim}_{\mathbb{F}_2}\Sha_{\textup{nd}}(A^\chi/K)[2]}=w(\chi)(-1)^{\textup{dim}_{\mathbb{F}_2}\Sha_{\textup{nd}}(A/K)[2]}\prod_{v\in \Sigma}(-1)^{2~\textup{inv}_v\mathfrak{g}(A/K_v,\lambda_v,\chi_v)}.
\end{equation}

If $\epsilon$ is a homomorphism then, as in the proof of \Cref{first parity lemma}, we have $w(\chi)=\prod_{v\in\Sigma}\chi_v(\Delta)$, whence \[(-1)^{\textup{dim}_{\mathbb{F}_2}\Sha_{\textup{nd}}(A^\chi/K)[2]}=(-1)^{\textup{dim}_{\mathbb{F}_2}\Sha_{\textup{nd}}(A/K)[2]}\prod_{v\in \Sigma}\Upsilon_v(\chi)\]
and the same argument as in the proof of \Cref{epsilon homomorphism thm} gives the result.

On the other hand, suppose that $\epsilon$ is a homomorphism and enlarge $\Sigma$ if neccesary so that \Cref{assumptions on places} holds, noting that \cref{random eq} still remains true. The result now follows from \Cref{independence theorem} (cf. proof of \Cref{epsilon non-trivial twisting theorem}).
\end{proof}

\section{Twisting data for abelian varieties (p>2)} \label{odd p twisting data}

As in the previous section, let $K$ be a number field and $(A/K,\lambda)$ a principally polarised abelian variety. This time we take $p$ to be an odd prime and $T=A[p]$. As with $A[2]$ in the previous section, we endow $T$ with a canonical global metabolic structure and twisting data so that the resulting Selmer groups have a classical interpretation. For elliptic curves this is done by Klagsbrun--Mazur--Rubin in \cite[Section 5]{MR3043582}. This time, the case of an arbitrary principally polarised abelian variety is almost identical to that of loc. cit. though to fix notation we repeat the relevant material. 

\subsection{The global metabolic structure on $A[p]$}

As with the case $p=2$, the polarisation $\lambda$ along with the Weil-pairing
\[(~,~)_{e_p}:A[p]\times A^{\vee}[p]\]
provides the desired (non-degenerate, alternating, $G_K$-equivariant) bilinear pairing
\[(~,~)_\lambda:T\times T\rightarrow \boldsymbol \mu_p\]
(defined by setting $(x,y)_\lambda=(x,\lambda(y))_{e_p}$ for $x,y\in T$). We take $\Sigma$ to be a finite set of places of $K$ containing all archimedean places, all primes over $p$, and all primes at which $A$ has bad reduction. Then $T$ is unramified outside $\Sigma$. 


Since $p$ is odd, the quadratic forms $q_v=\frac{1}{2}\left \langle ~,~\right \rangle_v$ (here $v$ a place of $K$ and $\left \langle ~,~\right \rangle_v$ denotes the local Tate pairing associated to $(~,~)_\lambda$) are Tate quadratic forms which endow $T$ with a global metabolic structure $\textbf{q}$ (cf. \Cref{local tate pairing subsection}). 

\subsection{Twisting data associated to $A[p]$} 

Here we  associate canonical twisting data to $(A[p],\Sigma,\textbf{q})$.

\begin{defi} \label{definition of the odd p twist}
Let $\chi\in \mathcal{C}(K)$ be non-trivial and let $L$ denote the associated cyclic $p$-extension $L=\bar{K}^{\textup{ker}(\chi)}$ of $K$. We write $A^\chi$ for the abelian variety denoted $A_L$ in \cite[Definition 5.1]{MR2331769}, so that $A^\chi/K$ is an abelian variety of dimension $(p-1)\textup{dim}A$ which may be defined as the kernel of the `norm' homomorphism $\textup{Res}_{L/K}A\rightarrow A$ (here $\textup{Res}_{L/K}A$ denotes the restriction of scalars of $A$ from $L$ to $K$). 

By \cite[Theorem 5.5 (iv)]{MR2331769}, $\chi$ induces an inclusion of $\mathbb{Z}[\boldsymbol \mu_p]$ into $\textup{End}_K(A^\chi)$. Moreover, by Theorem 2.2 (iii) of op. cit. we have a canonical isomorphism $\psi:A[p]\stackrel{\sim}{\longrightarrow}A^\chi[\mathfrak{p}]$ where $\mathfrak{p}$  denotes the unique prime of $\mathbb{Z}[\boldsymbol \mu_p]$ lying over $p$.

If $\mathbbm{1}_{K_v}\neq \chi\in \mathcal{C}(K_v)$ for some place of $K$ then we define $A^{\chi}/K_v$ similarly. 
\end{defi}

\begin{remark} \label{the pi selmer groups}
Fix $\chi \in \mathcal{C}(K)$ non-trivial, and let $\pi$ be a generator of the prime $\mathfrak{p}$ of $\mathbb{Z}[\boldsymbol \mu_p]$ lying over $p$. View $\pi$ inside $\textup{End}_K(A^\chi)$ as above. Then $\pi$ is an isogeny and we have an associated $\pi$-Selmer group
\[\textup{Sel}^\pi(A^\chi/K)=\{\mathfrak{a}\in H^1(K,A^\chi[\mathfrak{p}])~~:~~\mathfrak{a}_v\in \textup{im}(\delta_v)~~\forall v\in M_K\},\]
where here for each place $v$ of $K$, $\delta_v:A^\chi(K_v)/\pi A^\chi(K_v)\hookrightarrow H^1(K_v,A^\chi[\mathfrak{p}])$ is the connecting homomorphism associated to the multiplication-by-$\pi$ Kummer sequence for $A^\chi/K_v$. 

One checks that $\textup{Sel}^\pi(A^\chi/K)$ does not depend on the choice of generator $\pi$ for $\mathfrak{p}$.
\end{remark}

We now define the twisting data. 

\begin{defi}
Let $v$ be a place of $K$ and $\chi \in \mathcal{C}(K_v)$. Define $\alpha_v(\chi)\subseteq H^1(K_v,A[p])$ as follows:
\begin{itemize}
\item[(i)]  If $\chi$ is trivial, define $\alpha_v(\chi)$ to be the image of $A(K)/pA(K)$ under the connecting homomorphism assoicated to the multiplication-by-$p$ Kummer sequence for $A/K_v$,
\item[(ii)] If $\chi$ is non-trivial, let $\pi$ be a generator of the prime $\mathfrak{p}$ of $\mathbb{Z}[\boldsymbol \mu_p]$ lying over $p$. Then we define $\alpha_v(\chi)$ to be the image of $A^\chi(K_v)/\pi A^\chi(K_v)$ under the composition
\[A^\chi(K_v)/\pi A^\chi(K_v)\stackrel{\delta_v}{\longrightarrow}H^1(K_v,A^\chi[\mathfrak{p}])\stackrel{\sim}{\longrightarrow}H^1(K_v,A[p]),\]
where  the rightmost map is induced by the isomorphism $\psi:A[p]\stackrel{\sim}{\longrightarrow}A[\mathfrak{p}]$ of \Cref{definition of the odd p twist}.  One sees easily that $\alpha_v(\chi)$ does not depend on the choice of $\pi$, and depends only on the extension cut out by $\chi$. 
\end{itemize}
\end{defi}

As usual, for $v$ a place of $K$ and $\chi_1,\chi_2\in \mathcal{C}(K_v)$, write 
\[h_v(\chi_1,\chi_2)=\dim_{\mathbb{F}_p}\left(\alpha_v(\chi_1)/(\alpha_v(\chi_1)\cap \alpha_v(\chi_2))\right).\]
 As in the case $p=2$, we have.

\begin{lemma} \label{norm cokernel for odd p}
Let $v$ be a place of $K$, $\chi \in \mathcal{C}(K_v)$ and $L_\chi$ the extension of $K_v$ cut out by $\chi$. Then 
\[h_v(\mathbbm{1}_{K_v},\chi)=\dim_{\mathbb{F}_p}A(K_v)/N_{L_\chi/K_v}A(L_\chi)\]
where $N_{L_\chi/K_v}:A(L_\chi)\rightarrow A(K_v)$ is the norm map. 

Moreover, if $v\nmid p$ is a nonarchimedean place of $K$ at which $A$ has good reduction then
\begin{itemize}
\item[(i)] if $\chi$ is unramified, we have 
\[h_v(\mathbbm{1}_{K_v},\chi)=\dim_{\mathbb{F}_p}A(K_v)/N_{L_\chi/K_v}A(L_\chi)=0,\]
\item[(ii)] if $\chi$ is ramified, we have 
\[h_v(\mathbbm{1}_{K_v},\chi)=\dim_{\mathbb{F}_p}A(K_v)/N_{L_\chi/K_v}A(L_\chi)=\dim_{\mathbb{F}_p}A(K_v)[p].\]
\end{itemize}
\end{lemma}

\begin{proof}
As in the case $p=2$, the first claim is shown for elliptic curves in \cite[Proposition 5.2]{MR2373150} and the argument is identical. The evaluation of the cokernel of the local norm map is \cite[Corollary 4.4]{MR0444670} for $\chi$ unramified, and for $\chi$ ramified the case where $A$ is an elliptic curve is \cite[Lemma 5.5 (ii)]{MR2373150} and  the same agument works in the general case.
\end{proof}

\begin{proposition} \label{p odd canonical twisting data}
The maps $\boldsymbol{\alpha}=(\alpha_v)_v$ define twisting data for $T=(A[p],\textbf{q},\Sigma)$ and the associated Selmer groups $\textup{Sel}(A[p],\chi)$ satisfy
\[\textup{Sel}(A[p],\chi)\cong\textup{Sel}^\pi(A^\chi/K).\] 
\end{proposition}

\begin{proof}
We first claim that for each place $v$ of $K$ and $\chi \in \mathcal{C}(K_v)$, we have $\alpha_v(\chi)\in \mathcal{H}(q_v)$ (i.e. $\alpha_v(\chi)$ is Lagrangian). That is (since $p$ is odd), that $\alpha_v(\chi)\subseteq H^1(K_v,A[p])$ is its own orthogonal complement under the Tate pairing. For $\chi$ trivial, that (the image in $H^1(K_v,A[p])$ of) $A(K_v)/pA(K_v)$ is its own orthogonal complement is a well known consequence of Tate local duality, see e.g. \cite[I.3.4]{MR2261462}. For $\chi$ non-trivial this is shown for $A$ an elliptic curve in \cite[Proposition A.7]{MR2373150} and the argument for a general principally polarised abelian variety is identical (with the Weil pairing associated to the principal polarisation $\lambda$ providing the pairing on the $p$-adic Tate-module $T_p(A)$ required for Definition A.5 of op. cit.). We remark that in the above, unlike the case $p=2$, the twist $A^\chi$ need not possess a principal polarisation (see \cite[Theorem 1.1]{MR1827021})  so one cannot deduce the result by just applying Tate duality to $A^\chi/K_v$ as one does not have an approporiate Weil-pairing on $A[\mathfrak{p}]$.

 To show that $\boldsymbol{\alpha}$ defines twisting data, it remains to show that for each place $v\notin\Sigma$ with $\boldsymbol \mu_p\subseteq K_v$, we have $\alpha_v(\chi)\in \mathcal{H}_{\textup{ram}}(q_v)$. That is, that $\alpha_v(\chi)\cap H^1_{\textup{ur}}(K_v,A[p])=0$. Again, the agrument is the same as in the case $p=2$. Indeed,  for such places we have $\alpha_v(\mathbbm{1}_{K_v})=H^1_{\textup{ur}}(K_v,A[p])$ (again, see e.g. \cite[Proposition 4.12]{MR2833483} and the preceeding remark) and we conclude by \Cref{norm cokernel for odd p} (ii).
 
The isomorphism $\textup{Sel}(A[p],\chi)\cong\textup{Sel}^\pi(A^\chi/K)$ is also proven identically to the case $p=2$ by comparing the local conditions defining the two Selmer groups.
 \end{proof}

\begin{cor} \label{p-cyclic abelian variety cor}
Let $p$ be an odd prime, $K$ a number field, $A/K$ a principally polarised abelian variety and  $\Sigma$ the set consisting of all archimedean places of $K$, all places of bad reduction for $A$, and all places dividing $p$. Define $\epsilon:\textup{Gal}(K(A[p])/K)\rightarrow \{\pm 1\}$ by $\sigma \mapsto (-1)^{\dim_{\mathbb{F}_p}A[p]^\sigma}$. 
\begin{itemize}
\item[(i)] If $\epsilon$ is non-trival when restricted to $\textup{Gal}(K(A[p])/K(\boldsymbol \mu_p))$ then \[\lim_{X\rightarrow \infty}
\frac{|\{\chi \in \mathcal{C}(K,X)~:~\textup{dim}_{\mathbb{F}_p}\textup{Sel}^\pi(A^\chi/K)~~\textup{is even}\}|}{|\mathcal{C}(K,X)|}=1/2.\]
\item[(ii)] Suppose $\epsilon$ is trivial when restricted to $\textup{Gal}(K(A[p])/K(\boldsymbol \mu_p))$. For each $v\in \Sigma$ and  character $\chi \in \mathcal{C}(K_v)$, write $L_\chi/K_v$ for the extension cut out by $\chi$ and define
\[\omega_v(\chi):=(-1)^{\dim_{\mathbb{F}_p}A(L_\chi)/N_{L_\chi/K_v}A(L_\chi)}.\] 
Finally, define
\[\delta_v:=\frac{1}{|\mathcal{C}(K_v)|}\sum_{\chi \in \mathcal{C}(K_v)}\omega(\chi)\phantom{hello}\textup{and}\phantom{hello}\delta:=\prod_{v\in \Sigma}\delta_v.\]
Then for all sufficiently large $X$, 
\[\frac{|\{\chi \in \mathcal{C}(K,X)~:~\textup{dim}_{\mathbb{F}_p}\textup{Sel}^\pi(A^\chi/K)~~\textup{is even}\}|}{|\mathcal{C}(K,X)|}=\frac{1+(-1)^{\textup{dim}_{\mathbb{F}_p}\textup{Sel}^p(A/K)}\cdot\delta}{2}.\]
\end{itemize}
\end{cor}

\begin{proof}
Combine \Cref{main twisting theorem combined} with \Cref{p odd canonical twisting data} and \Cref{norm cokernel for odd p}.
\end{proof}

\begin{remark}
As observed by Klagsbrun--Mazur--Rubin immediately before the statement of \cite[Theorem 8.2]{MR3043582}, as each $|\mathcal{C}(K_v)|$ has odd sixe we cannot have $\delta=0$ in Case (ii) above.
\end{remark}

\bibliographystyle{plain}

\bibliography{references}

\def\Dbar{\leavevmode\lower.6ex\hbox to 0pt{\hskip-.23ex \accent"16\hss}D}
  \def\cfac#1{\ifmmode\setbox7\hbox{$\accent"5E#1$}\else
  \setbox7\hbox{\accent"5E#1}\penalty 10000\relax\fi\raise 1\ht7
  \hbox{\lower1.15ex\hbox to 1\wd7{\hss\accent"13\hss}}\penalty 10000
  \hskip-1\wd7\penalty 10000\box7}
  \def\cftil#1{\ifmmode\setbox7\hbox{$\accent"5E#1$}\else
  \setbox7\hbox{\accent"5E#1}\penalty 10000\relax\fi\raise 1\ht7
  \hbox{\lower1.15ex\hbox to 1\wd7{\hss\accent"7E\hss}}\penalty 10000
  \hskip-1\wd7\penalty 10000\box7}
\begin{thebibliography}{10}

\bibitem{MR0219512}
M.~F. Atiyah and C.~T.~C. Wall.
\newblock Cohomology of groups.
\newblock In {\em Algebraic {N}umber {T}heory ({P}roc. {I}nstructional {C}onf.,
  {B}righton, 1965)}, pages 94--115. Thompson, Washington, D.C., 1967.

\bibitem{MR1484478}
Wieb Bosma, John Cannon, and Catherine Playoust.
\newblock The {M}agma algebra system. {I}. {T}he user language.
\newblock {\em J. Symbolic Comput.}, 24(3-4):235--265, 1997.
\newblock Computational algebra and number theory (London, 1993).

\bibitem{KES14}
Kestutis Cesnavicius.
\newblock The {l}-parity conjecture over the constant quadratic extension.
\newblock {\em Preprint}, http://arxiv.org/abs/1402.2939, 2014.

\bibitem{MR1865865}
Gunther Cornelissen.
\newblock Two-torsion in the {J}acobian of hyperelliptic curves over finite
  fields.
\newblock {\em Arch. Math. (Basel)}, 77(3):241--246, 2001.

\bibitem{MR2169307}
Gunther Cornelissen.
\newblock Erratum to: `{T}wo-torsion in the {J}acobian of hyperelliptic curves
  over finite fields'.
\newblock {\em Arch. Math. (Basel)}, 85(6):loose erratum, 2005.

\bibitem{MR0453639}
R.~H. Dye.
\newblock A geometric characterization of the special orthogonal groups and the
  {D}ickson invariant.
\newblock {\em J. London Math. Soc. (2)}, 15(3):472--476, 1977.

\bibitem{MR1079004}
Matthias Flach.
\newblock A generalisation of the {C}assels-{T}ate pairing.
\newblock {\em J. Reine Angew. Math.}, 412:113--127, 1990.

\bibitem{MR1827021}
Everett~W. Howe.
\newblock Isogeny classes of abelian varieties with no principal polarizations.
\newblock In {\em Moduli of abelian varieties ({T}exel {I}sland, 1999)}, volume
  195 of {\em Progr. Math.}, pages 203--216. Birkh\"auser, Basel, 2001.

\bibitem{MR3043582}
Zev Klagsbrun, Barry Mazur, and Karl Rubin.
\newblock Disparity in {S}elmer ranks of quadratic twists of elliptic curves.
\newblock {\em Ann. of Math. (2)}, 178(1):287--320, 2013.

\bibitem{Klagsbrun_Mazur_Rubin_2014}
Zev Klagsbrun, Barry Mazur, and Karl Rubin.
\newblock A markov model for selmer ranks in families of twists.
\newblock {\em Compositio Mathematica}, 150(7):1077–1106, Jul 2014.

\bibitem{MR597871}
Kenneth Kramer.
\newblock Arithmetic of elliptic curves upon quadratic extension.
\newblock {\em Trans. Amer. Math. Soc.}, 264(1):121--135, 1981.

\bibitem{MR2331769}
B.~Mazur, K.~Rubin, and A.~Silverberg.
\newblock Twisting commutative algebraic groups.
\newblock {\em J. Algebra}, 314(1):419--438, 2007.

\bibitem{MR0444670}
Barry Mazur.
\newblock Rational points of abelian varieties with values in towers of number
  fields.
\newblock {\em Invent. Math.}, 18:183--266, 1972.

\bibitem{MR2373150}
Barry Mazur and Karl Rubin.
\newblock Finding large {S}elmer rank via an arithmetic theory of local
  constants.
\newblock {\em Ann. of Math. (2)}, 166(2):579--612, 2007.

\bibitem{MR861974}
J.~S. Milne.
\newblock Abelian varieties.
\newblock In {\em Arithmetic geometry ({S}torrs, {C}onn., 1984)}, pages
  103--150. Springer, New York, 1986.

\bibitem{MR2261462}
J.~S. Milne.
\newblock {\em Arithmetic duality theorems}.
\newblock BookSurge, LLC, Charleston, SC, second edition, 2006.

\bibitem{morgan}
Adam Morgan.
\newblock 2-selmer parity for hyperelliptic curves in quadratic extensions.
\newblock {\em Preprint, arXiv:1504.01960}, 2015.

\bibitem{MUMFORD1966}
D.~Mumford.
\newblock On the equations defining abelian varieties. i.
\newblock {\em Inventiones mathematicae}, 1:287--354, 1966.

\bibitem{MR2392026}
J{\"u}rgen Neukirch, Alexander Schmidt, and Kay Wingberg.
\newblock {\em Cohomology of number fields}, volume 323 of {\em Grundlehren der
  Mathematischen Wissenschaften [Fundamental Principles of Mathematical
  Sciences]}.
\newblock Springer-Verlag, Berlin, second edition, 2008.

\bibitem{MR0280596}
Harriet Pollatsek.
\newblock First cohomology groups of some linear groups over fields of
  characteristic two.
\newblock {\em Illinois J. Math.}, 15:393--417, 1971.

\bibitem{MR2915483}
Bjorn Poonen and Eric Rains.
\newblock Self cup products and the theta characteristic torsor.
\newblock {\em Math. Res. Lett.}, 18(6):1305--1318, 2011.

\bibitem{MR2833483}
Bjorn Poonen and Eric Rains.
\newblock Random maximal isotropic subspaces and {S}elmer groups.
\newblock {\em J. Amer. Math. Soc.}, 25(1):245--269, 2012.

\bibitem{MR1740984}
Bjorn Poonen and Michael Stoll.
\newblock The {C}assels-{T}ate pairing on polarized abelian varieties.
\newblock {\em Ann. of Math. (2)}, 150(3):1109--1149, 1999.

\bibitem{MR770063}
Winfried Scharlau.
\newblock {\em Quadratic and {H}ermitian forms}, volume 270 of {\em Grundlehren
  der Mathematischen Wissenschaften [Fundamental Principles of Mathematical
  Sciences]}.
\newblock Springer-Verlag, Berlin, 1985.

\bibitem{Silhol}
Robert Silhol.
\newblock Digression on real abelian varieties and classification of real
  abelian surfaces.
\newblock In {\em Real Algebraic Surfaces}, volume 1392 of {\em Lecture Notes
  in Mathematics}, pages 75--94. Springer Berlin Heidelberg, 1989.

\bibitem{Smith16}
Alexander Smith.
\newblock Governing fields and statistics for 4-selmer groups and 8-class
  groups.
\newblock {\em Preprint, arXiv:1607.07860}, 2016.

\bibitem{YU15}
Myungjun Yu.
\newblock Selmer ranks of twists of hyperelliptic curves and superelliptic
  curves.
\newblock {\em Preprint}, http://arxiv.org/abs/1511.07511, 2015.

\end{thebibliography}

\end{document}